\documentclass[12pt]{amsart}
\usepackage[mathscr]{eucal}
\usepackage{amsthm}
\usepackage{amssymb}
\usepackage{amscd}
\usepackage{mathrsfs}

\usepackage{verbatim}
\usepackage{version}
\usepackage{pb-diagram}
\usepackage{color}

\makeatletter
\@ifclasslater{amsart}{1999/11/24}{}{
\renewcommand*\subjclass[2][1991]{%
  \def\@subjclass{#2}%
  \@ifundefined{subjclassname@#1}{%
    \ClassWarning{\@classname}{Unknown edition (#1) of Mathematics
      Subject Classification; using '1991'.}%
  }{%
    \@xp\let\@xp\subjclassname\csname subjclassname@#1\endcsname
  }%
}
\renewcommand{\subjclassname}{%
  \textup{1991} Mathematics Subject Classification}
\@xp\let\csname subjclassname@1991\endcsname \subjclassname
\@namedef{subjclassname@2000}{%
  \textup{2000} Mathematics Subject Classification}
}
\makeatother

\usepackage[all]{xy}
\def\Bbb{\mathbb}
\def\frak{\mathfrak}

\newenvironment{pf*}[1]{\proof[#1]}{\endproof}
\newenvironment{aenume}{%
  \begin{enumerate}%
  }{\end{enumerate}}

\renewcommand{\MR}[1]{}

\makeatother
\newenvironment{NB}{
\color{red}{\bf NB}. \footnotesize 
}{}
\excludeversion{NB}
\newenvironment{NB2}{
\color{blue}{\bf NB}. \footnotesize
}{}

\newtheorem{Theorem}[equation]{Theorem}
\newtheorem{Corollary}[equation]{Corollary}
\newtheorem{Lemma}[equation]{Lemma}
\newtheorem{Proposition}[equation]{Proposition}

\newtheorem{Conjecture}[equation]{Conjecture}
\newtheorem{main}{Theorem}

\theoremstyle{definition}
\newtheorem{Definition}[equation]{Definition}
\newtheorem{Example}[equation]{Example}

\theoremstyle{remark}
\newtheorem{Remark}[equation]{Remark}
\newtheorem{Remarks}[equation]{Remarks}
\newtheorem*{Claim}{Claim}

\numberwithin{equation}{section}

\newcommand{\thmref}[1]{Theorem~\ref{#1}}
\newcommand{\secref}[1]{\S\ref{#1}}
\newcommand{\lemref}[1]{Lemma~\ref{#1}}
\newcommand{\propref}[1]{Proposition~\ref{#1}}
\newcommand{\corref}[1]{Corollary~\ref{#1}}
\newcommand{\subsecref}[1]{\S\ref{#1}}

\newcommand{\defref}[1]{Definition~\ref{#1}}
\newcommand{\remref}[1]{Remark~\ref{#1}}
\newcommand{\lsp}[2]{{}^{#1}{#2}}
\newcommand{\defeq}{\overset{\operatorname{\scriptstyle def.}}{=}}
\newcommand{\C}{{\Bbb C}}
\newcommand{\Z}{{\Bbb Z}}
\newcommand{\Q}{{\Bbb Q}}

\newcommand{\proj}{{\Bbb P}}

\newcommand{\GL}{\operatorname{GL}}

\newcommand{\g}{{\frak g}}

\newcommand{\End}{\operatorname{End}}
\newcommand{\Hom}{\operatorname{Hom}}
\newcommand{\Ext}{\operatorname{Ext}}
\newcommand{\Ker}{\operatorname{Ker}}

\newcommand{\Ima}{\operatorname{Im}}

\newcommand{\rank}{\operatorname{rank}}

\newcommand{\tr}{\operatorname{tr}}

\newcommand{\id}{\operatorname{id}}
\newcommand{\ve}{\varepsilon}
\newcommand{\vq}{q}
\newcommand{\vin}{i} 
\newcommand{\vout}{o} 
\newcommand{\M}{{\frak M}} 

\newcommand{\dslash}{/\!\!/} 
\newcommand{\bC}{{\mathbf C}} 
\newcommand{\bA}{{\mathbf A}} 
\newcommand{\bB}{{\mathbf B}} 

\newcommand{\bq}{{\mathbf q}}
\newcommand{\codim}{\operatorname{codim}} 

\newcommand{\Ulq}{{\mathbf U}_{q}({\mathbf L}\mathfrak g)}

\newcommand{\shfO}{\mathcal O}

\newcommand{\Zm}{\mathfrak Z^\bullet}

\newcommand{\bfR}{\mathbf R}
\newcommand{\N}{\mathfrak M^\bullet}
\newcommand{\Nreg}{\mathfrak M^{\bullet\operatorname{reg}}_0}
\newcommand{\NLa}{\mathfrak L^\bullet}
\newcommand{\bM}{{\mathbf M}^\bullet} 
\newcommand{\HomE}{\operatorname{E}^\bullet}
\newcommand{\HomL}{\operatorname{L}^\bullet}

\newcommand{\res}{\operatorname{Res}}
\newcommand{\bE}{\mathbf E}
\newcommand{\pr}{{\operatorname{pr}}}
\newcommand{\fr}{{\operatorname{fr}}}
\newcommand{\bx}{\mathbf x}
\newcommand{\by}{\mathbf y}
\newcommand{\Gr}{\operatorname{Gr}}
\newcommand{\ext}{\operatorname{ext}}
\newcommand{\bL}{\mathbf L}
\newcommand{\cQ}{\mathcal Q}
\newcommand{\cG}{\mathcal G}
\newcommand{\add}{{\operatorname{add}}}
\newcommand{\dec}{{\operatorname{dec}}}
\newcommand{\rep}{{\operatorname{rep}}}

\newcommand{\kW}{\lsp{\bar\sigma}W^k}
\newcommand{\skW}{\lsp{\bar\sigma*}W^k}
\newcommand{\lW}{\lsp{\bar\sigma}W^l}
\newcommand{\slW}{\lsp{\bar\sigma*}W^l}

\usepackage[dvips,bookmarks=false]{hyperref}

\begin{document}
\title[Quiver varieties and cluster algebras]
{Quiver varieties and
 cluster algebras
}
\author{Hiraku Nakajima}
\dedicatory{To the memory of the late Professor Masayoshi Nagata}
\address{Research Institute for Mathematical Sciences,
Kyoto University, Kyoto 606-8502, Japan
}
\email{nakajima@kurims.kyoto-u.ac.jp}
\urladdr{http://www.kurims.kyoto-u.ac.jp/\textasciitilde nakajima}
\thanks{Supported by the Grant-in-aid
for Scientific Research (No.19340006),
JSPS}
%
\begin{abstract}
  Motivated by a recent conjecture by Hernandez and Leclerc
  \cite{HerLec}, we embed a Fomin-Zelevinsky cluster algebra
  \cite{FomZel} into the Grothendieck ring $\bfR$ of the category of
  representations of quantum loop algebras $\Ulq$ of a symmetric
  Kac-Moody Lie algebra, studied earlier by the author via perverse
  sheaves on graded quiver varieties \cite{Na-qaff}.
  Graded quiver varieties controlling the image can be identified with
  varieties which Lusztig used to define the canonical base.
  The cluster monomials form a subset of the base given by the classes
  of simple modules in $\bfR$, or Lusztig's dual canonical base.
  The conjectures that cluster monomials are positive and linearly
  independent (and probably many other conjectures) of \cite{FomZel}
  follow as consequences, when there is a seed with a bipartite
  quiver.
  Simple modules corresponding to cluster monomials factorize into
  tensor products of `prime' simple ones according to the cluster
  expansion.
\end{abstract}
\maketitle
\tableofcontents
\section{Introduction}

\subsection{Cluster algebras}

Cluster algebras were introduced by Fomin and Zelevinsky \cite{FomZel}. A
cluster algebra $\mathscr A$ is a subalgebra of the rational function
field $\Q(x_1,\dots, x_n)$ of $n$ indeterminates equipped with
a distinguished set of variables (cluster variables) grouped into
overlapping subsets (clusters) consisting of $n$ elements, defined by
a recursive procedure (mutation) on quivers.
Let us quote the motivation from the original text [loc.\ cit.,
p.498, the second paragraph]:
\begin{quote}
  This structure should serve as an algebraic framework for the study
  of ``dual canonical bases'' in these coordinate rings and their
  $q$-deformations. In particular, we conjecture that all monomials in
  the variables of any given cluster (the cluster {\it monomials\/})
  belong to this dual canonical basis.
\end{quote}
Here ``dual canonical base'' means a conjectural analog of the dual of
Lusztig's canonical base of $\mathbf U_\vq^-$, the $-$ part of the
quantized enveloping algebra (\cite{Lu-book}). One of the deepest
properties of the dual canonical base is positivity: the structure
constants are in $\Z_{\ge 0}[\vq,\vq^{-1}]$.
But the existence and positivity are not known for cluster algebras
except some examples.

The theory of cluster algebras has been developed in various
directions different from the original motivation. See the list of
references in a recent survey \cite{Keller}.
\begin{NB}
The reason is probably that many people have found combinatorial
aspects of cluster algebras in many other places.
\end{NB}%

One of the most active directions is the theory of the cluster
category \cite{BMRRT}. It is defined as the orbit category of the
derived category $\mathscr D(\rep\cQ)$ of finite dimensional
representations of a quiver $\cQ$ under the action of an
automorphism. This theory is quite useful to understand combinatorics
of the cluster algebra: clusters are identified with tilting objects,
and mutations are interpreted as exchange triangles. See the survey
\cite{Keller} for more detail.

However the cluster category {\it does not\/} have enough structures,
compared with the cluster algebra. For example, multiplication
of the cluster algebra roughly corresponds to the direct sum of the
cluster category, but addition remains obscure.
So the cluster category is called {\it additive categorification\/} of
the cluster algebra.
The cluster algebra is recovered from the cluster category by the
so-called cluster character. (Somebody calls Caldero-Chapoton map.)
But it is not clear how to obtain all the ``dual canonical base''
elements from this method.
%

Very recently Hernandez and Leclerc \cite{HerLec} propose another
categorical approach. They conjecture that there exists a monoidal
abelian category $\mathscr M$ whose Grothendieck ring is the cluster
algebra. All of structures of the cluster algebra can be conjecturally
lifted to the monoidal category. For example, the dual canonical base
is given by simple objects, the combinatorics of mutation is explained
by decomposing tensor products into simple objects, etc. Here we give
the table of structures:

\begin{center}
\begin{tabular}{c|c|c}
  cluster algebra & additive categorification & monoidal categorification
  \\
  \hline
  $+$ & ? & $\oplus$
  \\
  $\times$ & $\oplus$ & $\otimes$
  \\
  clusters & cluster tilting objects & real simple objects
  \\
  mutation & exchange triangle & 
  $0\to S \to X_i\otimes X_i^* \to S'\to 0$
  \\
  cluster variables & rigid indecomposables & real prime simple objects
  \\
  dual canonical base & ? & simple objects
  \\
  ? & ? & prime simple objects
\end{tabular}
\end{center}

In the bottom line, we have a definition of {\it prime simple\/}
objects, those which cannot be factored into smaller simple
objects. There is no counter part in the theory of the cluster
algebra, so completely new notion.

However the monoidal categorification seems to have a drawback. We do
not have many tools to study the tensor product factorization in an
abstract setting. We need an additional input from other sources.
Therefore it is natural to demand functors connecting two
categorifications exchanging $\oplus$ and $\otimes$, and hopefully
`$?$' and $\oplus$. We call them {\it tropicalization\/} and
de-tropicalization functors\footnote{Leclerc himself already had a
  hope to make a connection between two
  categorifications (\cite{Lec2}). He calls `exponential' and `log'.}
expecting the top ? in the additive categorification column is
something like `$\min$':
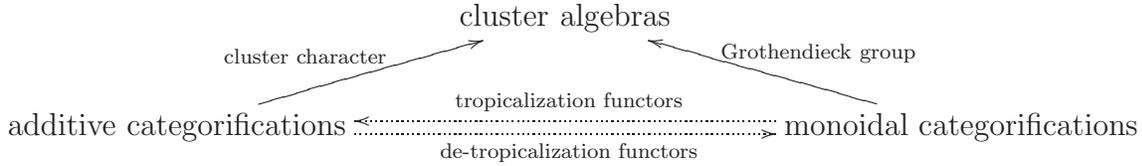
\begin{equation*}
 \xymatrix@C=3pc{
  & \text{cluster algebras} &
\\   
  \text{additive categorifications}
  \ar@{.>} @<-.2pc>[rr]_{\text{de-tropicalization functors}}
  \ar[ur]^{\txt<14pc>{\tiny cluster character}}
  && \ar@{.>} @<-.2pc>[ll]_{\text{tropicalization functors}}
  \ar[ul]_{\txt<12pc>{\tiny Grothendieck group}}
  \text{monoidal categorifications}
}
\end{equation*}

The author believes this is an interesting idea to pursue, but it is
so far just a {\it slogan}: it seems difficult to make even
definitions of (de)tropicalization functors precise.
Therefore we set aside categorical approaches, return back to the
origin of the cluster algebra, i.e.\ the construction of the canonical
base, and ask why it has many structures ?

The answer is simple: Lusztig's construction of the canonical base is
based on the category of perverse sheaves on the space of
representations $\bE_W$ of the quiver. Therefore
\begin{aenume}
\item it has the structure of the monoidal abelian category, where the
  tensor product is given by the convolution diagram coming from exact
  sequences of quiver representations;

\item it inherit various combinatorial structure from the module
  category $\rep\cQ$, and probably also from the cluster
  category.
\end{aenume}
In this sense, we already have (de)tropicalization functors~!

Thus we are led to ask a naive question, sounding much more elementary
compared with categorical approaches:
\begin{quote}
  Is it possible to realize a cluster algebra entirely in Lusztig's
  framework, i.e.\ via a certain category of perverse sheaves on the
  space $\bE_W$ of representations of a quiver~?
\end{quote}

If the answer is affirmative, the positivity conjecture is a direct
consequence of that of the canonical base.

As far as the author searches the literature in the subject, there is
no explicit mention of this conjecture, though many examples of
cluster algebras arise really as subalgebras of $\mathbf U_{\vq}^-$.
Usually Lusztig's perverse sheaves appear only as a motivation, and is
not used in a fundamental way. A closest result is Geiss, Leclerc and
Schr\"oer's work \cite{GLS,GLS2} where the cluster algebra is
realized as a space of constructible functions on $\Lambda_W$, the
space of nilpotent representations of the preprojective algebra. This
$\Lambda_W$ is a lagrangian in the cotangent space $T^*\bE_W$ of the
space $\bE_W$ of representations.
The space of constructible functions was used also by Lusztig to
construct the {\it semicanonical\/} base \cite{Lu:semi}.
Constructible functions are vaguely related to perverse sheaves (or
$D$-modules) via characteristic cycle construction, though nobody
makes the relation precise.
And it was proved that cluster monomials are indeed elements of the
dual semicanonical base \cite{GLS,GLS2}.
But constructible functions have less structures than perverse
sheaves, in particular, the positivity of the multiplication is unknown.

\begin{NB}
Let us state our result in an abstract way.

\begin{main}
  Let $\mathscr A(\widetilde \bB)$ be the cluster algebra associated
  with the $\bx$-quiver given in \defref{def:x-quiver}.
  Then there is a monoidal abelian category $\mathscr C_1$ such that
  the followings hold.
  \begin{enumerate}
  \item Its Grothendieck ring $\mathcal K(\mathscr C_1)$ has the
    structure of a commutative algebra isomorphic to the polynomial
    ring $\Z[x_i, x_i']_{i\in I}$ of variables indexed by $I$ and its
    copy $I_\fr$.

  \item The cluster algebra $\mathscr A(\widetilde \bB)$ is isomorphic
    to $\mathcal K(\mathscr C_1)$ so that $x_i$ and $f_i$ given by
    \eqref{eq:T-system} are the cluster variables from the initial
    seed.

  \item The simple objects $L(W)$ of $\mathscr C_1$ give a base of
    $\mathcal K(\mathscr C_1)$, naturally indexed by isomorphism
    classes of $I\sqcup I_\fr$-graded vector spaces $W$.

  \item Cluster monomials form a subset of the base given by simple objects.

  \item If $L(W)$ corresponds to a cluster monomial $m$, it factors as
    $L(W)\cong L(W^1)\otimes \cdots \otimes L(W^s)$ according to the
    factorization $m = m^1\cdots m^s$ to the product of cluster
    variables.
  \end{enumerate}
\end{main}
\end{NB}

Before explaining our framework for the cluster algebra via perverse
sheaves, we need to explain the author's earlier work
\cite{Na-qaff}. It is another child of Lusztig's work.

\subsection{Graded quiver varieties and quantum loop algebras}
\label{subsec:gradedquiver}

In \cite{Na-qaff} the author studied the category $\mathscr R$ of {\it
  l\/}-integrable representations of the quantum loop algebra $\Ulq$
of a symmetric Kac-Moody Lie algebra $\g$
via perverse sheaves on graded quiver varieties $\N_0(W)$ (denoted by
$\M_0(\infty,{\mathbf w})^A$ in [loc.\ cit.]).
If $\g$ is a simple Lie algebra of type $ADE$,
$\Ulq$ is a subquotient of Drinfeld-Jimbo quantized enveloping algebra
of affine type $ADE$ (usually called the quantum affine algebra), and
$\mathscr R$ is nothing but the category of finite dimensional
representations of $\Ulq$.
The graded quiver varieties are fixed point sets of the quiver
varieties $\M_0(W)$ introduced in \cite{Na-quiver,Na-alg} with respect
to torus actions.
The main result says that the Grothendieck group $\bfR$ of $\mathscr
R$ has a natural $t$-deformation $\bfR_t$ which can be constructed
from a category $\mathscr P_W$ of perverse sheaves on $\N_0(W)$ so
that simple (resp.\ standard) modules correspond to dual of
intersection cohomology complexes (resp.\ constant sheaves) of natural
strata of $\N_0(W)$. Here the parameter $t$ comes from the
cohomological grading.
Furthermore the transition matrix of two bases of simple and standard
modules (= dimensions of stalks of IC complexes) is given by analog of
Kazhdan-Lusztig polynomials, which can be computed\footnote{The
  meaning of the word {\it compute} will be explained in
  \remref{rem:E8}.} by using purely combinatorial objects
$\chi_{\vq,t}$, called $t$-analog of $q$-characters
\cite{MR1872257,MR2144973}.
If we set $t=1$, we get the $q$-character defined by \cite{Knight,Fre-Res}
as the generating function of the dimensions of {\it l\/}-weight
spaces, simultaneous generalized eigenspaces with respect to a
commutative subalgebra of $\Ulq$.
For the simple module corresponding to an $IC$ complex $L$,
$\chi_{\vq,t}$ is the generating function of multiplicities of $L$ in
direct images of constant sheaves on various nonsingular graded quiver
varieties $\N(V,W)$ under morphisms $\pi\colon \N(V,W)\to \N_0(W)$.

We have a noncommutative multiplication on $\bfR_t$, which is a
$t$-deformation of a commutative multiplication on $\bfR$. When $\g$
is of type $ADE$, the commutative multiplication on $\bfR$ comes from
the tensor product $\otimes$ on the category $\mathscr R$ as $\Ulq$ is
a Hopf algebra.
(It is not known whether the quantum loop algebra $\Ulq$ can be
equipped with the structure of a Hopf algebra in general.)
The $t$-deformed multiplication was originally given in terms of
$t$-analog of $q$-characters, but Varagnolo-Vasserot \cite{VV2} later
introduced a convolution diagram on $\N_0(W)$ which gives the
multiplication in more direct and geometric way.

These geometric structures are similar to ones used to define the
canonical base of $\mathbf U^-_q$ by Lusztig \cite{Lu-book}. We have
the following table of analogy:
\begin{center}
\begin{tabular}{c|c|c}
$\bfR_t$ & geometry & dual of $\mathbf U_q^-$
\\
\hline
standard modules $M(W)$ & constant sheaves & dual PBW base elements
\\
simple modules $L(W)$ & IC complexes & dual canonical base elements
\\
$t$-deformed $\otimes$ & convolution diagram & multiplication
\end{tabular}
\end{center}
Note that $\mathbf U^-_q$ is not commutative even at $q=1$, while
its dual $\left.(\mathbf U_q^-)^*\right|_{q=1}$ is the coordinate ring
$\C[\mathfrak n^-]$, hence commutative. Hence we should compare
$\bfR_t$ with $(\mathbf U_q^-)^*$, not with $\mathbf U_q^-$.
Also the convolution diagram looks similar to one for the
comultiplication, not to one for the multiplication.
The only difference is relevant varieties: Lusztig used the vector
spaces $\bE_W$ of representations of the quiver with group actions (or
the moduli {\it stacks\/} of representations of the quiver), while the
author used graded quiver varieties, which are framed moduli {\it
  spaces\/} of graded representations of the preprojective algebra
associated with the underlying graph.

The computation of the transition matrix is hard to use in practice,
like the Kazhdan-Lusztig polynomials.
On the other hand many peoples have been studying special modules (say
tame modules, Kirillov-Reshetikhin modules, minimal affinization,
etc.)~by purely algebraic approaches, at least when $\g$ is of finite
type. See \cite{CharHer} and the references therein. Their structure
is different from that of general modules.
Thus it is natural to look for a special geometric property which
holds {\it only\/} for graded quiver varieties corresponding to these
classes of modules.
In \cite[\S10]{Na-qaff} the author introduced two candidates of such
properties. We name corresponding modules {\it special\/} and {\it
  small\/} respectively.
These properties are easy to state both in geometric and algebraic
terms, but it is difficult to check whether a given module is special
or small.
Since [loc.~cit.], we have been gradually understanding that
smallness is not a right concept as there are only very few examples
(see \cite{Her-small}),
but the speciality is a useful concept and there are many special
modules, say Kirillov-Reshetikhin modules\footnote{Special modules
  form a {\it special\/} class of modules. Small modules form a {\it
    small\/} class of modules. But the name `small' originally comes
  from the smallness of a morphism.}.
One of applications of this study was a proof of the $T$-system, which
was conjectured by Kuniba-Nakanishi-Suzuki in 1994 (see
\cite{MR1993360} and the references therein). Several steps in the
proof of the main result in [loc.\ cit.]\ depended on the geometry,
but they were replaced by purely algebraic arguments, and generalized
to nonsymmetric quantum loop algebras cases later by
Hernandez~\cite{Her}. It was a fruitful interplay between geometric
and algebraic approaches.

\subsection{Realization of cluster algebras via perverse sheaves}

Hernandez and Leclerc \cite{HerLec} not only give an abstract
definition of a monoidal categorification, but also its candidate for
a certain cluster algebra. It is a monoidal (i.e.\ closed under the
tensor product) subcategory $\mathscr C_1$ of $\mathscr R$ when $\g$
is of type $ADE$.
They indeed show that $\mathscr C_1$ is a monoidal categorification
for type $A$ and $D_4$.
Therefore we have a strong evidence that it is a right candidate.
From what we have reviewed just above, if it indeed is a monoidal
categorification, the cluster algebra is a subalgebra of $\bfR$,
constructed via perverse sheaves on graded quiver varieties~!
Moreover, from the philosophy explained above, we could expect that
graded quiver varieties corresponding to $\mathscr C_1$ have very
{\it special\/} features compared with general ones.

In this paper, we show that it turns out to be true.
The first main observation (see \propref{prop:ast_1}) is that the
graded quiver varieties $\N_0(W)$ become just the vector spaces
$\bE_W$ of representations of the {\it decorated quiver}. Here the
decorated quiver\footnote{The decorated quiver is different from one
  in \cite{MRZ}, where there are no arrows between $i$ and $i'$.} is
constructed from a given finite graph with a bipartite orientation by
adding a new (frozen) vertex $i'$ and an arrow $i'\to i$ (resp.\ $i\to
i'$) if $i$ is a sink (resp.\ source) for each vertex $i$. (See
\defref{def:decorated}.)
Therefore the underlying variety is nothing but what Lusztig used.
Also the convolution diagram turns out to be the same as Lusztig's
one. Thus the Grothendieck group $\mathcal K(\mathscr C_1)$ is also a
subquotient of the dual of $\mathbf U_\vq^-$, associated with the
Kac-Moody Lie algebra corresponding to the decorated quiver.

To define a cluster algebra with frozen variables (or with
coefficients in the terminology in \cite{FomZel}), we choose a quiver
with choices of frozen vertexes.
We warn the reader that this quiver for the cluster algebra (we call
$\bx$-quiver, see \defref{def:x-quiver}) is slightly different from
the decorated quiver: the principal part has the opposite orientation
while the frozen part is the same.

\subsection{Second key observation}

Once we get a correct candidate of the class of perverse sheaves, we
next study structures of the dual canonical base and try to pull out
the cluster algebra structure from it. We hope to see a shadow of the
structure of a cluster category.

As we mentioned above, our $\mathcal K(\mathscr C_1)$ is a subquotient
of the dual of $\mathbf U_\vq^-$.
In particular, we introduce an equivalence relation on the canonical
base. The second key observation is that each equivalence class
contains an exactly one skyscraper sheaf $1_{\{0\}}$ of the origin $0$
of $\bE_W$ (the simplest perverse sheaf~!). This equivalence relation
is built in the theory of graded quiver varieties.
From this observation together with the first observation that the
graded quiver varieties are vector spaces, we can apply the
Fourier-Sato-Deligne transform \cite{KaSha,Lau} to make a reduction to
a study of constant sheaves $1_{\bE_W^*}$ on the whole space.

There is a certain natural family of projective morphisms
$\pi^\perp\colon \widetilde{\mathcal F}(\nu,W)^\perp \to \bE_W^*$ from
nonsingular varieties $\widetilde{\mathcal F}(\nu,W)^\perp$.
This family appears as monomials in Lusztig's context, and
$\vq$-characters in the theory reviewed in
\subsecref{subsec:gradedquiver}.
Using these morphisms, we define a homomorphism from $\mathbf R$ to
the cluster algebra.
Fibers of these morphisms are what are called {\it quiver
  Grassmannian\/} varieties. People study their Euler characters and
define the cluster character as their generating function.
This is clearly related to the study of the pushforward
\begin{equation*}
  \pi^\perp_!(1_{\widetilde{\mathcal F}(\nu,W)^\perp}[\dim
  \widetilde{\mathcal F}(\nu,W)^\perp]).
\end{equation*}
If $\bE_W^*$ contains an open orbit, then the Euler number of the
fiber over a point in the orbit is nothing but the coefficient
of $1_{\bE_W^*}[\dim \bE_W]$ in the above push-forward.
When the dual canonical base element is a cluster monomial, $\bE_W^*$
indeed contains an open orbit.
Therefore we immediately see that all cluster monomials are dual
canonical base elements.
This very simple observation between the cluster character and the
push-forward was appeared in the work of Caldero-Reineke
\cite{caldero-reineke}\footnote{There is a gap in the proof of
  \cite[Theorem~1]{caldero-reineke} since Lusztig's $v$ is identified
  with $q$. The correct identification is $v = -\sqrt{q}$. We give a
  corrected proof in \secref{sec:app}.}.

To be more precise, we need to apply reflection functors at all sink
vertexes in the decorated quiver with opposite orientations to
identify fibers of $\widetilde F(\nu,W)^\perp$ with quiver
Grassmannian varieties. The resulting quiver corresponds to the cluster
algebra with principal coefficients \cite{FomZel4}.

An appearance of the cluster character formula in the category
$\mathscr C_1$ was already pointed out in \cite[\S12]{HerLec}, as it is
nothing but a leading part of the $\vq$-character mentioned above.
(We call the leading part the {\it truncated\/} $\vq$-character.)

From a result on graded quiver varieties, it also follows that quiver
Grassmannian varieties have vanishing odd cohomology groups under the
above assumption. The generating function of all Betti numbers is
nothing but the truncated $t$-analog of $\vq$-character of a simple
modules. In particular, it was computed in \cite{MR2144973}.

We have assumed that $\bE_W^*$ contains an open orbit.
But the only necessary assumption we need is that perverse sheaves
corresponding to canonical base elements have strictly smaller
supports than $\bE_W^*$ except $1_{\bE_W^*}[\dim \bE_W^*]$.
Even if this condition is not satisfied, we can consider the {\it
  almost simple module\/} $\mathbb L(W)$ corresponding to the sum of
perverse sheaves whose supports are the whole $\bE_W^*$.
Then the total sum of Betti numbers (the Euler number is not natural
in this wider context) of the quiver Grassmannian give the truncated
$\vq$-character of the almost simple module.
An almost simple module $\mathbb L(W)$ is not necessarily simple in
general.

It is rather simple to study tensor product factorization of $\mathbb
L(W)$ since we computed their truncated $q$-characters.
First we observe that Kirillov-Reshetikhin modules simply factor
out. Then we may assume that $W$ have $0$ entries on frozen vertexes.
Thus $W$ is supported on the first given vertexes.
We next observe that $\mathbb L(W)$ factors as
\begin{equation*}
  \mathbb L(W) \cong \mathbb L(W^1)\otimes\cdots \otimes \mathbb L(W^s),
\end{equation*}
according to the {\it canonical decomposition\/} $W = W^1\oplus
\cdots\oplus W^s$ of $W$. Recall the canonical decomposition is the
decomposition of a general representation of $\bE_W$ first introduced
by Kac~\cite{Kac-quiver,Kac-quiver2}, and studied further by Schofield
\cite{Schofield}.
It is known that each $W^k$ is a Schur root (i.e.\ a general
representation is indecomposable) and $\Ext^1$ between general
representations from two different factors $W^k$, $W^l$ vanish.

We prove that a simple module $L(W)$ corresponds to a cluster monomial
if and only if the canonical decomposition contains only real Schur
roots. In this case, $\bE_W^*$ contains an open orbit. Then we have
$\mathbb L(W) = L(W)$, $\mathbb L(W^k) = L(W^k)$ and each $L(W^k)$
corresponds to a cluster variable, and the above tensor factorization
corresponds to the cluster expansion.

\subsection{To do list}
In this paper, basically due to laziness of the author, at least four
natural topics are not discussed:
\begin{itemize}
\item Our Grothendieck ring $\bfR$ has a natural noncommutative
  deformation $\bfR_t$. It should contain the quantum cluster algebra
  in \cite{BerZel}. In fact, we already give our main formula (in
  \thmref{thm:main}) in Poincar\'e polynomials of quiver Grassmannian
  varieties. Therefore the only remaining thing is to prove the
  quantum version of the cluster character formula. Any proof in the
  literature should be modified to the quantum version naturally, as
  it is based on counting of rational points.

  (After an earlier version of this article was posted on the arXiv,
  Qin proved the quantum version of the cluster character formula for
  an acyclic cluster algebra \cite{FQ}. This is the most essential
  part for this problem, but it still need to check that the
  multiplication $\bfR_t$ is the same as that of the quantum cluster
  algebra. This will be checked elsewhere.)

\item We only treat the case when the underlying quiver is
  bipartite. Since the choice of the quiver orientation is not
  essential in Lusztig's construction (in fact, the Fourier transform
  provides a technique to change orientations), this assumption
  probably can be removed.

\item We only treat the symmetric cases. Symmetrizable cases can be
  studied by considering quiver automorphisms as in Lusztig's
  work. Though the corresponding theory was not studied in author's
  theory, it should corresponds to the representations of twisted
  quantum affine algebras.

\item In \cite{GLS,GLS2} it was proved that cluster monomials are
  semicanonical base elements. It was conjectured that they are also
  canonical base elements. It is desirable to study the precise
  relation of this work to ours.
\end{itemize}
The author or his friends will hopefully come back to these
problems in near future.

In \cite{HerLec} a further conjecture is proposed for the monoidal
subcategory $\mathscr C_{\ell}$, where $\mathscr C_1$ is the special
case $\ell = 1$. Since the graded quiver varieties are no longer
vector spaces for $\ell > 1$, the method of this paper does not
work. But it is certainly interesting direction to pursue.
We also remark that other connections between the cluster algebra
theory and the representation theory of quantum affine algebras have
been found by Di~Francesco-Kedem \cite{DF-K} and
Inoue-Iyama-Kuniba-Nakanishi-Suzuki \cite{IIKNS}. It is also
interesting to make a connection to their works.

This article is organized as
follows. \S\S\ref{sec:cluster_pre},~\ref{sec:quiver} are preliminaries
for cluster algebras and graded quiver varieties respectively.  In
\secref{sec:C_1} we introduce the category $\mathscr C_{1}$ following
\cite{HerLec} and study the corresponding graded quiver varieties.
In \secref{sec:hom} we define a homomorphism from the Grothendieck
group $\mathbf R_{\ell = 1}$ of $\mathscr C_{1}$ to a rational
function field which is endowed with a cluster algebra structure.
In \secref{sec:prime} we explain the relation between the cluster
character and the push-forward and derive several consequences on 
factorizations of simple modules.
In \secref{sec:cluster} we prove that cluster monomials are dual
canonical base elements.
In \secref{sec:app} we prove that the quiver Grassmannian of a rigid
module of an acyclic quiver has no odd cohomology. It implies the
positive conjecture for an acyclic cluster algebra for the special
case of an {\it initial\/} seed.

\subsection*{Acknowledgments}
I began to study cluster algebras after
Bernard Leclerc's talk
at the meeting `Enveloping Algebras and
Geometric Representation Theory' at the Mathematisches
Forschungsinstitut Oberwolfach (MFO) in March 2009.
I thank him and David Hernandez for discussions during/after the
meeting. They kindly taught me many things on cluster algebras.
Alexander Braverman's question/Leclerc's answer (Conway-Coxeter frieze
\cite{ConwayCoxeter}) and discussions with Rinat Kedem at the meeting
were also very helpful.
I thank the organizers, as well as the staffs at
the MFO for providing me nice environment for the research.
I thank Osamu Iyama for correcting my understanding on the cluster
mutation.
I thank Andrei Zelevinsky for comments on a preliminary version of
this paper.
I also thank Fan Qin for correspondences.
Finally I thank George Lusztig's works which have been a source of my
inspiration more than a decade.

\section{Preliminaries (I) -- Cluster algebras}\label{sec:cluster_pre}

We review the definition and properties of cluster algebras.

\subsection{Definition}
Let $\cG = (I,E)$ be a finite graph, where $I$ is the set of vertexes and
$E$ is the set of edges.
Let $H$ be the set of pairs consisting of an edge together with its
orientation. For $h\in H$, we denote by $\vin(h)$ (resp.\ $\vout(h)$)
the incoming (resp.\ outgoing) vertex of $h$.
For $h\in H$ we denote by $\overline h$ the same edge as $h$ with the
reverse orientation.
A {\it quiver\/} $\cQ = (I,\Omega)$ is the finite graph $\cG$ together
with a choice of an {\it orientation\/} $\Omega\subset H$ such that
$\Omega\cap \overline{\Omega} = \emptyset$,
$\Omega\cup\overline{\Omega} = H$.

We will consider a pair of a quiver $\cQ = (I,\Omega)$ and a larger
quiver $\widetilde\cQ = (\widetilde I,\widetilde\Omega)$ containing
$\cQ$, where $I$ is a subset of $\widetilde I$ and $\Omega$ is
obtained from $\widetilde\Omega$ by removing arrows incident to a
point in $\widetilde I\setminus I$.
Set $I_\fr = \widetilde I\setminus I$. We call $i\in I_\fr$ (resp.\
$i\in I$) a {\it frozen\/} (resp.\ {\it principal\/}) vertex.

We assume that $\widetilde\cQ$ has no loops nor $2$-cycles and
there are no edges connecting points in $I_\fr$.
We define a matrix $\widetilde{\bB} = (b_{ij})_{i\in \widetilde I,j\in
  I}$ by
\begin{equation*}
  b_{ij} :=
  \begin{aligned}[t]
  & (\text{the number of oriented edges from $j$ to $i$})
  \\
  & \text{or $-$(the number of oriented edges from $i$ to $j$)}.
  \end{aligned}
\end{equation*}
Since we have assumed $\widetilde\cQ$ contains no $2$-cycles, this is
well-defined. Moreover, giving $\widetilde\bB$ is equivalent to a
quiver $\widetilde\cQ$ with the decomposition $\widetilde I = I\sqcup
I_\fr$ as above.
The {\it principal part\/} $\bB$ of $\widetilde \bB$ is the matrix
obtained from $\widetilde \bB$ by taking entries for $I\times I$. From
the definition $\bB$ is skew-symmetric.

For a vertex $k \in I$ we define the {\it matrix mutation\/}
$\mu_k(\widetilde{\bB})$ of $\widetilde{\bB}$ in direction $k$ as the
new matrix $(b_{ij}')$ indexed by $(i,j)\in \widetilde I \times I$
given by the formula
\begin{equation}
b'_{ij}=
\begin{cases}
   -b_{ij} & \text{if $i=k$ or $j=k$},\\
   b_{ij} + \operatorname{sgn}(b_{ik})\max(b_{ik}b_{kj},0) &
   \text{otherwise}.
\end{cases}
\end{equation}
If $\widetilde\Omega^*$ denotes the corresponding quiver, it is obtained
from $\widetilde\Omega$ by the following rule:
\begin{enumerate}
\item For each $i\to k$, $k\to j\in \widetilde\Omega$, create a new
  arrow $i\to j$ if either $i$ or $j\in I$.
\item Reverse all arrows incident to $k$.
\item Remove $2$-cycles between $i$ and $j$ of the resulting quiver
  after (1) and (2).
\end{enumerate}
Graphically it is given by
\begin{equation*}
\widetilde\Omega : 
 \xymatrix@R=.5pc{
 i \ar[rd]_{s} \ar[rr]^r & & j 
\\
  & \ar[ru]_t k &
}
\qquad\Longrightarrow\qquad
\widetilde\Omega^* : 
 \xymatrix@R=.5pc{
 i \ar[rr]^{r+st} & & \ar[ld]^t j
\\
  & \ar[lu]^s k &
},
\end{equation*}
where $s$, $t$ are nonnegative integers and $i\xrightarrow{l}j$ means
that there are $l$ arrows from $i$ to $j$ if $l\ge 0$, $(-l)$ arrows
from $j$ to $i$ if $l\le 0$. The new quiver $\widetilde\Omega^*$ has no
loops nor 2-cycles.

Let ${\mathscr F} = \Q(x_i)_{i\in \widetilde I}$ be the field of
rational functions in commuting indeterminates $\bx = (x_i)_{i\in
  \widetilde I}$ indexed by $\widetilde I$. For $k\in I$ we define a
new variable $x_k^*$ by the {\it exchange relation}:
\begin{equation}\label{eq:exchange}
  x_k^* = 
  \frac{\prod_{b_{ik}> 0} x_i^{b_{ik}} + \prod_{b_{ik}< 0} x_i^{-b_{ik}}}{x_k}.
\end{equation}
Let $\mu_k(\bx)$ be the set of variables obtained from $\bx$ by
replacing $x_k$ by $x_k^*$. The pair
$(\mu_k(\bx),\mu_k(\widetilde{\bB}))$ is called the {\it mutation\/} of
$(\bx,\widetilde{\bB})$ in direction $k$.
We can iterate this procedure and obtain new pairs by mutating
$(\mu_k({\mathbf x}),\mu_k(\widetilde{\bB}))$ in any direction $l\in
I$.
We do not make mutations in direction of a frozen vertex $k\in
I_\fr$. Variables $x_i$ for $i\in I_\fr$ are always in $\mu_k({\mathbf
  x})$; they are called {\it frozen variables\/} (or {\it
  coefficients\/} in \cite{FomZel}).

Now a {\it seed\/} is a pair $(\by,\widetilde{\mathbf C})$ of $\by =
(y_i)_{i\in \widetilde I}\in \mathscr F^{\widetilde I}$ and a matrix
$\widetilde{\mathbf C} = (c_{ij})_{i\in \widetilde I,j\in I}$ obtained
from the {\it initial seed\/} $(\bx,\widetilde{\mathbf B})$ by a
successive application of mutations in various direction $k\in I$.
The set of seeds is denoted by ${\mathscr S}$.
A {\it cluster\/} is $\{ y_i \mid i\in \widetilde I\}$ of a seed $(\by,
\widetilde{\mathbf C})$, considered as a subset of ${\mathscr F}$ by
forgetting the $\widetilde I$-index.
A {\it cluster variable\/} is an element of the union of all clusters. 
Note that clusters may overlap: a cluster variable may belong to
another cluster. Also the $\widetilde I$-index may be different from
the original one.
The {\it cluster algebra\/} ${\mathscr A}(\widetilde{\bB})$ is the
subalgebra of ${\mathscr F}$ generated by all the cluster variables.
The integer $\# I$ is called the {\it rank\/} of ${\mathscr
  A}(\widetilde{\bB})$.
A {\it cluster monomial\/} is a monomial in the cluster variables 
of a {\it single\/} cluster.
The exchange relation \eqref{eq:exchange} is of the form
\begin{equation}\label{eq:exchange2}
   x_k x_k^* = m_+ + m_-
\end{equation}
where $m_\pm = \prod_{\pm b_{ik} > 0} x_i^{\pm b_{ik}}$ are cluster
monomials.

When we say a {\it cluster algebra}, it may mean the subalgebra
${\mathscr A}(\widetilde{\bB})$ or all the above structures.

One of important results in the cluster algebra theory is the Laurent
phenomenon: every cluster variable $z$ in $\mathscr
A(\widetilde{\mathbf B})$ is a Laurent polynomial in any given cluster
${\mathbf y}$ with coefficients in $\Z$. It is conjectured that the
coefficients are nonnegative.
A cluster monomial is a subtraction free rational expression in $\bx$,
but this is not enough to ensure the positivity of its Laurent
expansion, as an example $x^2 - x + 1 =  (x + 1)^3/(x + 1)$ shows.

\subsection{$F$-polynomial}\label{subsec:F-pol}

It is known that cluster variables of $\mathscr A(\widetilde{\mathbf
  B})$ are expressed by the ${\mathbf g}$-vectors and $F$-polynomials
\cite{FomZel4}, which are constructed from another cluster algebra
with the same principal part, but a simpler frozen part. We recall
their definition in this subsection.

We first prepare some notation.
We consider the multiplicative group $\mathbb P$ of all Laurent monomials
in $(x_i)_{\in I}$. We introduce the addition $\oplus$ by
\begin{equation*}
  \prod_i x_i^{a_i} \oplus \prod_i x_i^{b_i} = \prod_i x_i^{\min(a_i,b_i)}.
\end{equation*}
With this operation together with the ordinary multiplication and
division, $\mathbb P$ becomes a semifield, called the {\it tropical
  semifield}.
Let $F$ be a subtraction-free rational expression with integer
coefficients in variables $y_i$. Then we evaluate it in $\mathbb P$ by
specializing the $y_i$ to some elements $p_i$ of $\mathbb P$.  We
denote it by $F|_{\mathbb P}(\mathbf p)$, where $\mathbf p =
(p_i)_{i\in I}$.

Let $\mathscr A_{\mathrm{pr}}$ be the cluster algebra with {\it
  principal coefficients}. It is given by the initial seed $((\mathbf
u,\mathbf f),\widetilde{\mathbf B}_{\mathrm{pr}})$ with $(\mathbf
u,\mathbf f) = (u_i, f_i)_{i\in I}$, and $\widetilde{\mathbf
  B}_{\mathrm{pr}}$ is the matrix indexed by $(I\sqcup I)\times I$
with the same principal $\mathbf B$ as $\widetilde{\mathbf B}$ and the
identity matrix in the frozen part. 
Here $I_\fr$ is a copy of $I$ and $\widetilde I = I\sqcup I$.
We write a cluster variable $\alpha$ as
\begin{equation*}
   \alpha = X_\alpha(\mathbf u,\mathbf f)
\end{equation*}
a subtraction free rational expression in $\mathbf u$, $\mathbf f$.
We then specialize all the $u_i$ to $1$:
\begin{equation*}
  F_\alpha(\mathbf f) = \left.X_\alpha(\mathbf u,\mathbf f)\right|_{u_i=1}.
\end{equation*}
It becomes a polynomial in $f_i$, and called the $F$-polynomial
([loc.\ cit., \S3]).
It is also known ([loc.\ cit., \S6]) that $X_\alpha$ is homogeneous with
respect to $\Z^I$-grading given by
\begin{equation*}
   \deg u_i = i, \qquad
   \deg f_j = -\sum_i b_{ij} i,
\end{equation*}
where $b_{ij}$ is the matrix entry for the principal part $\bB$,
and the vertex $i$ is identified with the coordinate vector in $\Z^I$.
We then define $\mathbf g$-vector by
\begin{equation*}
  \mathbf g_\alpha \defeq \deg X_\alpha\in \Z^I.
\end{equation*}

We now return back to the original cluster algebra
$\mathscr A(\widetilde{\bB})\subset \Q(x_i)_{i\in \widetilde I}$.
We introduce the following variables:
\begin{equation*}
  y_j = \prod_{i\in I_\fr} x_i^{b_{ij}}, \qquad
  \widehat y_j = y_j \prod_{i\in I} x_i^{b_{ij}}
\qquad
  (j\in I).
\end{equation*}
We write $\mathbf y = (y_i)_{i\in I}$, $\widehat{\mathbf y} =
(\widehat y_i)_{i\in I}$ .

We consider the corresponding cluster variable $x[\alpha]$ in the seed
of the original cluster algebra $\mathscr A(\widetilde{\mathbf B})$
obtained by the same mutation processes as we obtained $\alpha$ in the
cluster algebra with principal coefficients. We then have
\cite[Cor.~6.5]{FomZel4}:
\begin{equation*}
  x[\alpha] = \frac{F_\alpha(\widehat{\mathbf y})}
  {\left.F_\alpha\right|_{\mathbb P}(\mathbf y)} \bx^{\mathbf g_\alpha},
\end{equation*}
where $\bx^{\mathbf g_\alpha} = \prod_{i\in I} x_i^{(\mathbf
  g_\alpha)_i}$ if $(\mathbf g_\alpha)_i$ is the
$i^{\mathrm{th}}$-entry of $\mathbf g_\alpha$.

\subsection{Hernandez-Leclerc monoidal categorification conjecture}
\label{subsec:HerLec}

We recall Hernandez-Leclerc's monoidal categorification conjecture in
this subsection.

Let $\mathscr A$ be a cluster algebra and $\mathscr M$ be an abelian
monoidal category.
A simple object $L\in\mathscr M$ is {\it prime\/} if there exists no
non trivial factorization $L\cong L_1 \otimes L_2$.  We say that $L$
is {\it real\/} if $L\otimes L$ is simple.

\begin{Definition}[\cite{HerLec}]\label{def:categorification}
  Let ${\mathscr A}$ and $\mathscr M$ as above. We say that $\mathscr
  M$ is a {\it monoidal categorification\/} of ${\mathscr A}$ if the
  Grothendieck ring of $\mathscr M$ is isomorphic to ${\mathscr A}$, and if
\begin{enumerate}
\item the cluster monomials $m$ of ${\mathscr A}$ are the classes of all
  the real simple objects $L(m)$ of $\mathscr M$;
\item the cluster variables of ${\mathscr A}$ (including the frozen
  ones) are the classes of all the real prime simple objects of
  $\mathscr M$.
  \begin{NB}
\item for an exchange relation $x_k x_k^* = m_+ + m_-$ we have an
  short exact sequence
  \begin{equation*}
     0 \to L(m_+) \to L(x_k)\otimes L(x_k^*) \to L(m_-)\to 0.
  \end{equation*}
  \end{NB}
\end{enumerate}
\end{Definition}

If two cluster variables $x$, $y$ belong to the common cluster, then
$xy$ is a cluster monomial. Therefore the corresponding simple objects
$L(x)$, $L(y)$ satisfy $L(x)\otimes L(y) \cong L(y)\otimes L(x) \cong
L(xy)$.
\begin{NB}
Probably the condition (3), which is added by the author, also follows
from (1) and (2) together with a reasonable assumption on the category
$\mathscr M$.
\end{NB}

\begin{Proposition}[\protect{\cite[\S2]{HerLec}}]\label{prop:HL}
  Suppose that a cluster algebra $\mathscr A$ has a monoidal
  categorification $\mathscr M$.

  \textup{(1)} Every cluster monomial has a Laurent expansion with
  positive coefficients with respect to any cluster $\by =
  (y_i)_{i\in\widetilde I}\in \mathscr S$;
  \begin{equation*}
     m = \frac{N_m(\mathbf y)}{\prod_i y_i^{d_i}};
     \qquad d_i\in\Z_{\ge 0},\ N(y_i)\in \Z_{\ge 0}[y_i^\pm].
  \end{equation*}
  In fact, the coefficient of $\prod y_i^{k_i}$ in $N_m(\mathbf y)$ is
  equal to the multiplicity of $L(\prod y_i^{k_i}) = \bigotimes
  L(y_i)^{\otimes k_i}$ in $L(m)\otimes L(\prod_i y_i^{d_i}) =
  L(m)\otimes\bigotimes L(y_i)^{\otimes d_i}$.

\textup{(2)} The cluster monomials of $\mathscr A$ are linearly independent.
\end{Proposition}

\begin{Conjecture}[\protect{\cite{HerLec}}]\label{HLconj}
  The cluster algebra for the quiver defined in \secref{sec:hom} has a
  monoidal categorification, when the underlying graph is of type
  $ADE$.
  More precisely it is given by a certain explicitly defined monoidal
  subcategory $\mathscr C_1$ of the category of finite dimensional
  representations of the quantum affine algebra $\Ulq$.
\end{Conjecture}

The monoidal subcategory will be defined in \subsecref{subsec:deco} in
terms of graded quiver varieties for arbitrary symmetric Kac-Moody
cases.
And we prove the conjecture for type $ADE$. This is new for $D_n$ for
$n\ge 5$ and $E_6$, $E_7$, $E_8$ since the conjecture was already
proved in \cite{HerLec} for type $A$ and $D_4$.

However we cannot control the prime factorization of arbitrary simple
modules except $ADE$ cases.
We can just prove cluster monomials are real simple objects.
And there are {\it imaginary\/} simple objects for types other than
$ADE$.
So it is still not clear that our monoidal subcategory is a monoidal
categorification in the above sense in general.
See the paragraph at the end of \secref{sec:prime} for a partial
result.
Nonetheless the statement that cluster monomials are classes of simple
objects is enough to derive the conclusions (1),(2) of
\propref{prop:HL}.

\begin{NB}
Earlier version: It was not correct.

We will prove this conjecture in this paper. In fact, we will find
that we need a modification of the definition.
\begin{itemize}
\item In view of \corref{cor:real} it will become that the terminology
  `{\it real\/} simple objects' is not correct. `Real or isotropic
  imaginary simple objects' is the correct terminology.

\item Cluster variables give classes of all the real prime simple
  objects of $\mathscr M$, but not all.
\end{itemize}

In view of Propositions~\ref{prop:Schur},\ref{prop:*****}, we have a
characterization of cluster variables in our example $\mathscr C_1$ by
using the language for the category of quiver representation.
But we do not know how to characterize them by intrinsic terms of the
monoidal category. Probably it shows that the definition of the
monoidal categorification must be changed.
\end{NB}

\begin{NB}
\section{Preliminaries (II) -- representations of quivers}

Let $\cQ = (I,\Omega)$ be a quiver as in \secref{sec:cluster}. Let
$\C\cQ$ be its path algebra defined over $\C$.
We consider the category $\rep\cQ$ of finite dimensional
representations of $\cQ$ over $\C$, which is identified with
the category of finite dimensional $\C\cQ$-modules.
The notation for quiver representations will be reviewed in the next
section in more complicated case.

\subsection{Almost splitting sequences}\label{subsec:almost}

Our reference for almost splitting sequences is \cite[IV.2 and VII.1]{ASS}.

Let us define the Auslander-Reiten translation $\tau$ by
\begin{equation*}
  \tau M = D\Ext^1(M,\C\cQ),
\end{equation*}
where $D$ is the duality operator
\(
   D(\bullet) = \Hom_\C(\bullet,\C).
\)
We also define $\tau^- M = \Ext^1(DM, \C\cQ)$.
Then we have the duality (`Grothendieck-Serre duality')
\begin{equation*}
  \Ext^1(M,N) \cong D\Hom(N,\tau M),
\qquad
  \Ext^1(M,N) \cong D\Hom(\tau^- N, M).
\end{equation*}

If $N$ is indecomposable nonprojective, we have
\(
   \tau^- \tau N \cong N.
\)
Hence the Euler characteristic satisfies the following important
property:
\begin{equation}\label{eq:Serre_dual}
  \begin{split}
  & \chi(M,\tau N) = \dim \Hom(M,\tau N) - \dim \Ext^1(M,\tau N)
\\
  =\; & \dim \Ext^1(N, M) - \dim \Hom(\tau^-\tau N,M)
  = \chi(N,M).
  \end{split}
\end{equation}

If $N$ is indecomposable and non-projective, then $\Ext^1(N,\tau
N)\cong D\End(N)$, and the latter has a special element
sending $\id$ to $1$ and the radical to $0$.
We consider the corresponding exact sequence
\begin{equation*}
  0 \to \tau N\to M \to N \to 0.
\end{equation*}
This is called an {\it almost splitting sequence\/}, has an intrinsic
definition, and has number of special properties. We use the following
\cite[IV.7.7]{ASS}:
\begin{itemize}
\item If $0\to L\xrightarrow{\alpha} M\xrightarrow{\beta} N\to 0$ is
  an almost splitting sequence and a submodule $0\neq N'\subsetneq N$,
  $0\to L\to \beta^{-1}N'\to N'\to 0$ splits, i.e.\ 
\(
   \Ext^1(N,L)\to \Ext^1(N',L)
\)
vanishes.
\begin{NB2}
Consider
\begin{equation*}
  D\Hom(L,\tau N) = \Ext^1(N,L) \to \Ext^1(N',L) = D\Hom(L,\tau N').
\end{equation*}
Since $L = \tau N$, this is transpose of
\begin{equation*}
  \Hom(\tau N,\tau N')\to \Hom(\tau N,\tau N) = \Hom(N,N).
\end{equation*}
Since image is in the radical, the assertion is clear.
\end{NB2}

\item If $0\to L\xrightarrow{\alpha} M\xrightarrow{\beta} N\to 0$ is
  an almost splitting sequence and a submodule $0\neq L'\subsetneq L$,
  $0\to L/L'\to M/\alpha(L')\to N\to 0$ splits, i.e.\ 
\(
   \Ext^1(N,L)\to \Ext^1(N,L/L')
\)
vanishes.
\end{itemize}

\subsection{Tilting modules}

Our reference for tilting modules is \cite[VI]{ASS} and
\cite{HappelUnger}.

A module $M$ of the quiver is said to be a {\it tilting module\/} if
the following two conditions are satisfied:
\begin{enumerate}
\item $M$ is {\it rigid}, i.e.\ $\Ext^1(M,M) = 0$.
\item There is an exact sequence $0\to \C\cQ \to M_0 \to M_1\to 0$
  with $M_0$, $M_1\in\add M$, where $\add M$ denotes the additive
  category generated by the direct summands of $M$.
\end{enumerate}
We usually assume $M$ is multiplicity free.

It is known that the number of indecomposable summands of $M$
equals to the number of vertexes $\# I$, i.e.\ rank of $K_0(\C\cQ)$.

A module $M$ is said to be an {\it almost complete tilting module\/}
if it is rigid and the number of indecomposable summands of $M$ is $\#
I - 1$. We say an indecomposable module $X$ is {\it complement\/} of
$M$ if $M\oplus X$ is a tilting module.

We have the following structure theorem:

\begin{Theorem}[\protect{Happel-Unger~\cite{HappelUnger}}]
  Let $M$ be an almost complete tilting module.

  \textup{(1)} If $M$ is sincere, there exists two nonisomorphic
  complements $X$, $Y$ which are related by an exact sequence
  \begin{equation*}
     0 \to X \to E\to Y \to 0
  \end{equation*}
  with $E\in\add M$. Moreover, we have $\Ext^1(Y,X) \cong \C$,
  $\Ext^1(X,Y) = 0$.

  \begin{NB2}
    I have thought that the above is an almost splitting sequence, but
    it is not correct.
  \end{NB2}

  \textup{(2)} If $M$ is not sincere, there exists only one
  complement $X$ up to isomorphisms.
\end{Theorem}

Here a module $M$ is said to be {\it sincere\/} if $M_i\neq 0$ for any
vertex $i$.

\begin{NB2}
The following was wrong. This holds only if either $\Hom(M,X) = 0$ or
$\Hom(X,M) = 0$.

In the second case and if the omitted vertex $i$ is either sink or
source (which will be always true in our setting), the complement $X$
is the indecomposable injective or projective module corresponding to
$i$ \cite[\S3]{HappelUnger}.
\end{NB2}
\end{NB}

\section{Preliminaries (II) -- Graded quiver
  varieties}\label{sec:quiver}

We review the definition of graded quiver varieties and the
convolution diagram for the tensor product in this section. Our
notation mainly follows \cite{MR2144973}. Some materials are borrowed
from \cite{VV2}.

We do not explain anything about representations of the quantum loop
algebra $\Ulq$ except in \thmref{thm:qaff}.
This is because we can work directly in the category of perverse
sheaves on graded quiver varieties.
Another reason is that it is not known whether the quantum loop
algebra $\Ulq$ can be equipped with the structure of a Hopf algebra
in general. Therefore tensor products of modules do not make sense. On
the other hand, the category of perverse sheaves {\it has\/} the
coproduct induced from the convolution diagram.

\subsection{Definition of graded quiver varieties}\label{subsec:def}

Let $q$ be a nonzero complex number. We will assume that it is not a
root of unity later, but can be at the beginning.

Suppose that a finite graph $\cG = (I,E)$ is given. We assume the graph
$\cG$ contains no edge loops. Let $\bA = (a_{ij})$ be the adjacency
matrix of the graph, namely
\begin{equation*}
   a_{ij} =
     (\text{the number of edges joining $i$ to $j$})
     .
\end{equation*}
Let $\bC = 2\mathbf I - \bA = (c_{ij})$ be the Cartan matrix.

Let $H$ be the set of pairs consisting of an edge together with its
orientation as in \secref{sec:cluster_pre}.
We choose and fix an orientation $\Omega$ of $\cG$ and define
$\varepsilon(h) = 1$ if $h\in\Omega$ and $-1$ otherwise.

Let $V$, $W$ be $I\times \C^*$-graded vector spaces such that
its $(i\times a)$-component, denoted by $V_i(a)$, is finite dimensional
and $0$ for all but finitely many $i\times a$.
In what follows we consider only $I\times\C^*$-graded vector spaces
with this condition.
We say the pair $(V,W)$ of $I\times\C^*$-graded vector spaces is
{\it $l$-dominant\/} if
\begin{equation}\label{eq:dominant}
  \dim W_i(a) - \dim V_i(a\vq) - \dim V_i(a\vq^{-1}) 
  - \sum_{j:j\neq i} c_{ij} \dim V_j(a) \ge 0
\end{equation}
for any $i$, $a$. 

Let $\bC_\vq$ ($\vq$-analog of the Cartan matrix) be an endomorphism
of $\Z^{I\times\C^*}$ given by
\begin{equation}
  (v_i(a)) \mapsto
  (v'_i(a)); \quad
  v'_i(a) = v_i(a\vq) + v_i(a\vq^{-1}) 
  + \sum_{j:j\neq i} c_{ij} v_j(a).
\end{equation}
Considering $\dim V$, $\dim W$ as vectors in $\Z_{\ge
  0}^{I\times\C^*}$, we view the left hand side of \eqref{eq:dominant}
as the $(i,a)$-component of $(\dim W - \bC_\vq\dim V)$. This is an
analog of a weight.

We say $V\le V'$ if
\begin{equation*}
  \dim V_i(a) \ge \dim V'_i(a)
\end{equation*}
for any $i$, $a$. We say $V < V'$ if $V\le V'$ and $V\neq V'$.
This is analog of the dominance order.
We say $(V,W) \le (V',W')$ if there exists ${\mathbf v}''\in \Z_{\ge
  0}^{I\times\C^*}$ whose entries are $0$ for all but finitely many
$(i,a)$ such that
\begin{equation*}
   \dim W - \bC_\vq \dim V = \dim W' - \bC_\vq (\dim V' + {\mathbf v}'').
\end{equation*}
When $W = W'$, $(V,W)\le (V',W')$ if and only if $V\le V'$.

These conditions originally come from the representation theory of the
quantum loop algebra $\Ulq$.

For an integer $n$, we define vector spaces by
\begin{equation}
\label{eq:LE}
\begin{gathered}[m]
  \HomL(V, W)^{[n]} \defeq
  \bigoplus_{i\in I, a\in\C^*}
    \Hom\left(V_i(a), W_i(a\vq^{n})\right),
\\
  \HomE(V, W)^{[n]} \defeq
  \bigoplus_{h\in H, a\in\C^*}
    \Hom\left(V_{\vout(h)}(a), W_{\vin(h)}(a\vq^{n})\right).
\end{gathered}
\end{equation}

If $V$ and $W$ are $I\times\C^*$-graded vector spaces as above, we
consider the vector spaces
\begin{equation}\label{def:bM}
  \bM \equiv \bM(V, W) \defeq
  \HomE(V, V)^{[-1]} \oplus \HomL(W, V)^{[-1]}
  \oplus \HomL(V, W)^{[-1]},
\end{equation}
where we use the notation $\bM$ unless we want to specify $V$, $W$.
The above three components for an element of $\bM$ is denoted by $B$,
$\alpha$, $\beta$ respectively.
({\bf NB}: In \cite{Na-qaff} $\alpha$ and $\beta$ were denoted by $i$, 
$j$ respectively.)
The
$\Hom\left(V_{\vout(h)}(a),V_{\vin(h)}(a\vq^{-1})\right)$-component of 
$B$ is denoted by $B_{h,a}$. Similarly, we denote by $\alpha_{i,a}$,
$\beta_{i,a}$ the components of $\alpha$, $\beta$.

We define a map $\mu\colon\bM\to \HomL(V,V)^{[-2]}$ by
\begin{equation*}
   \mu_{i,a}(B,\alpha,\beta)
   =  \sum_{\vin(h)=i} \varepsilon(h)
      B_{h,a\vq^{-1}} B_{\overline{h},a} +
   \alpha_{i,a\vq^{-1}}\beta_{i,a},
\end{equation*}
where $\mu_{i,a}$ is the $(i,a)$-component of $\mu$.

Let $G_V \defeq \prod_{i,a} \GL(V_i(a))$. It acts on $\bM$ by
\begin{equation*}
  (B, \alpha, \beta) \mapsto g\cdot (B, \alpha, \beta)
  \defeq \left(g_{\vin(h),a\vq^{-1}} B_{h,a} g_{\vout(h),a}^{-1},\,
  g_{i,a\vq^{-1}}\alpha_{i,a},\,
  \beta_{i,a} g_{i,a}^{-1}\right).
\end{equation*}
The action preserves the subvariety $\mu^{-1}(0)$ in $\bM$.

\begin{Definition}\label{def:stable}
A point $(B, \alpha, \beta) \in \mu^{-1}(0)$ is said to be {\it stable\/} if 
the following condition holds:
\begin{itemize}
\item[] if an $I\times\C^*$-graded subspace $V'$ of $V$ is
$B$-invariant and contained in $\Ker \beta$, then $V' = 0$.
\end{itemize}
Let us denote by $\mu^{-1}(0)^{\operatorname{s}}$ the set of stable points.
\end{Definition}
Clearly, the stability condition is invariant under the action of
$G_V$. Hence we may say an orbit is stable or not.

We consider two kinds of quotient spaces of $\mu^{-1}(0)$:
\begin{equation*}
   \N_0(V,W) \defeq \mu^{-1}(0)\dslash G_V, \qquad
   \N(V,W) \defeq \mu^{-1}(0)^{\operatorname{s}}/G_V.
\end{equation*}
Here $\dslash$ is the affine algebro-geometric quotient, i.e.\ the
coordinate ring of $\N_0(V,W)$ is the ring of $G_V$-invariant
functions on $\mu^{-1}(0)$. In particular, it is an affine variety. It
is the set of closed $G_V$-orbits.
The second one is the set-theoretical quotient, but coincides with a
quotient in the geometric invariant theory (see \cite[\S3]{Na-alg}).
The action of $G_V$ on $\mu^{-1}(0)^{\operatorname{s}}$ is free thanks 
to the stability condition (\cite[3.10]{Na-alg}).
By the general theory, there exists a natural projective morphism
\begin{equation*}
   \pi\colon \N(V,W) \to \N_0(V,W).
\end{equation*}
(See \cite[3.18]{Na-alg}.)
The inverse image of $0$ under $\pi$ is denoted by $\NLa(V,W)$.
We call these varieties {\it cyclic quiver varieties\/} or {\it graded
quiver varieties}, according as $\vq$ is a root of unity or {not}.
In this paper we only consider the case $q$ is not a root of unity
hereafter.
When we want to distinguish $\N(V,W)$ and $\N_0(V,W)$, we call the
former (resp.\ latter) the {\it nonsingular\/} (resp.\ {\it affine\/})
{\it graded quiver variety}.
But it does not mean that $\N_0(V,W)$ is actually singular. As we see
later, it is possible that $\N_0(V,W)$ happens to be nonsingular.

We have
\begin{equation*}
  \dim \N(V,W) = \boldsymbol\langle\dim V, (\vq  + \vq^{-1})\dim W
  - \vq^{-1}\bC_\vq\dim V\boldsymbol\rangle,
\end{equation*}
where $\vq^\pm\cdot$ is an automorphism of $\Z^{I\times\C^*}$ given by
\(
  (v_i(a)) \mapsto
  (v'_i(a));
\)
\(
  v'_i(a) = v_i(a\vq^\pm)
\)
and $\boldsymbol\langle\ ,\ \boldsymbol\rangle$ is the natural pairing
on $\Z^{I\times\C^*}$ (\cite[4.11]{MR2144973}).

The original quiver varieties \cite{Na-quiver,Na-alg} are the special
case when $\vq = 1$ and $V_i(a) = W_i(a) = 0$ except $a=1$.
On the other hand, the above varieties $\N(W)$, $\N_0(W)$ are fixed
point set of the original quiver varieties with respect to a
semisimple element in a product of general linear groups. (See
\cite[\S4]{Na-qaff}.) In particular, it follows that $\N(V,W)$ is
nonsingular, since the corresponding original quiver variety is so.
This can be also checked directly.

It is known that the coordinate ring of $\N_0(V,W)$ is generated by
the following type of elements:
\begin{equation}\label{eq:generator}
    (B,\alpha,\beta)\mapsto
    \langle \chi, \beta_{j,a \vq^{-n-1}} B_{h_n,a \vq^{-n}}
      \dots
      B_{h_1,a\vq^{-1}}\alpha_{i,a}\rangle
\end{equation}
where $\chi$ is a linear form on $\Hom(W_i(a),W_j(a q^{-n-2}))$.
(See \cite{Lusztig-On-Quiver}.)
Here we do not need to consider generators of a form $\tr(B_{h_N,a
  q^{N-1}}B_{h_{N-1}, a q^{N-2}}\cdots B_{h_1, a})$ corresponding to
an oriented cycle $h_1$, \dots, $h_N$ as they automatically vanish as
$q$ is not a root of unity.
(Our definition of the graded quiver variety is different from one in
\cite{Na-qaff} when there are multiple edges joining two vertexes.
See \remref{rem:def_quiver} for more detail.
The above generators may not vanish in the original definition, but
{\it does\/} vanish in our definition.)

Let $\Nreg(V,W)\subset\N_0(V,W)$ be a possibly empty open subset of
$\N_0(V,W)$ consisting of closed free $G_V$-orbits. It is known that
$\pi$ is an isomorphism on $\pi^{-1}(\Nreg(V,W))$ \cite[3.24]{Na-alg}.
In particular, $\Nreg(V,W)$ is nonsingular and is pure dimensional.

The $G_V$-orbit though $(B,\alpha,\beta)$, considered as a point of
$\N(V,W)$ is denoted by $[B,\alpha,\beta]$.

Suppose that we have two $I\times\C^*$-graded vector spaces $V$, $V'$
such that $V_i(a) \subset V'_i(a)$ for all $i$, $a$. Then $\N_0(V,W)$
can be identified with a closed subvariety of $\N_0(V',W)$ by the
extension by $0$ to the complementary subspace (see
\cite[2.5.3]{Na-qaff}). We consider the limit
\begin{equation*}
   \N_0(W) \defeq \bigcup_{V} \N_0(V,W).
\end{equation*}
(It was denoted by $\M_0(\infty,{\mathbf w})^A$ in \cite{Na-qaff},
and by $\N_0(\infty,W)$ in \cite{MR2144973}.)

We have $\N_0(V,0) = \{ 0\}$ for $W = 0$ since generators
\eqref{eq:generator} vanish.
Then \cite[6.5]{Na-quiver} or \cite[3.27]{Na-alg} implies that
\begin{equation}\label{eq:stratum}
   \N_0(W) = \bigsqcup_{[V]} \Nreg(V,W),
\end{equation}
where $[V]$ denotes the isomorphism class of $V$. It is known that
\begin{align}
&\begin{minipage}[m]{0.85\textwidth}
\noindent
$\Nreg(V,W)\neq\emptyset$ if and only if $\N(V,W)\neq \emptyset$ and
$(V,W)$ is {\it l\/}-dominant. (See \cite[14.3.2(2)]{Na-qaff}.)
\end{minipage}\label{eq:l-dom}
\\
&\begin{minipage}[m]{0.85\textwidth}
\noindent
If $\Nreg(V,W)\subset\overline{\Nreg(V',W)}$, then
${V'} \le V$. (This follows from \cite[\S3.3]{Na-qaff}.)
\end{minipage}\label{eq:closure}
\end{align}
It is also easy to show that 
\begin{equation}
  \label{eq:empty}
  \text{$\Nreg(V,W) = \emptyset$ if $V$ is
    sufficiently large.}
\end{equation}
(See the argument in the proof of \propref{prop:ast_1}(1).)
\begin{NB}
  This is NOT true if $\vq$ is root of unity.
\end{NB}%
Thus $\N_0(W) \defeq \bigcup_{V} \N_0(V,W)$ stabilizes at some $V$.
\begin{NB}
It is known that the above stabilizes at some $V$ if the graph is of
type $ADE$ (see \cite[2.6.3, 2.9.4]{Na-qaff}). But it is not true in
general, and dimensions of $\N_0(V,W)$ may go to $\infty$ when $V$
becomes large.
In the following, we will use $\N_0(W)$ as a brevity of the notation,
and consider its geometric structure on each $\N_0(V,W)$ individually.
We will never consider it as an infinite dimensional variety.
\end{NB}

On the other hand, we consider the disjoint union for $\N(V,W)$:
\[
   \N(W) \defeq \bigsqcup_{[V]} \N(V,W).
\]
Note that there are no obvious morphisms between $\N(V,W)$ and
$\N(V',W)$ since the stability condition is not preserved under the
extension.
We have a morphism $\N(W)\to \N_0(W)$, still denoted by $\pi$.

It is known that $\N(V,W)$ becomes empty if $V$ is sufficiently large
when $\g$ is of type $ADE$.
(Since the usual quiver variety $\M(V,W)$ is nonempty if and only if
$(\dim W - \bC\dim V)$ is a weight of the irreducible representation
with the highest weight $\dim W$. See \cite[10.2]{Na-alg}.)
But it is not true in general, and dimensions of $\N(V,W)$ may go to
$\infty$ when $V$ becomes large.
In the following, we will use $\N(W)$ as a brevity of the notation,
and consider its geometric structure on each $\N(V,W)$ individually.
We will never consider it as an infinite dimensional
variety. Furthermore, we will only need $\N(V,W)$ such that
$\Nreg(V,W)\neq\emptyset$ in practice. From the above remark, we can
stay in finite $V$'s.

\begin{NB}
Since the action is free, $V$ and $W$ can be considered as
$I\times\C^*$-graded vector bundles over $\N(V,W)$. We denote them by the
same notation.  We consider $\HomE(V,V)$, $\HomL(W,V)$, $\HomL(V, W)$
as vector bundles defined by the same formula as in \eqref{eq:LE}. By
the definition, $B$, $\alpha$, $\beta$ can be considered as sections of those
bundles.
\end{NB}

The following three term complex plays an important role:
\begin{NB}
We define a three term sequence of vector bundles over $\N(V,W)$ by
\end{NB}%
\begin{equation}
\label{eq:taut_cpx_fixed}
  C_{i,a}^\bullet(V,W):
  V_i(a\vq)
\xrightarrow{\sigma_{i,a}}
  \displaystyle{\bigoplus_{h:\vin(h)=i}}
     V_{\vout(h)}(a)
    \oplus W_i(a)
\xrightarrow{\tau_{i,a}}
  V_i(a\vq^{-1}),
\end{equation}
where
\begin{equation*}
  \sigma_{i,a} = \bigoplus_{\vin(h)=i} B_{\overline h,a\vq}
      \oplus \beta_{i,a\vq},
  \qquad
  \tau_{i,a} = \sum_{\vin(h)=i} \varepsilon(h) B_{h,a} + \alpha_{i,a}.
\end{equation*}
This is a complex thanks to the equation $\mu(B,\alpha,\beta) = 0$.
If $(B,\alpha,\beta)$ is stable, $\sigma_{i,a}$ is injective as the
$I\times\C^*$-graded vector space $V'$ given by $V'_i(aq) :=
\Ker\sigma_{i,a}$, $V'_j(b) := 0$ (otherwise) is $B$-invariant and
contained in $\Ker\beta$, and hence must be $0$.

We assign the degree $0$ to the middle term.
We define the rank of complex $C^\bullet$ by $\sum_p (-1)^p \rank
C^p$. It is exactly the left hand side of
\eqref{eq:dominant}. Therefore $(V,W)$ is {\it l\/}-dominant if and
only if
\begin{equation*}
   \rank C_{i,a}^\bullet(V,W) \ge 0
\end{equation*}
for any $i,a$. From this observation the `only-if' part of
\eqref{eq:l-dom} is clear: If we consider the complex at a point
$\Nreg(V,W)$, it is easy to see $\tau_{i,a}$ is surjective. Therefore
$\rank C_{i,a}^\bullet(V,W)$ is the dimension of the middle cohomology
group.
When $(V,W)$ is {\it l\/}-dominant, we define an $I\times\C^*$-graded
vector space $C^\bullet(V,W)$ by
\begin{equation}\label{eq:C(V,W)}
  \dim \left(C^\bullet(V,W)\right)_i(a) = \rank C_{i,a}^\bullet(V,W).
\end{equation}

\begin{Remark}\label{rem:def_quiver}
  Since we only treat graded quiver varieties of type $ADE$ in
  \cite{MR2144973}, we explain what must be modified for general
  types.

  In \cite{Na-qaff} the graded quiver varieties are the $\C^*$-fixed
  points of the ordinary quiver varieties. When there are multiple
  edges joining two vertexes, there are several choices of the
  $\C^*$-action. A choice corresponds to a choice of the $\vq$-analog
  $\bC_\vq$ of the Cartan matrix $\bC$ which implicitly appears in the
  defining relation of the quantum loop algebras. See [loc.\ cit.,
  (1.2.9)] for the defining relation and [loc.\ cit., (2.9.1)] or
  \eqref{eq:taut_cpx_fixed} for its relation to the $\C^*$-action.
  For example, consider type $A^{(1)}_1$. In [loc.\ cit.]\ the
  $q$-analog of the Cartan matrix was
  \begin{equation*}
    \begin{pmatrix}
      [2]_q & -[2]_q \\ -[2]_q & [2]_q
    \end{pmatrix}
    =
    \begin{pmatrix}
      q+q^{-1} & -(q+q^{-1}) \\ -(q+q^{-1}) & q+q^{-1}
    \end{pmatrix},
  \end{equation*}
while it is
\begin{equation*}
  \begin{pmatrix}
    [2]_q & -2 \\ -2 & [2]_q
  \end{pmatrix}
  =
    \begin{pmatrix}
      q+q^{-1} & -2 \\ -2 & q+q^{-1}
    \end{pmatrix}
\end{equation*}
in this paper. When there is at most one edge joining two vertexes, we
do not have this choice as $[1]_q = 1$.
The theory developed in \cite{Na-qaff} works for any choice of the
$\C^*$-action.

For results in \cite{MR2144973}, we need a little care.
First of all, [loc.~cit., Cor.~3.7] does not make sense since it is
not known whether we have tensor products in general as we already
mentioned.
For the choice of the $\C^*$-action in this paper, all other results
hold without any essential changes, except assertions when
$\varepsilon$ is a root of unity or $\pm 1$. (In these cases, we will
get new types of strata so the assertion must be modified. For the
affine type, they can be understood from \cite{MR1989196}.)
If we take the $\C^*$-action in \cite{Na-qaff}, the recursion used to
prove {\bf Axiom 2} does not work. So we first take the $\C^*$-action
in this paper, and then apply the same trick used to deal with cyclic
quiver varieties. In particular, we need to include analog of {\bf
  Axiom 4}. Details are left as an exercise for the reader of
\cite{MR2144973}.
\end{Remark}

\subsection{Transversal slice}\label{subsec:slice}

Take a point $x\in \Nreg(V^0,W)$. Let $T$ be the tangent space of
$\Nreg(V^0,W)$ at $x$.
Since $\Nreg(V^0,W)$ is nonempty, $(V^0,W)$ is {\it l\/}-dominant,
i.e.\ \eqref{eq:dominant} holds by \eqref{eq:l-dom}.
Let $W^\perp = C^\bullet(V^0,W)$ as in \eqref{eq:C(V,W)}.
\begin{NB}
\begin{equation*}
    \dim W^\perp = \dim W - \bC_\vq \dim V^0,
\end{equation*}
\end{NB}

We consider another graded quiver variety $\N_0(V,W)$ which contains
$x$ in its closure. By \eqref{eq:closure} we have $V\le
V^0$. Therefore we can consider $V^\perp$, $I\times\C^*$-graded vector
space whose $(i,a)$-component has the dimension $\dim V_i(a) - \dim
V^0_i(a)$. We have
\(
  \dim W - \bC_\vq \dim V
  = \dim W^\perp - \bC_{\vq} \dim V^\perp,
\)
which means the `weight' is unchanged under this procedure.

\begin{Theorem}[\protect{\cite[\S3.3]{Na-qaff}}]\label{thm:slice}
  We work in the complex analytic topology. There exist neighborhoods
  $U$, $U_T$, $U_{\mathfrak S}$ of $x\in \N_0(V,W)$, $0\in T$,
  $0\in\N_0(V^\perp,W^\perp)$ respectively, and biholomorphic maps
  $U\to U_T\times U_{\mathfrak S}$, $\pi^{-1}(U)\to U_T\times
  \pi^{-1}(U_{\mathfrak S})$ such that the following diagram
  commutes\textup:
\begin{equation*}
\begin{CD}
  \N(V,W) \;@. \supset \; @.
   \pi^{-1}(U) @>>\cong> U_T \times \pi^{-1}(U_{\mathfrak S})
     \;@.\subset \;@. T\times \N(V^\perp,W^\perp)
\\
   @. @. @V{\pi}VV @VV{\operatorname{id}\times\,\pi}V @.@. \\
   \N_0(V,W)\;@.\supset \; @.
   U @>>\cong> U_T\times U_{\mathfrak S}
                  \;@.\subset \;@. T\times\N_0(V^\perp,W^\perp)
\end{CD}
\end{equation*}
Furthermore, a stratum $\Nreg(V',W)$ of $\N_0(V,W)$ is mapped to a
product of $U_T$ and the stratum $\Nreg(V^{\prime\perp},W^\perp)$ of
$\N_0(V^\perp,W^\perp)$.
\end{Theorem}
Here $V^{\prime\perp}$ is defined exactly as $V^\perp$ replacing $V$
by $V'$, i.e.\ $\dim V^{\prime\perp} = \dim V' - \dim V^0$.

Note that $V''\le V' \Leftrightarrow V^{\prime\prime\perp}\le V^{\prime\perp}$
if we define $V^{\prime\prime\perp}$ for $V''$ in the same way.

See also \cite{CB:normal} for the same result in the \'etale topology.

\subsection{The additive category $\mathscr Q_W$ and the Grothendieck
  ring}

Let $X$ be a complex algebraic variety. Let $\mathscr D(X)$ be the
bounded derived category of constructible sheaves of $\C$-vector
spaces on $X$. For $j\in\Z$, the shift functor is denoted by
$L\mapsto L[j]$. The Verdier duality is denoted by $D$.
For a locally closed subvariety $Y\subset X$, we denote by $1_Y$ the
constant sheaf on $Y$. We denote by $IC(Y)$ the intersection
cohomology complex associated with the trivial local system $1_{Y}$ on
$Y$. Our degree convention is so that $IC(Y)|_Y = 1_Y[\dim Y]$.

Since $\pi\colon \N(V,W)\to \N_0(V,W)$ is proper and $\N(V,W)$ is
smooth, $\pi_!(1_{\N(V,W)})$ is a direct sum of shifts of simple
perverse sheaves on $\N_0(V,W)$ by the decomposition theorem
\cite{BBD}.
We denote by $\mathscr P_W$ the set of isomorphism classes of
simple perverse sheaves obtained in this manner, considered as a
complex on $\N_0(W)$ by extension by $0$ to the complement of
$\N_0(V,W)$.
By \cite[\S14]{Na-qaff}
$\mathscr P_W = \{ IC(\Nreg(V,W)) \mid \Nreg(V,W)\neq\emptyset\}$.
By \eqref{eq:empty} $\# \mathscr P_W < \infty$.
Set $IC_W({V})\defeq IC(\Nreg(V,W))$.
Let $\mathscr Q_W$ be the full subcategory of $\mathscr D(\N_0(W))$
whose objects are the complexes isomorphic to finite direct sums of
$IC_{W}(V)[k]$ for various $IC_W(V)\in\mathscr P_W$, $k\in\Z$.
Let $\pi_W({V}) \defeq \pi_!(1_{\N(V,W)}[\dim \N(V,W)])$. By the
definition, we have $\pi_W({V})\in \mathscr Q_W$.
The subcategory $\mathscr Q_W$ is preserved under $D$ and elements
in $\mathscr P_W$ are fixed by $D$.

Let $\mathcal K(\mathscr Q_W)$ be the abelian group with one generator
$(L)$ for each isomorphism class of objects of $\mathscr Q_W$ and with
relations $(L)+(L')=(L'')$ whenever $L''$ is isomorphic to $L\oplus
L'$.  It is a module over $\mathcal A = \Z[t,t^{-1}]$ by $t(L) =
(L[1])$, $t^{-1}(L) = (L[-1])$. It is a free $\mathcal A$-module with
base $\{ (IC_{W}(V)) \mid IC_{W}(V)\in\mathscr P_W\}$. The duality $D$
defines the {\it bar involution\/}
$\setbox5=\hbox{A}\overline{\rule{0mm}{\ht5}\hspace*{\wd5}}$ on
$\mathcal K(\mathscr Q_W)$ fixing $(IC_{W}(V))$ and satisfying
$\overline{t(L)} = t^{-1}\overline{(L)}$.
Since $\pi$ is proper and $\N(V,W)$ is smooth, we also have
$\overline{(\pi_W({V}))} = (\pi_W({V}))$.
We do not write $(\ )$ hereafter.

There is another base 
\begin{equation*}
  \{ \pi_W({V}) \mid
  \text{$(V,W)$ is {\it l\/}-dominant, $\N(V,W)\neq\emptyset$}\}.
\end{equation*}
Note that $\pi_W({V})$ make sense for any $V$ without the {\it
  l\/}-dominance condition, but we need to take only {\it
  l\/}-dominant ones to get a base.
Let us define $a_{V,V';W}(t)\in\mathcal A$ by
\begin{equation}\label{eq:a}
   \pi_W({V}) = \sum_{V'} a_{V,V';W}(t)\, IC_W({V'}).
\end{equation}
Then we have $a_{V,V';W}(t)\in\Z_{\ge 0}[t,t^{-1}]$, $a_{V,V;W}(t) = 1$ and
$a_{V,V';W} = 0$ unless $V'\le V$.
Since both $\pi_W({V})$ and $IC_W({V'})$ are fixed by the bar involution,
we have $a_{V,V';W}(t) = a_{V,V';W}(t^{-1})$.
It also follows that we only need to consider $\pi_!(1_{\N(V,W)})$ for
which $(V,W)$ is {\it l-dominant\/} in the definition of $\mathscr
P_W$.

Take $V^0$ such that $\Nreg(V^0,W)\neq\emptyset$.
Taking account the transversal slice in \subsecref{subsec:slice}, we
define a surjective homomorphism
\(
  p_{W^\perp,W}\colon
   \mathcal K(\mathscr Q_W)\to \mathcal K(\mathscr Q_{W^\perp})
\)
by
\begin{equation*}
\quad
   IC_W({V}) \mapsto
   \begin{cases}
     IC_{W^\perp}(V^\perp) & 
     \text{if $\Nreg(V^0,W)\subset\overline{\Nreg(V,W)}$,}
\\
     0 & \text{otherwise}.
   \end{cases}
\end{equation*}
By \thmref{thm:slice} this homomorphism is also compatible with
$\pi_W({V})$. Taking various $V$'s, $\mathcal K(\mathscr Q_W)$'s form a
projective system.

We consider the dual $\mathcal K(\mathscr Q_W)^* = \Hom_{\mathcal
  A}(\mathcal K(\mathscr Q_W),\mathcal A)$.
Let $\{ L_W({V}) \}$, $\{ \chi_W({V})\}$ be the bases of $\mathcal
K(\mathscr Q_W)^*$ dual to $\{ IC_W({V})\}$, $\{ \pi_W({V})\}$
respectively. Here $V$ runs over the set of isomorphism classes of
$I\times\C^*$-graded vector spaces such that $(V,W)$ is {\it
  l\/}-dominant.
We consider yet another base $\{ M_W({V})\}$ of $\mathcal K(\mathscr
Q_W)^*$ given by
\begin{equation*}
  \mathcal K(\mathscr Q_W)\ni (L)\mapsto
  \sum_k t^{\dim \Nreg(V,W)-k} \dim H^k(i_{x_{V,W}}^! L)\in\mathcal A,
\end{equation*}
where $x_{V,W}$ is a point in $\Nreg(V,W)$ and $i_{x_{V,W}}$ is the
inclusion of the point $x_{V,W}$ in $\N_0(W)$. By \thmref{thm:slice}
it is independent of the choice of $x_{V,W}$. Also it is compatible
with the projective system: 
if $V^0\ge V' \ge V$,
\( 
   \langle M_W(V'), IC_W(V)\rangle = 
   \langle M_{W^\perp}(V^{\prime\perp}), IC_{W^\perp}(V^\perp)\rangle.
\)

By the defining property of perverse sheaves, we have
\begin{equation}\label{eq:upper}
  L_W({V}) \in M_W({V}) + \sum_{V':V' > V} t^{-1}\Z[t^{-1}] M_W({V'}).
\end{equation}
\begin{NB}
We must have $\Nreg(V',W)\subset \overline{\Nreg(V,W)}$.
\end{NB}%
Since there are only finitely many $V'$ with $V'> V$, this is a finite
sum. This shows that $\{ M_W(V)\}_V$ is a base.
Recall also that the canonical base $L_W({V})$ is characterized by
this property together with $\overline{L_W(V)} = L_W(V)$.
It is the analog of the characterization of Kazhdan-Lusztig base. This
is not relevant in this paper, but was important to compute $L_W(V)$
explicitly in \cite{MR2144973}.

Let
\begin{equation*}
  \bfR_t \defeq 
  \left\{\left. (f_W) \in \prod_W 
      \Hom_{\mathcal A}(\mathcal K(\mathscr Q_W),\mathcal A)
      \,\right|
    \txt{$
      \langle f_W, IC_W(V)\rangle 
      = \langle f_{W^\perp}, IC_{W^\perp}(V^\perp)\rangle$
      \\ for any $W$, $W^\perp$ as above}
  \right\}.
\end{equation*}
A functional $(f_W)\in\bfR_t$ is determined when all values
$\langle f_{W^\perp}, IC_{W^\perp}(0)\rangle$ are given for any $W^\perp$.
Let $L(W)$, $\chi({W})$, $M(W)$ be the functional determined from
$L_W({0})$, $\chi_W({0})$, $M_W({0})$ respectively.
For example, 
\(
  \langle L(W), IC_{W'}(V')\rangle 
  = \delta_{\dim W,\dim W'-\bC_\vq\dim V'}.
\)
\begin{NB}
$\langle L(W'), IC_W(V)\rangle
= \langle L_{W}(\bC_q^{-1}(W - W'), IC_{W}(V)\rangle
= \delta_{\bC_q^{-1}(W-W'),V}
= \delta_{W-\bC_qV, W'}$
\end{NB}%
They form analog of {\it canonical\/}, {\it monomial\/} and {\it
  PBW\/} bases of $\bfR_t$ respectively.
From \eqref{eq:upper} the transition matrix between the canonical and
monomial bases are upper triangular with respect to the ordering
$(0,W)\le (0,W')$.

\begin{NB}
  If $\g$ is of finite type, there exists $V(W)$ such that
$M_W(V) \neq \emptyset$ only if $0\ge V \ge V(W)$.
\end{NB}

The following is the main result in \cite{Na-qaff}.
\begin{Theorem}[\protect{\cite[14.3.10]{Na-qaff}}]\label{thm:qaff}
  As an abelian group, $\left.\bfR_t\right|_{t=1}$ is isomorphic to
  the Grothendieck group of the category $\mathscr R$ of {\it
    l\/}-integrable representations of the quantum loop algebra $\Ulq$
  of the symmetric Kac-Moody Lie algebra $\g$ given by the Cartan
  matrix $\bC$ so that
  \begin{itemize}
  \item $L(W)$ corresponds to the class of the simple module whose
    Drinfeld polynomial is given by
    \begin{equation*}
      P_i(u) = \prod_{a\in\C^*} (1 - au)^{\dim W_i(a)}\qquad
      (i\in I).
    \end{equation*}
  \item $M(W)$ corresponds to the class of the standard module whose
    Drinfeld polynomial is given by the same formula.
  \end{itemize}
\end{Theorem}

Since we do not need this result in this paper, except for an
explanation of our approach to one in \cite{HerLec}, we do not explain
terminologies and concepts in the statement. See \cite{Na-qaff}.

From a general theory of the convolution algebra (see \cite{CG}),
$\mathcal K(\mathscr Q_W)$ is the Grothendieck group of the category
of {\it graded\/} representations of the convolution algebra
\(
  H_*(\N(W)\times_{\N_0(W)}\N(W)) 
  \cong \bigoplus_{V^1,V^2}
  \Ext^*_{\mathscr D(\N_0(W))}(\pi_W(V^1),\pi_W(V^2)),
\)
where the grading is for $\Ext^\bullet$-group. And $\{ L_W(V)\}$ is the base
given by classes of simple modules.

Let us briefly explain how we glue the abelian categories for various
$W$ to get a single abelian category. A family of graded module
structures $\{ \rho_W\colon H_*(\N(W)\times_{\N_0(W)}\N(W))
\to \End_\C(M)\}_W$ on a single vector space $M$ is said to be {\it
  compatible}, if $\rho_W$ factors through various restrictions to
open subsets in \thmref{thm:slice} and the restrictions are compatible
with the restriction of $\rho_{W^\perp}$ under the local isomorphisms
in \thmref{thm:slice}. For example, we fix $W^0$ and choose various
points $x_{V,W}\in \Nreg(V,W)$ with $\dim W - \bC_\vq\dim V = \dim
W^0$. We identify $H_*(\pi^{-1}(x_{V,W}))$ with a single vector space
$M$, say $H_*(\pi^{-1}(x_{0,W^0}))$, by the local isomorphisms.  It is
a compatible family of module structures.
Compatible families form an abelian category. Let us denote it by
$\mathscr R_{\mathrm{conv}}$. Then we have $\mathcal K(\mathscr
R_{\mathrm{conv}})\cong \bfR_t$.
In the above theorem, we have families of homomorphisms $\Ulq\to
H_*(\N(W)\times_{\N_0(W)}\N(W))$ compatible with the local
isomorphisms. Therefore we have a functor from $\mathscr
R_{\mathrm{conv}}$ to the category $\mathscr R$ of {\it
  l\/}-integrable representations of $\Ulq$. It sends a simple object
to a simple module. We do not know whether it is an equivalence (after
forgetting the grading on $\mathscr R_{\mathrm{conv}}$), but we can
get enough information practically.

\subsection{$t$-analog of $q$-characters}\label{subsec:q-char}

For each $(i,a)\in I\times\C^*$, we introduce an indeterminate
$Y_{i,a}$. Let
\begin{equation*}
  \mathscr Y_t \defeq
  \mathcal A[Y_{i,a}, Y_{i,a}^{-1}]_{i\in I,a\in\C^*}.
\end{equation*}
We associate polynomials $e^{W}$, $e^{V}\in\mathscr Y_t$ to graded
vector spaces $V$, $W$ by
\begin{equation*}
   e^{W} = \prod_{i\in I,a\in\C^*} Y_{i,a}^{\dim W_i(a)},
\quad
   e^{V} = \prod_{i\in I,a\in\C^*} V_{i,a}^{\dim V_i(a)},
\quad\text{where }
   V_{i,a} =  Y_{i,a\vq^{-1}}^{-1} Y_{i,a\vq}^{-1}
   \prod_{\substack{h\in H\\ \vout(h)=i}} Y_{\vin(h),a}.
\end{equation*}

We define the {\it $t$-analog of $\vq$-character\/} for $M(W)$ by
\begin{equation*}
  \chi_{\vq,t}(M(W)) \defeq \sum_V
  \sum_k t^{-k} \dim H^k(i_0^! \pi_W({V})) e^W e^V,
\end{equation*}
where $0$ is the unique point of $\N_0(0,W)$.
\begin{NB}
  Check : 
  \begin{equation*}
    P_t(X[1]) =
    \sum_k t^{-k} \dim H^k(X[1])
    =  \sum_k t^{-k} \dim H^{k+1}(X) 
    = t \sum_k t^{-k-1} \dim H^{k+1}(X) = t P_t(X).
  \end{equation*}

  This is equal to the generating function of the Poincar\'e
  polynomials
  \begin{equation*}
    \sum_i t^{i-\dim \N(V,W)} \dim H_i(\NLa(V,W)).
  \end{equation*}
  This follows from \cite[8.5.4]{CG}:
  \begin{equation*}
    H^k(i_0^! \pi_W(V))
    \cong H_{\dim \N(V,W)-k}(\NLa(V,W)).
  \end{equation*}
  Now we put $i = \dim \N(V,W)-k$, and hence $-k = i - \dim \N(V,W)$.
\end{NB}%
From the definition in the previous subsection, this is nothing but
the generating function of pairings $\langle M_{W}(0),
\pi_{W}(V)\rangle$ for various $V$. If $\g$ is of type $ADE$,
$\N(V,W)$ becomes empty for large $V$,
as we mentioned in \subsecref{subsec:def}.
Therefore this is a finite sum. If $\g$ is not of type $ADE$, this
becomes an infinite series, so it lives in a completion of $\mathscr
Y_t$. Since the difference is not essential, we keep the notation
$\mathscr Y_t$. Anyway we will only use the truncated $\vq$-character,
which is in $\mathscr Y_t$, in this paper.

Suppose $(V^0,W)$ is {\it l\/}-dominant and we define $V^\perp$,
$W^\perp$ as in \subsecref{subsec:slice}. Then
\begin{equation*}
  \sum_V \langle M_W(V^0),  \pi_W(V)\rangle
  e^W e^V
  = \sum_V \langle M_{W^\perp}(0), \pi_{W^\perp}(V^\perp) \rangle
  e^{W^\perp} e^{V^\perp}
  = \chi_{\vq,t}(M(W^\perp))
\end{equation*}
as $e^W e^V = e^{W^\perp} e^{V^\perp}$.

Since $\{ M(W)\}$ is a base of $\bfR_t$, we can extend $\chi_{\vq,t}$ 
to $\bfR_t$ linearly. We have
\begin{equation}\label{eq:q-char}
   \chi_{\vq,t}(L(W))
   =  \sum_V \langle L_W(0),\pi_W(V)\rangle\, e^W e^V
   =  \sum_V a_{V,0;W}(t)\, e^W e^V,
\end{equation}
where $a_{V,0;W}$ is the coefficient of $IC_{W}(0)=1_{\{0\}}$ in
$\pi_{W}(V)$ (in $\mathcal K(\mathscr Q_W)$) as in \eqref{eq:a}.
\begin{NB}
Suppose $L(W) = \sum_{W'} a^{W'}_W(t) M(W')$. Take the coefficient
of $e^W e^V$ in $\chi_{\vq,t}(L(W))$:
\begin{equation*}
  \begin{split}
   & \sum_{W'} a^{W'}_W(t) [ e^W e^V : \chi_{\vq,t}(M(W'))]
   = 
   \sum_{W',V'} a^{W'}_W(t) [ e^W e^V : 
   \langle M_{W'}(0), \pi_{W'}(V') \rangle e^{W'}e^{V'}]
\\
  =\; &  
   \sum_{W'} a^{W'}_W(t) 
   \langle M_{W'}(0), \pi_{W'}(\bC_{\vq}^{-1}(W'-W)+V) \rangle
   =
   \sum_{W'} a^{W'}_W(t) 
   \langle M_{W}(\bC_{\vq}^{-1}(W-W')), \pi_{W}(V) \rangle
\\
  =\; &  
  \sum_{W'} a^{W'}_W(t) 
   \langle M({W'}), \pi_{W}(V) \rangle
  = \langle L(W),\pi_W(V)\rangle = \langle L_W(0),\pi_W(V)\rangle.
  \end{split}
\end{equation*}
Therefore we have
\begin{equation*}
  \chi_{\vq,t}(L(W))
  = \sum_{V} \langle L_W(0),\pi_W(V)\rangle e^W e^V.
\end{equation*}

\end{NB}

Since $\{ M_W(V)\}_{\text{$(V,W)$:{\it l\/}-dominant}}$ forms a base
of $\bfR_t$, we have
\begin{Theorem}
The $q$-character homomorphism $\chi_{\vq,t}\colon
\bfR_t\to \mathscr Y_t$ is injective.
\end{Theorem}

But $\chi_{\vq,t}$ also contains terms from $\pi_W(V)$ with $(V,W)$
not necessary {\it l\/}-dominant. These are redundant information.

\begin{Remark}\label{rem:negative}
  By \cite[Th.~3.5]{MR2144973} the coefficient of $e^W e^V$ in the
  $t$-analog of $\vq$-characters for standard modules $M(W)$ is in
  $t^{\dim \N(V,W)}\Z_{\ge 0}[t^{-2}]$. This was a consequence of
  vanishing of odd cohomology groups of $\NLa(V,W)$.
  From the proof of \cite[Lem.~8.7.8]{CG} together with the above
  vanishing result, we have
\[
   a_{V,0;W}(t) \in t^{\dim \N(V,W)} \Z_{\ge 0}[t^{-2}].
\]
\end{Remark}

\subsection{A convolution diagram}\label{subsec:conv}

Let us take a $2$-step flag $0 \subset W^2\subset W$ of
$I\times\C^*$-graded vector spaces. We put $W/W^2 = W^1$. Following
\cite{Na-tensor}, we introduce closed subvarieties in $\N_0(W)$ and
$\N(W)$:
\begin{equation*}
  \begin{split}
  & \Zm_0(W^1;W^2)
  = \{ [B,\alpha,\beta]\in \N_0(W) \mid 
  \text{$W^2$ is invariant under
    $\beta B^k \alpha$ for any $k\in\Z_{\ge 0}$}\},
\\
  & \Zm(W^1;W^2) = \pi^{-1}(\Zm_0(W^1;W^2)).
  \end{split}
\end{equation*}
This definition is different from the original one, but equivalent
[loc.\ cit., 3.6, 3.7].
The latter has an $\alpha$-partition
\begin{equation*}
   \Zm(W^1;W^2) = \bigsqcup \Zm(V^1,W^1; V^2,W^2)
\end{equation*}
such that $\Zm(V^1,W^1; V^2,W^2)$ is a vector bundle over
$\N(V^1,W^1)\times\N(V^2,W^2)$ of rank
\begin{equation*}
  \boldsymbol\langle \dim V^1,
  \vq^{-1}(\dim W^2 - \bC_\vq\dim V^2)\boldsymbol\rangle
  + \boldsymbol\langle \dim V^2, \vq\dim W^1\boldsymbol\rangle.
\end{equation*}
(See [loc.\ cit., 3.8].)
Let us denote this rank by
\begin{equation*}
   d({V^1}, {W^1}; {V^2}, {W^2}).
\end{equation*}
(It was denoted by $d(e^{V^1}e^{W^1}, e^{V^2}e^{W^2})$ in \cite{MR2144973}.)

Following \cite{VV2} we consider the diagram
\begin{equation*}
   \N_0(W^1)\times \N_0(W^2) \xleftarrow{\kappa}
   \Zm_0(W^1;W^2)\xrightarrow{\iota} \N_0(W),
\end{equation*}
where $\iota$ is the inclusion and $\kappa$ is given by the induced
maps from $\beta B^k\alpha$ to $W^1 = W/W^2$, $W^2$. Then we define
a functor
\begin{equation*}
  \Tilde\res_{W^1,W^2}\defeq \kappa_! \iota^*
  \colon \mathscr D(\N_0(W))
  \to \mathscr D(\N_0(W^1)\times \N_0(W^2)).
\end{equation*}
We have
\begin{equation*}
  \Tilde\res_{W^1,W^2}(\pi_W(V))
  = \bigoplus_{V^1+V^2=V} \pi_{W^1}(V^1)\boxtimes \pi_{W^2}(V^2)
   [d(V^2,W^2;V^1,W^1) - d(V^1,W^1;V^2,W^2)].
\end{equation*}
(See \cite[Lemma~4.1]{VV2}. A weaker statement was given in
\cite[6.2(3)]{MR2144973}.)
From this observation objects in $\mathscr Q_W$ are sent to $\mathscr
Q_{W^1\times W^2}$, the full subcategory of $\mathscr
D(\N_0(W^1)\times \N_0(W^2))$ whose objects are complexes isomorphic
to finite direct sums of $IC_{W^1}(V^1)\boxtimes IC_{W^2}(V^2)[k]$ for
various $IC_{W^1}(V^1)\in\mathscr P_{W^1}$, $IC_{W^2}(V^2)\in\mathscr
P_{W^2}$, $k\in\Z$ (\cite[Lemma~4.1]{VV2}). 
Therefore this functor induces a homomorphism $\mathcal K(\mathscr
Q_W)\to \mathcal K(\mathscr Q_{W^1})\otimes_{\mathcal A}\mathcal
K(\mathscr Q_{W^2})$. 
It is coassociative, as $\mathcal K(\mathscr Q_W)$ is spanned by
classes $\pi_W(V)$ and they satisfy the coassociativity from the above
formula.
We denote it also by $\Tilde\res_{W^1,W^2}$.

Let $\bC_\vq^{-1}$ be the inverse of $\bC_\vq$. We define it by
solving the equation $(u_i(a)) = \bC_{\vq} (x_i(a))$ recursively
starting from $x_i(a\vq^s) = 0$ for sufficiently small $s$. Note that
$x_i(a)$ may be nonzero for infinitely many $a$.
\begin{NB}
In other words, $u_{i}(a)$ is defined recursively by
\begin{equation*}
  u_i(a) = w_i(a\vq^{-1}) - \sum_j a_{ij} u_j(a\vq^{-1}).
\end{equation*}
\end{NB}%
We then observe
\begin{equation*}
  d(V^1,W^1;V^2,W^2) 
  - \boldsymbol\langle \bC_\vq^{-1}\dim W^1,
  \vq^{-1}\dim W^2\boldsymbol\rangle
\end{equation*}
is preserved under the replacement $\N(V^1,W^1)\times \N(V^2,W^2)
\rightsquigarrow \N(V^{1\perp},W^{1\perp})\times
\N(V^{2\perp},W^{2\perp})$ by the transversal slice
(\cite[Lemma~3.2]{VV2}).
\begin{NB}
We observe
\begin{equation*}
\begin{split}
  &   \boldsymbol\langle \dim V^1,
  \vq^{-1}(\dim W^2 - \bC_\vq\dim V^2)\boldsymbol\rangle
  + \boldsymbol\langle \dim V^2, \vq\dim W^1\boldsymbol\rangle
\\
  & =\;
  \begin{aligned}[t]
  & - \boldsymbol\langle \bC_\vq^{-1}(\dim W^1 - \bC_\vq\dim V^1),
  \vq^{-1}(\dim W^2 - \bC_\vq\dim V^2)\boldsymbol\rangle
\\
  & + \boldsymbol\langle \dim V^2, \vq\dim W^1\boldsymbol\rangle
  +   \boldsymbol\langle \bC_\vq^{-1}\dim W^1,
  \vq^{-1}(\dim W^2 - \bC_\vq\dim V^2)\boldsymbol\rangle
  \end{aligned}
\\
  =\; &
  - \boldsymbol\langle \bC_\vq^{-1}(\dim W^1 - \bC_\vq\dim V^1),
  \vq^{-1}(\dim W^2 - \bC_\vq\dim V^2)\boldsymbol\rangle
  + \boldsymbol\langle \bC_\vq^{-1}\dim W^1, \vq^{-1}\dim W^2\boldsymbol\rangle.
\end{split}
\end{equation*}
\end{NB}%
Therefore we define
\begin{equation*}
  \begin{split}
    \varepsilon(W^1,W^2) & \defeq
    \boldsymbol\langle \bC_\vq^{-1}\dim W^1, \vq^{-1}\dim W^2\boldsymbol\rangle
  - \boldsymbol\langle \bC_\vq^{-1}\dim W^2, \vq^{-1}\dim W^1\boldsymbol\rangle,
\\
    \res & \defeq \sum_{W = W^1\oplus W^2} \Tilde\res[\varepsilon(W^1,W^2)]
  \end{split}
\end{equation*}
Then its transpose defines a multiplication on $\bfR_t$, which
is denoted by $\otimes$. 

\begin{NB}
Concretely we have
\begin{equation*}
  \begin{split}
  & \langle \bC_{\vq}^{-1}\dim W^1, \vq^{-1}\dim W^2\rangle
  = \sum_i u_{i}^1(\vq) w^2_i(1) + u_i^1(\vq^2) w^2_i(\vq)
  + u_i^1(\vq^3) w^2_i(\vq^2) + u_i^1(\vq^4) w^2_i(\vq^3)
  + \cdots
\\
  =\; &
  \sum_i
 \begin{aligned}[t]
  & w_{i}^1(1) w^2_i(1) 
  + (w_i^1(\vq) - \sum_{j} a_{ij} w_j^1(1)) w^2_i(\vq)
\\
  & + (w_i^1(\vq^2) - \sum_j a_{ij} u_j^1(\vq^2) - u_i^1(\vq)) w^2_i(\vq^2)
  + (w_i^1(\vq^3) - \sum_j a_{ij} u_j^1(\vq^3) - u_i^1(\vq^2)) w^2_i(\vq^3)
  + \cdots.
  \end{aligned}
  \end{split}
\end{equation*}
Then we have
\begin{equation*}
  \begin{split}
  & a_{ij} u_j^1(\vq^2) w^2_i(\vq^2)
  = a_{ij} \left(w_j^1(\vq) - \sum_{k} a_{jk} w_k^1(1)\right)
  w^2_i(\vq^2),
\\
  & u_i^1(\vq^2) w_i^2(\vq^3)
  = \left( w_i^1(\vq) - \sum_k a_{ik} w_k^1(1)\right) w_i^2(\vq^3),
\\  
  &
  a_{ij} u_j^1(\vq^3) w^2_i(\vq^3)
  = a_{ij} \left(w_j^1(\vq^2) 
      - \sum_{k} a_{jk} u_k^1(\vq^2)\right) w^2_i(\vq^3)
\\
  & =
  \sum_j 
  \left(
  a_{ij} \left( w^1_j(\vq^2) - \sum_k a_{jk} \left( w^1_k(\vq) - \sum_l a_{kl} w^1_l(1)\right)\right)\right)
  w^2_i(\vq^3).
  \end{split}
\end{equation*}
Therefore
\begin{multline*}
  \langle \bC_{\vq}^{-1}\dim W^1, \dim W^2\rangle
\\
  = \sum_{i,a} w_i^1(a) w_i^2(a) - w_i^1(a) w_i^2(a\vq^2)
    - \sum_{i,j,a} a_{ij} (w_i^1(a) w_j^2(a\vq)
    - w_i^1(a) w_j^2(a\vq^3))
\\
    + \sum_{i,j,k,a} a_{ij}a_{jk} w_i^1(a) w_k^2(a\vq^2)
    - \sum_{i,j,k,l,a} a_{ij}a_{jk}a_{kl} w_i^1(a) w_l^2(a\vq^3).
\end{multline*}
No further terms when we assume $(\ast_{\ell=1})$.
\end{NB}

We also define the twisted multiplication on $\mathscr Y_t$ given by
\begin{equation}\label{eq:twist}
  m_1 \ast m_2 = t^{\varepsilon(\vec{m}_1,\vec{m}_2)
    } m_1 m_2,
\end{equation}
where $m_1$, $m_2$ are monomials in $Y_{i,a}^\pm$ and
$\vec{m}_\alpha = (m^\alpha_i(a))$ is given by
$m_\alpha = \prod Y_{i,a}^{m_i^\alpha(a)}$.
\begin{NB}
We observe
\begin{equation*}
  \bigoplus_V \operatorname{Res}(\pi_W(V)) = 
  \bigoplus_{\substack{W^1 + W^2 = W \\ V^1+V^2=V}}
  \pi_{W^1}(V^1)\boxtimes \pi_{W^2}(V^2)[k]
\end{equation*}
with
\begin{equation*}
  \begin{split}
  k & =
  d(V^2,W^2;V^1,W^1) - d(V^1,W^1;V^2,W^2) + \ve(W^1,W^2)
\\
   &=
  d(V^2,W^2;V^1,W^1) 
  - \boldsymbol\langle\bC_{\vq}^{-1}\dim W^2,\vq^{-1}\dim W^1
    \boldsymbol\rangle
  - 
  d(V^1,W^1;V^2,W^2) 
  + \boldsymbol\langle\bC_{\vq}^{-1}\dim W^1,\vq^{-1}\dim W^2
    \boldsymbol\rangle.
  \end{split}
\end{equation*}
Since we have
\begin{equation*}
  \begin{split}
 & d(V^2,W^2;V^1,W^1) 
  - \boldsymbol\langle\bC_{\vq}^{-1}\dim W^2,\vq^{-1}\dim W^1
    \boldsymbol\rangle 
    = -\boldsymbol\langle\bC_{\vq}^{-1} (\dim W^2 - \bC_\vq\dim V^2),
    \vq^{-1}(\dim W^1 - \bC_\vq\dim V^1)
    \boldsymbol\rangle,
\\
 & d(V^1,W^1;V^2,W^2) 
  - \boldsymbol\langle\bC_{\vq}^{-1}\dim W^1,\vq^{-1}\dim W^2
    \boldsymbol\rangle 
    = -\boldsymbol\langle\bC_{\vq}^{-1} (\dim W^1 - \bC_\vq\dim V^1),
    \vq^{-1}(\dim W^2 - \bC_\vq\dim V^2)
    \boldsymbol\rangle,
  \end{split}
\end{equation*}
we have
\begin{equation*}
   k = \ve(\vec{m}_1,\vec{m}_2),
\qquad
   e^{W^1} e^{V^1} = \prod Y_{i,a}^{m_i^1(a)},
\    
   e^{W^2} e^{V^2} = \prod Y_{i,a}^{m_i^2(a)}.
\end{equation*}
\end{NB}%

The following is the main result of \cite{VV2}.

\begin{Theorem}\label{thm:positive}
\textup{(1)}
  The structure constant of the product with respect to the base $\{
  L(W)\}$ is positive:
  \begin{equation*}
     L(W^1)\otimes L(W^2) \in \sum_W a^W_{W^1,W^2}(t) L(W)
  \end{equation*}
with $a^W_{W^1,W^2}(t)\in \Z_{\ge 0}[t,t^{-1}]$.

\textup{(2)} $\chi_{\vq,t}\colon\bfR_t\to \mathscr Y_t$ is an algebra
homomorphism with respect to $\otimes$ and the twisted product $\ast$.
\end{Theorem}

The following corollary of the positivity is also due to \cite{VV2}.

\begin{Corollary}
  The followings are equivalent:
  \begin{enumerate}
  \item $L(W^1)\otimes L(W^2) = L(W^1\oplus W^2)$ holds at $t=1$.
  \item $L(W^1)\otimes L(W^2) = t^{\varepsilon(W^1,W^2)} L(W^1\oplus
    W^2)$.
  \end{enumerate}
\end{Corollary}

It is tiresome to keep powers of $t$ when tensor products of simple
modules are simple. From this corollary, there is no loss of
information even if we forget powers.
Therefore we do not write $t^{\varepsilon(W^1,W^2)}$ hereafter.

The restriction functor defines an algebra homomorphism
\begin{multline*}
  H_*(\N(W)\times_{\N_0(W)}\N(W))\to
\\
H_*(\N(W^1)\times_{\N_0(W^1)}\N(W^1))\otimes
H_*(\N(W^2)\times_{\N_0(W^2)}\N(W^2)).
\end{multline*}
It gives us a monoidal structure on the un-graded version of $\mathscr
R_{\mathrm{conv}}$.
\begin{NB}
  The homomorphism is not compatible with the gradings.
\end{NB}

\section{Graded quiver varieties for the monoidal subcategory
  $\mathscr C_1$}\label{sec:C_1}

\subsection{Graded quiver varieties and the decorated quiver}
\label{subsec:deco}

The monoidal subcategory $\mathscr C_1$ introduced in \cite{HerLec} is, in
fact, the first (or second) of series of subcategories $\mathscr
C_\ell$ indexed by $\ell\in\Z_{\ge 0}$.
Let us describe all of them in terms of the category $\mathscr
R_{\mathrm{conv}}$.

We suppose that $(I,E)$ contains no odd cycles and take a {\it
  bipartite\/} partition $I = I_0\sqcup I_1$, i.e.\ every edge
connects a vertex in $I_0$ with one in $I_1$. We set
\begin{equation*}
  \xi_i =
  \begin{cases}
    0 & \text{if $i\in I_0$},
\\
    1 & \text{if $i\in I_1$}.
  \end{cases}
\end{equation*}
Fix a nonnegative integer $\ell$.
We consider the graded quiver varieties $\N(V,W)$, $\N_0(V,W)$ under the
following condition
\begin{equation}
   \text{$W_i(a) = 0$ unless $a = q^{\xi_i}$, $q^{\xi_i+2}$, \dots,
     $q^{\xi_i + 2\ell}$}.
\tag{$\ast_\ell$}
\end{equation}
It is clear that if $W$ satisfies $(\ast_\ell)$, both $W^1$ and $W^2$
satisfy $(\ast_\ell)$ in the convolution product $\res\colon \mathscr
Q_W\to \mathscr Q_{W^1}\times\mathscr Q_{W^2}$. 
Also from the proof of \propref{prop:ast_1}(1) below, it is clear that
$\Nreg(V,W)\neq\emptyset$ implies $V_i(a) = 0$ unless
$a=\vq^{\xi_i+1},\dots, \vq^{\xi_i+2\ell-1}$. Since $W^\perp_{i}(a)$
in \subsecref{subsec:slice} is the middle cohomology of the complex
\eqref{eq:taut_cpx_fixed}, $W^\perp_i(a)$ also satisfies
$(\ast_\ell)$.
Therefore the condition $(\ast_\ell)$ is also compatible with the
projective system $\mathcal K(\mathscr Q_W)\to \mathcal K(\mathscr
Q_{W^\perp})$.
Therefore we have the subring $\bfR_{t,\ell}$ of $\bfR_t$. We set
$\bfR_\ell = \left.\bfR_{t,\ell}\right|_{t=1}$. It is also clear that
the definition in \cite{HerLec} in terms of roots of Drinfeld polynomials
corresponds to our definition when $\g$ is of type $ADE$ from the
theory developed in \cite{Na-qaff}.

\begin{Example}
  Consider the simplest case $\ell = 0$.
  By \cite[4.2.2]{Na-qaff} or the argument below we have $\N_0(V,W) =
  \{ 0\}$ if $W$ satisfies $(\ast_0)$. Therefore $\mathscr Q_W$
  consists of finite direct sums of shifts of a single object
  $1_{\N_0(0,W)}$. We have $\res(1_{\N_0(0,W)}) =
  1_{\N_0(0,W^1)}\boxtimes 1_{\N_0(0,W^2)}$. This corresponds to the
  fact that any tensor product of simple modules in $\mathscr C_0$
  remains simple. (See \cite[3.3]{HerLec}.)
\end{Example}

We now start to analyze the condition $(\ast_{\ell=1})$. Let
\begin{equation}
  \label{eq:E_W}
  \bE_W \defeq
    \bigoplus_i \Hom(W_i(q^{\xi_i+2}),W_i(q^{\xi_i}))
    \oplus
    \bigoplus_{h : \vout(h)\in I_1, \vin(h)\in I_0}
    \Hom(W_{\vout(h)}(q^3), W_{\vin(h)}(1))
\end{equation}
This vector space $\bE_W$ is the space of representations of the {\it
  decorated quiver\/}.

\begin{Definition}\label{def:decorated}
  Suppose that a finite graph $\cG = (I,E)$ together with a bipartite
  partition $I = I_0\sqcup I_1$ is given. We define the {\it decorated
    quiver\/} $\widetilde \cQ = (\widetilde I, \widetilde\Omega)$ by
  the following two steps.

  (1) We put an orientation to each edge in $E$ so that vertexes in
  $I_0$ (resp.\ $I_1$) are sinks (resp.\ sources).
  Let $\Omega$ be the set of all oriented edges and $\cQ = (I,\Omega)$
  be the corresponding quiver.

  (2) Let $I_\fr$ be a copy of $I$. For $i\in I$, we denote by $i'$
  the corresponding vertex in $I_\fr$.
  Then we add a new vertex $i'$ and an arrow $i'\to i$ (resp.\ $i\to
  i'$) if $i\in I_0$ (resp.\ $i\in I_1$) for each $i\in I$.
  Let $\Omega_\dec$ be the set of these arrows. The decorated quiver
  is $\widetilde\cQ = (\widetilde I,\widetilde\Omega_\dec) = (I\sqcup
  I_\fr, \Omega\sqcup\Omega_\dec)$.

  We call $\cQ = (I,\Omega)$ the {\it principal part\/} of the
  decorated quiver.
\end{Definition}

For example, for type $A_3$ with $I_0 = \{1,3\}$, we get the following
quiver:
\begin{equation}\label{eq:decorated}
  \begin{CD}
    W_1(1) @<{\by_{1,2} = \beta_{1,\vq}B_{1,2,\vq^2}\alpha_{2,\vq^3}}<< W_2(\vq^3) 
    @>{\by_{3,2}=\beta_{3,\vq}B_{3,2,\vq^2}\alpha_{2,\vq^3}}>> W_3(1)
\\
    @A{\bx_1 = \beta_{1,\vq}\alpha_{1,\vq^2}}AA
    @VV{\bx_2 = \beta_{2,\vq^2}\alpha_{2,\vq^3}}V
    @AA{\bx_3 = \beta_{3,\vq}\alpha_{3,\vq^2}}A
\\
    W_1(\vq^2) @. W_2(\vq) @. W_3(\vq^2)
  \end{CD}
\end{equation}
The maps attached with arrows will soon be explained in the proof of
\propref{prop:ast_1}.

The following is a variant of a variety corresponding to a monomial in
$F_i$ in Lusztig's theory \cite[9.1.3]{Lu-book}.
\begin{Definition}
  \textup{(1)} Let $\nu = (\nu_i)\in \Z_{\ge 0}^I$. Let $\mathcal
  F(\nu,W)$ be the variety parametrizing collections of vector spaces
  $X = (X_i)_{i\in I}$ indexed by $I$ such that $\dim X_i = \nu_i$ and
\begin{equation*}
  X_i\subset W_i(1) \ (i\in I_0),
\qquad
  X_i\subset W_i(\vq)\oplus \bigoplus_{h\in\Omega:\vout(h)=i} X_{\vin(h)}
  \ (i\in I_1)
\end{equation*}
It is a kind of a partial flag variety and nonsingular projective.

\textup{(2)} Let $\Tilde{\mathcal F}(\nu,W)$ be the variety of all
triples $(\bigoplus \bx_i,\bigoplus \by_h,X)$ where $(\bigoplus
\bx_i,\bigoplus \by_h)\in \bE_W$ and $X\in\mathcal F(\nu,W)$
such that
\begin{equation*}
  \Ima \bx_i \subset X_i\ (i\in I_0),
\qquad
  \Ima \left(\bx_i \oplus \bigoplus_{h\in\Omega: \vout(h)=i} \by_h\right)
  \subset X_i\  (i\in I_1).
\end{equation*}
This is a vector bundle over $\mathcal F(\nu,W)$, and hence
nonsingular. Let $\pi_\nu\colon\Tilde{\mathcal F}(\nu,W)\to \bE_W$ be
the natural projection. It is a proper morphism.
\end{Definition}

\begin{Proposition}\label{prop:ast_1}
  Suppose $W$ satisfies $(\ast_\ell)$ with $\ell = 1$.

  \textup{(1)} If $\Nreg(V,W)\neq\emptyset$, we have
  \begin{equation}
    \label{eq:V1}
    \text{$V_i(a) = 0$ unless $a = \vq^{\xi_i+1}$.}
  \end{equation}
Moreover we have an isomorphism
\(
    \N_0(W) \cong \bE_W
\)
given by
\begin{equation*}
  [B,\alpha,\beta]
  \mapsto (\bigoplus_{i\in I} \bx_i ,
  \bigoplus_{h\in\Omega} \by_h);
  \qquad
  \bx_i = \beta_{i,\vq^{\xi_i+1}}\alpha_{i,\vq^{\xi_i+2}}, \quad
  \by_h =  \beta_{\vin(h),\vq} B_{h,\vq^2}\alpha_{\vout(h),\vq^3}.
\end{equation*}

\textup{(2)} Suppose that $V$ satisfies \eqref{eq:V1}.
Let us define $\nu\in\Z_{\ge 0}^I$ by $\nu_i = \dim V_i(\vq^{\xi_i+1})$.
Then $\N(V,W)$ is isomorphic to $\Tilde{\mathcal F}(\nu,W)$ and the
following diagram is commutative:
\begin{equation*}
  \begin{CD}
     \N(V,W) @>\cong>> \Tilde{\mathcal F}(\nu,W)
\\
     @V{\pi}VV  @VV{\pi_\nu}V
\\
     \N_0(W) @>\cong>> \bE_W
  \end{CD}
\end{equation*}
\end{Proposition}

\begin{NB}
  We assume $\bx_i \oplus \bigoplus_{h: \vout(h)=i} \by_h$ is injective
  and apply the reflection functor at the vertex $W_i(\vq^3)$
  for $i\in I_0$. We denote the resulting data by putting
  $\setbox5=\hbox{A}\overline{\rule{0mm}{\ht5}\hspace*{\wd5}}$.  We
  set 
\[
   \overline{W}_i(\vq^3) := 
   W_i(\vq)\oplus \bigoplus_h W_{\vin(h)}(1)/
   \Ima \left(\bx_i \oplus \bigoplus_{h: \vout(h)=i} \by_h\right)
\]
and define natural linear maps $\overline{\bx}_i\colon W_i(\vq)\to
\overline{W}_i(\vq^3)$, $\overline{\by}_{\overline{h}}\colon
W_{\vin(h)}(1)\to \overline{W}_i(\vq^3)$ as compositions of inclusions
and projections. From the last condition, we can view $X_i''$ as a subspace
of $\overline{W}_i(\vq^3)$.
\end{NB}

\begin{NB}
  If $V_i(a)$ is as in (2), the only relevant complexes are
  \begin{gather*}
   C_{i,\vq^3}^\bullet(V,W):
  0 = V_i(\vq^4)
\xrightarrow{\sigma_{i,a}}
  \displaystyle{\bigoplus_{h:\vin(h)=i}}
     V_{\vout(h)}(\vq^3)
    \oplus W_i(\vq^3)
    = W_i(\vq^3)
\xrightarrow{\tau_{i,a}}
  V_i(\vq^2)
  \qquad (i\in I_1),
\\
   C_{i,\vq^2}^\bullet(V,W):
  0 = V_i(\vq^3)
\xrightarrow{\sigma_{i,a}}
  \displaystyle{\bigoplus_{h:\vin(h)=i}}
     V_{\vout(h)}(\vq^2)
    \oplus W_i(\vq^2)
\xrightarrow{\tau_{i,a}}
  V_i(\vq)
  \qquad (i\in I_0),
\\
   C_{i,\vq}^\bullet(V,W):
  V_i(\vq^2)
\xrightarrow{\sigma_{i,a}}
  \displaystyle{\bigoplus_{h:\vin(h)=i}}
     V_{\vout(h)}(\vq)
    \oplus W_i(\vq)
\xrightarrow{\tau_{i,a}}
  V_i(1) = 0
  \qquad (i\in I_1),
\\
   C_{i,1}^\bullet(V,W):
   V_i(\vq)
\xrightarrow{\sigma_{i,a}}
  \displaystyle{\bigoplus_{h:\vin(h)=i}}
     V_{\vout(h)}(1)
    \oplus W_i(1)
    = W_i(1)
\xrightarrow{\tau_{i,a}}
  V_i(\vq^{-1}) = 0
  \qquad (i\in I_0).
  \end{gather*}
Therefore $(V,W)$ is {\it l\/}-dominant if and only if
\begin{gather*}
  \dim W_i(\vq^3)\ge \dim V_i(\vq^2)\ (i\in I_1),
\quad
  \sum_j a_{ij} \dim V_j(\vq^2) + \dim W_i(\vq^2) \ge \dim V_i(\vq)\ (i\in I_0),
\\
  \sum_j a_{ij} \dim V_j(\vq) + \dim W_i(\vq) \ge \dim V_i(\vq^2)\ (i\in I_1),
\quad
  \dim V_i(\vq) \le \dim W_i(1)\ (i\in I_0).
\end{gather*}
\end{NB}

\begin{NB}
  We give the formula of $d(V^1,W^1;V^2,W^2)$ under the assumption
($\ast_{\ell=1}$) and \eqref{eq:V1}:
\begin{equation*}
  \begin{split}
  & d(V^1,W^1;V^2,W^2) =
    \boldsymbol\langle \dim V^1,
  \vq^{-1}(\dim W^2 - \bC_\vq\dim V^2)\boldsymbol\rangle
  + \boldsymbol\langle \dim V^2, \vq\dim W^1\boldsymbol\rangle
\\
  =\; &
  \begin{gathered}[t]
  \sum_{i\in I_0} v^1_i(\vq) (w^2_i(1) - v^2_i(\vq))
  + \sum_{i\in I_1} v^1_i(\vq^2)\left(w^2_i(\vq) - v^2_i(\vq^2)
  + \sum_j a_{ij} v^2_j(\vq)\right)
\\
  + \sum_{i\in I} v^2_i(\vq^{\xi_i+1}) w^2_i(\vq^{\xi_i+2})
  \end{gathered}
  \end{split}
\end{equation*}
If we further assume $W_i(\vq^2) = 0 = W_j(\vq)$, we have
\begin{equation*}
  \sum_{i\in I_0} v^1_i(\vq) (w^2_i(1) - v^2_i(\vq))
  + \sum_{i\in I_1} v^1_i(\vq^2)\left(- v^2_i(\vq^2)
  + \sum_j a_{ij} v^2_j(\vq)\right)
  + \sum_{i\in I_1} v^2_i(\vq^{2}) w^2_i(\vq^{3})
\end{equation*}
\end{NB}

\begin{proof}
  (1) Recall that the coordinate ring of $\N_0(V,W)$ is generated by
  functions given by \eqref{eq:generator}.

Consider a map
  \begin{equation*}
    \beta_{j,a \vq^{-n-1}} B_{h_n,\vq^{-n}}
      \dots
      B_{h_1,a\vq^{-1}}\alpha_{i,a}
      \colon
    W_i(a) \to W_j(a q^{-n-2})
  \end{equation*}
  with $\vin(h_a) = \vout(h_{a+1})$ for $a=1,\dots, n-1$. From the
  assumption $(\ast_1)$, this is nonzero only when $i=j$, $n=0$, $a =
  q^{\xi_i+2}$ or $n=1$, $i\in I_1$, $j\in I_0$, $a = q^3$. From this
  observation we have
  \begin{equation*}
    \N_0(W) = \N_0(V,W),
  \end{equation*}
  for some $V$ with $V_i(a) = 0$ unless $a = q^{\xi_i+1}$. 
  Thus we obtain the first assertion.
  Moreover, the equation $\mu(B,\alpha,\beta) = 0$ is automatically
  satisfied, and the second assertion follows from a standard fact
  $\Hom(W,V)\oplus \Hom(V,W')\dslash \GL(V)\cong \Hom(W,W')$ for $V$
  with $\dim V \ge \min(\dim W,\dim W')$.
  \begin{NB}
    First consider the action by $\prod_{i\in I_0}\GL(V_i(2))$. Then
    from the above fact, the quotient is
    \begin{equation*}
      \begin{split}
      & \bigoplus_{i\in I_0} \Hom(W_i(3), W_i(1))
      \oplus \bigoplus_{\substack{i\in I_0 \\j\in I_1}}
      \Hom(W_i(3), V_j(1))^{\oplus a_{ij}} \oplus
      \bigoplus_{j\in I_1} \Hom(V_j(1), W_j(0))
      \oplus \Hom(W_j(2), V_j(1))
\\        
     =\; &
     \bigoplus_{i\in I_0} \Hom(W_i(3), W_i(1)) \oplus
     \bigoplus_{j\in I_1}      
      \Hom(W_j(2)\oplus \bigoplus_{i\in I_0} W_i(3)^{\oplus a_{ij}}, 
      V_j(1))
      \oplus \Hom(V_j(1), W_j(0)).
      \end{split}
    \end{equation*}
    Now we take the quotient by $\prod_{j\in I_1} \GL(V_j(1))$.
  \end{NB}

(2) We first observe the following:
\begin{Claim}
  Under the assumption $(B,\alpha,\beta)$ is stable if and only if the
  following linear maps are all injective:
  \begin{equation*}
    \beta_{i,\vq}\colon V_i(\vq)\to W_i(1) \ (i\in I_0),
    \qquad
    \sigma_{i,\vq} \colon V_i(\vq^2)
    \to \bigoplus_{h:\vout(h)=i} V_{\vin(h)}(\vq)\oplus W_i(\vq)
    \ (i\in I_1).
  \end{equation*}
(See \eqref{eq:taut_cpx_fixed} and the subsequent formula for the
definition of $\sigma_{i,\vq}$.)
\end{Claim}

Consider the $I\times\C^*$-graded vector space given by $V'_i(\vq) =
\Ker\beta_{i,\vq}$ and all other $V'_j(a) = 0$. Then the stability
condition implies $V'_i(\vq) = 0$. Therefore $\beta_{i,\vq}$ is injective.
The same argument shows the injectivity of $\sigma_{i,\vq}$.
Conversely suppose all the above maps are injective. Take an
$I\times \C^*$-graded subspace $V'$ of $V$ as in \defref{def:stable}.
First consider $V'_i(\vq)$ for $i\in I_0$. We have $\beta_{i,\vq}|
V'_i(\vq) = 0$. Therefore the injectivity of $\beta_{i,\vq}$ implies
$V'_i(\vq) = 0$. Next consider $V'_j(\vq^2)\subset V_j(\vq^2)$ for $j\in
I_0$. We have $\beta_{j,\vq^2}|V'_j(\vq^2) = 0$ from the assumption. We
also have $B_{\overline h,\vq^2}(V'_j(\vq^2))\subset V'_i(\vq) = 0$ from
what we have just proved. Therefore the injectivity of
$\sigma_{j,\vq}$ implies that $V'_j(\vq^2) = 0$. This completes the
proof of the claim.

Suppose $[B,\alpha,\beta]\in\N(V,W)$ is given. We set 
\begin{gather*}
   \widetilde\sigma_{i,\vq}
   := 
   \left(\bigoplus_{h: \vout(h)=i} \beta_{\vin(h),\vq}\oplus \id_{W_i(\bq)}
    \right)\circ \sigma_{i,\vq}\colon
    V_i(\vq^2)
    \to \bigoplus_{h:\vout(h)=i} W_{\vin(h)}(1)\oplus W_i(\vq),
\\  
   X_i :=\Ima\beta_{i,\vq}\ (i\in I_0),
\qquad
   X_i :=\Ima\widetilde\sigma_{i,\vq}\  (i\in I_1).
\end{gather*}
The spaces $X_i$ are independent of the choice of a representative
$(B,\alpha,\beta)$ of $[B,\alpha,\beta]$.
From the above claim, we have $\dim X_i = \dim V_i(\vq)$ ($i\in I_0$)
and $\dim X_i = \dim V_i(\vq^2)$ ($i\in I_1$). The remaining
properties are automatically satisfied by the construction.
\begin{NB}
  We have
  $\Ima \bx_i = \Ima (\beta_{i,\vq}\alpha_{i,\vq^2})
  \subset \Ima \beta_{i,\vq} = X_i$ for $i\in I_0$. We also have
  \begin{equation*}
    \begin{split}
      & \Ima\left(\bx_i \oplus \bigoplus_{h: \vout(h)=i} \by_h\right)
  = \Ima\left(\beta_{i,\vq^2}\alpha_{i,\vq^3}\oplus\bigoplus
  _{h: \vout(h)=i} \beta_{\vin(h),\vq} B_{h,\vq^2}\alpha_{i,\vq^3}\right)
\\
   \subset\; &
    \Ima\left(\beta_{i,\vq^2}\oplus\bigoplus
  _{h: \vout(h)=i} \beta_{\vin(h),\vq} B_{h,\vq^2}\right)
    = X_i.
    \end{split}
  \end{equation*}
\end{NB}%

Conversely suppose that $(\bigoplus \bx_i,\bigoplus \by_h, X)$ is
given. We set $V_i(\vq) := X_i$ ($i\in I_0$), $V_i(\vq^2) := X_i$
($i\in I_1$) and define linear maps $(B,\alpha,\beta)$ by
\begin{gather*}
   \beta_{i,\vq} := (\text{the inclusion $X_i\subset W_i(1)$}),
   \quad \alpha_{i,\vq^2} := \bx_i \ (i\in I_0),
\\
   \beta_{i,\vq^2} \oplus \bigoplus_{h:\vout(h)=i} B_{h,\vq^2}
   := \left(\text{the inclusion $X_i\subset W_i(\vq) \oplus\bigoplus
     X_{\vin(h)}$}\right),
   \quad \alpha_{i,\vq^3} := \bx_i \ (i\in I_1).
\end{gather*}
From the claim, the data $(B,\alpha,\beta)$ is stable and defines a
point in $\N(V,W)$. These two assignments are inverse to each other,
hence they are isomorphisms.
\end{proof}

\subsection{A contravariant functor $\sigma$}\label{subsec:sigma} 
For a later application we study the description in
\propref{prop:ast_1}(2) further.
By (2) $\N(V,W)\cong\Tilde{\mathcal F}(\nu,W)$ can be considered as a
vector bundle over $\mathcal F(\nu,W)$. It is naturally a subbundle of
the trivial bundle $\mathcal F(\nu,W)\times \bE_W$. Let
$\Tilde{\mathcal F}(\nu,W)^\perp$ be its annihilator in the dual
trivial bundle $\mathcal F(\nu,W)\times\bE_W^*$ and let $\pi^\perp\colon
\Tilde{\mathcal F}(\nu,W)^\perp\to\bE_W^*$ be the natural projection.
We denote the dual variables of $\bx_i$, $\by_h$ by $\bx_i^*$,
$\by^*_{\overline{h}}$ respectively, i.e.
\begin{equation*}
   \bx_i^*\in \Hom(W_i(\vq^{\xi_i}), W_i(\vq^{\xi_i+2})),
\qquad
   \by^*_{\overline{h}}\in \Hom(W_{\vin(h)}(1), W_{\vout(h)}(\vq^3)).
\end{equation*}
By (2) $((\bigoplus \bx_i^*,\bigoplus \by^*_{\overline{h}}),X)$ is
contained in $\Tilde{\mathcal F}(\nu,W)^\perp$ if and only if
\begin{equation}\label{eq:dual}
   \bx_i^*({X_i}) = 0 \ (i\in I_0),
\qquad
   \left(\bx_i^* + \sum_{h:\vout(h)=i} \by^*_{\overline{h}} \right)({X_i}) = 0
   \ (i\in I_1).
\end{equation}
\begin{NB}
  \begin{equation*}
  \begin{CD}
    W_1(1) @>{\by_{1,2}^{*}}>> W_2(\vq^3) 
    @<{\by_{3,2}^{*}}<< W_3(1)
\\
    @V{\bx_1^*}VV
    @AA{\bx_2^{*}}A
    @VV{\bx_3^*}V
\\
    W_1(\vq^2) @. W_2(\vq) @. W_3(\vq^2)
  \end{CD}
\end{equation*}
\end{NB}

It will be important to understand a fiber of $\pi^\perp$ on a general
point $(\bigoplus \bx^*_i,\bigoplus \by^*_{\overline{h}})$ in
$\bE_W^*$.
Since considering a subspace $X_i$ in $W_i(\vq)\oplus W_{\vin(h)}(1)$
looks slightly strange, let us apply the Bernstein-Gelfand-Ponomarev
reflection functors \cite{BGP} (see \cite[VII.5]{ASS}) to $(\bigoplus
\bx^*_i,\bigoplus \by^*_{\overline{h}})$ at all the vertexes $i\in I_1$
(where $W_i(\vq^3)$ is put).
First observe that $(\pi^\perp)^{-1}(\bigoplus \bx^*_i,\bigoplus
\by^*_{\overline{h}})$ is unchanged even if we replace $W_i(\vq^3)$ by
the image of the map
\begin{equation}\label{eq:surj}
     \bx_i^* + \sum_{h:\vout(h)=i} \by^*_{\overline{h}}
     \colon W_i(\vq) \oplus\bigoplus_{h:\vout(h)=i} W_{\vin(h)}(1)
     \to W_i(\vq^3)
\end{equation}
for all $i\in I_1$.
\begin{NB}
  We have
  \begin{equation*}
    \dim \Ima(\bx_i^* + \sum_{h:\vout(h)=i} \by^*_{\overline{h}})
    = \min
    \left(\dim W_i(\vq^3), \dim W_i(\vq)+\sum_{h:\vout(h)=i} \dim W_{\vin(h)}(1)
      \right).
  \end{equation*}
\end{NB}%
Then we may assume $\bx_i^* + \sum_{h:\vout(h)=i}
\by^*_{\overline{h}}$ is surjective. Then we can go back to
$(\bigoplus \bx_i^*, \bigoplus \by^*_{\overline{h}})$ by the inverse
reflection functor. Hence the following operation gives an isomorphism
between the relevant varieties.

We set
\begin{equation*}
   \lsp{\sigma}W_i(\vq^3) \defeq 
   \Ker \left(\bx_i^* + \sum_{h:\vout(h)=i} \by^*_{\overline{h}}\right),
\end{equation*}
and define linear maps
\(
   \lsp{\sigma}\bx_i\colon \lsp{\sigma}W_i(\vq^3)\to W_i(\vq)
\)
($i\in I_1$),
\(
   \lsp{\sigma}\by_h\colon \lsp{\sigma}W_i(\vq^3)\to W_{\vin(h)}(1)
\)
($h\in H$ with $\vout(h) = i\in I_1$) as the compositions of the
inclusion 
\(
  \lsp{\sigma}W_i(\vq^3)\to
  W_i(\vq) \oplus\bigoplus_{h:\vout(h)=i} W_{\vin(h)}(1)
\)
and the projections to factors.
We have
\begin{NB}
\begin{equation*}
   \dim \lsp{\sigma}W_i(\vq^3) = 
   \dim W_i(\vq)+\sum_{h:\vout(h)=i} \dim W_{\vin(h)}(1) -
   \min
    \left(\dim W_i(\vq^3), \dim W_i(\vq)+\sum_{h:\vout(h)=i} \dim W_{\vin(h)}(1)
      \right)   
\end{equation*}
\end{NB}
\begin{equation}\label{eq:sigmaW}
   \dim \lsp{\sigma}W_i(\vq^3) = \max\left(
        \dim W_i(\vq)+\sum_{h:\vout(h)=i} \dim W_{\vin(h)}(1) - \dim W_i(\vq^3),
        0\right).
\end{equation}
We denote by $\lsp{\sigma}{W}$ the new $I\times\C^*$-graded vector
space given obtained from $W$ by replacing $W_i(\vq^3)$ by
$\lsp{\sigma}W_i(\vq^3)$ for all $i\in I_1$. We also 
set $\lsp{\sigma}\bx_i = \bx_i^*$ for $i\in I_0$.
We do not change $W_i(1)$, $W_i(\vq^2)$ for $i\in I_0$ and $W_i(\vq)$
for $i\in I_1$.

We consider $X_i$ ($i\in I_1$) as a subspace of
$\lsp{\sigma}{W}_i(\vq^3)$ thanks to the second equation of
\eqref{eq:dual}.
Since $X_i$ was originally a subspace of
$W_i(q)\oplus\bigoplus_{h:\vout(h)=i} X_{\vin(h)}$, the above
definition implies $\lsp{\sigma}\by_h(X_{\vout(h)})\subset
X_{\vin(h)}$.

For the convenience we change the notation for a subspace from $X_i$
to $X_i(1)$ ($i\in I_0$) or $X_i(\vq^3)$ ($i\in I_1$) to indicate the
$\C^*$-grading. We also set $X_i(\vq^2) = 0$ for $i\in I_0$ and
$X_i(\vq) = \lsp{\sigma}{W}_i(\vq)$ for $i\in I_1$. Under these
definition, $\lsp{\sigma}\bx_i(X_i(1))\subset X_i(\vq^2)$ is nothing
but the first equation in \eqref{eq:dual}, and
$\lsp{\sigma}\bx_i(X_i(\vq^3))\subset X_i(\vq)$ is automatically true.
Thus the conditions can be phrased simply as `$X$ is invariant under
$(\bigoplus_{i\in I} \lsp{\sigma}\bx_i, \bigoplus_{h\in\Omega}
\lsp{\sigma}\by_h)$'.

\begin{Lemma}\label{lem:reflect}
Let $\lsp{\sigma}\bx_i$, $\lsp{\sigma}\by_h$ be as above.
Then $(\pi^\perp)^{-1}(\bigoplus \bx^*_i,\bigoplus
\by^*_{\overline{h}})$ is isomorphic to the variety of
$I\times\C^*$-graded subspaces $X$ of $\lsp{\sigma}{W}$ satisfying
\begin{gather*}
  X_i(\vq^2) = 0\ (i\in I_0), \qquad X_i(\vq) = \lsp{\sigma}{W}_i(\vq)
  \ (i\in I_1),
  \\
  \dim X_i(1) = \dim V_i(\vq)\ (i\in I_0), \qquad \dim X_i(\vq^3) =
  \dim V_i(\vq^2)\ (i\in I_1),
  \\
  \text{$X$ is invariant under $(\bigoplus_{i\in I} \lsp{\sigma}\bx_i,
    \bigoplus_{h\in\Omega} \lsp{\sigma}\by_h)$}.
\end{gather*}
\end{Lemma}

This variety is what people call the {\it quiver Grassmannian\/}
associated with the quiver representation $(\bigoplus_{i}
\lsp{\sigma}\bx_i, \bigoplus_{h\in\Omega} \lsp{\sigma}\by_h)$. Its
importance in the cluster algebra theory was first noticed in
\cite{CalderoChapoton}.
We will be interested only in its Poincar\'e polynomial, which is
independent of the choice of a general point, we denote this variety
simply by $\Gr_{V}(\lsp{\sigma}{W})$, suppressing the choice
$(\bigoplus_{i} \lsp{\sigma}\bx_i, \bigoplus_{h\in\Omega}
\lsp{\sigma}\by_h)$.
\begin{NB}
  I am not sure that the variety itself is independent of the choice.
\end{NB}%
Note also that the $I$-grading is only relevant in
$\Gr_V(\lsp{\sigma}W)$. Therefore we use this notation also for
an $I$-graded vector space $V$.

Note that the orientation is different from the decorated quiver
\eqref{eq:decorated}.
This corresponds to the cluster algebra with principal coefficients
considered in \subsecref{subsec:F-pol}. Therefore we call
it the quiver with {\it principal decoration}.
For example, in type $A_3$ with $I_0 = \{1,3\}$, we get the following
quiver:
\begin{equation}\label{eq:ori'}
  \begin{CD}
    W_1(1) @<{\lsp{\sigma}\by_{1,2}}<< \lsp{\sigma}W_2(\vq^3) 
    @>{\lsp{\sigma}\by_{3,2}}>> W_3(1)
\\
    @V{\lsp{\sigma}\bx_1=\bx_1^*}VV
    @VV{\lsp{\sigma}\bx_2}V
    @VV{\lsp{\sigma}\bx_3=\bx_3^*}V
\\
    W_1(\vq^2) @. W_2(\vq) @. W_3(\vq^2)
  \end{CD}
\end{equation}

\begin{Remark}
  The quiver Grassmannian is a fiber of a projective morphism, which
  played a fundamental role in Lusztig's construction of the canonical
  base. It is denoted by $\pi_{\boldsymbol\nu}\colon\Tilde{\mathcal
    F}_{\boldsymbol\nu}\to \bE_{\mathbf V}$ in
  \cite[Part~II]{Lu-book}. But note that Lusztig considered more
  generally various spaces of {\it flags\/} not only subspaces.
\end{Remark}

Later it will be useful to view $\sigma$ as a functor between category
of representations of quivers.
Let $\rep\widetilde\cQ$ be the category of finite dimensional
representations of the decorated quiver $\widetilde\cQ$.
Let $\lsp{\sigma}{\!\widetilde\cQ}$ be the quiver with the principal
decoration obtained by reversing the arrows between $i$ and $i'$ for
$i\in I_0$ as above. Let $\rep\lsp{\sigma}{\!\widetilde\cQ}$ be the
corresponding category and
$\rep\lsp{\sigma}{\!\widetilde\cQ^{\mathrm{op}}}$ be its opposite
category. Then $\sigma$ is the functor
\begin{equation*}
  \lsp{\sigma}(\bullet) = \prod_{i\in I_1} \Phi_i^- \circ D(\bullet)\colon 
  \rep\widetilde\cQ \to \rep\lsp{\sigma}{\!\widetilde\cQ^{\mathrm{op}}},
\end{equation*}
where $\Phi_i^-$ is the reflection functor at the vertex for
$W_i(\vq^3)$ and $D$ is the duality operator
\begin{equation*}
   D(\bullet) = \Hom_\C(\bullet,\C).
\end{equation*}
In order to make an identification with the above picture, we fix an
isomorphism $W\cong W^*$ of $(I\sqcup I_\fr)$-graded vector spaces.
\begin{NB}
  But it is for the Fourier transform, so the pairing between $\bE_W$
  and $\bE_W^*$ is important.
\end{NB}

Let $\rep^-\widetilde\cQ$ be the full subcategory of
$\rep\widetilde\cQ$ consisting of representations having no direct
summands isomorphic to simple modules corresponding to vertexes $i\in
I_1$. Similarly we define
$\rep^-\lsp{\sigma}{\!\widetilde\cQ^{\mathrm{op}}}$.  Then $\sigma$
defines an equivalence between $\rep^-\widetilde\cQ$ and
$\rep^-\lsp{\sigma}{\!\widetilde\cQ^{\mathrm{op}}}$.
We write the quasi-inverse functor $\sigma_- = D\circ \prod_{i\in
  I_1}\Phi^+_i$.

In fact, it is more elegant to consider $\sigma$ as a functor between
derived categories of $\rep\widetilde\cQ$ and
$\rep\lsp{\sigma}{\!\widetilde\cQ^{\mathrm{op}}}$ as in
\cite[IV.4.Ex.~6]{GM}. See also \remref{rem:derived}.

\section{From Grothendieck rings to cluster algebras}\label{sec:hom}

Since $W$ always satisfies ($\ast_{\ell=1}$) hereafter, we denote
$W_i(\vq^{3\xi_i})$ and $W_i(\vq^{2-\xi_i})$ by $W_i$ and $W_{i'}$
respectively. This is compatible with the notation in
\defref{def:decorated} as $W_i(\vq^{2-\xi_i})$ is on the new vertex
$i'$.

We denote the simple modules of the decorated quiver by $S_i$, $S_{i'}$
corresponding to vertexes $i\in I$, $i'\in I_\fr$.
We will consider modules of two completely different algebras,
(a) modules in $\mathscr R_{\mathrm{conv}}$ (or of $\Ulq$) and
(b) modules of the decorated quiver.
Simple modules for the former will be denoted by $L(W)$, while $S_i$,
$S_{i'}$ for the latter. We hope there will be no confusion. We denote
the underlying $\widetilde I = (I\sqcup I_\fr)$-graded vector space of
$S_i$, $S_{i'}$ also by the same letter.

The Grothendieck ring $\bfR_\ell$ is a polynomial ring in the classes
$L(W)$ with $\dim W = 1$ satisfying $(\ast_\ell)$ ({\it
  l\/}-fundamental representations in $\mathscr C_\ell$ when $\g$ is
of type $ADE$). This result was proved as a consequence of the theory
of $q$-characters in \cite[Prop.~3.2]{HerLec} for $\g$ of type
$ADE$. Since $q$-characters make sense for arbitrary $\g$, the same
argument works.
The corresponding result for the whole category $\mathscr R$ is
well-known.

For $\bfR_{\ell=1}$, we have $2\# I$ variables corresponding to {\it
  l\/}-fundamental representations. We denote them by $x_i$ and
$x_{i}'$ exchanging $i$ and $i'$ from the index of the decorated
quiver (\defref{def:decorated}):
\begin{equation}\label{eq:name}
   x_i = L(W) \longleftrightarrow 
   W = S_{i'},
\qquad
   x_i' = L(W) \longleftrightarrow 
   W = S_{i}.
\end{equation}
This is confusing, but we cannot avoid it to get a correct statement.
\begin{NB}
\begin{equation*}
   x_i \longleftrightarrow 
   L_i(\vq^{2-\xi_i}) =
   \begin{cases}
   L_i(\vq^2) & \text{if $i\in I_0$},
\\
   L_i(\vq) & \text{if $i\in I_1$},
   \end{cases}
\qquad
   x_i' \longleftrightarrow 
   L_i(\vq^{3\xi_i}) = 
   \begin{cases}
   L_i(1) & \text{if $i\in I_0$},
\\
   L_i(\vq^3) & \text{if $i\in I_1$},
   \end{cases}
\end{equation*}
where $L_i(\vq^n)$ is the class $L(W)\in\bfR_t$ corresponding to the
$1$-dimensional $I\times\C^*$-graded vector space $W$ whose nonzero
entry is $W_i(\vq^n)$. In the isomorphism $\N_0(W)\cong \bE_W$ in
\propref{prop:ast_1} it corresponds to a simple module of the
decorated quiver, putting the $1$-dimensional vector space on the
vertex corresponding to $W_i(\vq^n)$ and $0$ on the other vertexes.
Note that $x_i$ (resp.\ $x_i'$) corresponds to the new (resp.\ old)
vertex $i'$ (resp.\ $i$).
\end{NB}

We denote the class of the Kirillov-Reshetikhin module in $\mathscr
C_1$ by $f_i$. It corresponds to the class $L(W)$, where $W$ is a
$2$-dimensional $\widetilde I$-graded vector space with $\dim W_i =
\dim W_{i'} = 1$, and $0$ at other gradings.
We have
\begin{equation}\label{eq:T-system}
  f_i = x_i x'_i - \prod_{h\in H:\vout(h)=i} x_{\vin(h)}.
\end{equation}
This is an example of the $T$-system proved in \cite{MR1993360}, but
in fact, easy to check by studying the convolution diagram as
$\bE_W\cong\C$ has only two strata, the origin and the complement.
It is also a simple consequence of \thmref{thm:main} below.
It is a good exercise for the reader.

\begin{Remark}\label{rem:T-system}
  In \cite{MR1993360} a more precise relation at the level of modules,
  not only in the Grothendieck group, was shown: for $i\in I_0$, there
  exists a short exact sequence
  \begin{equation*}
    0 \to \bigotimes_{h\in H:\vout(h)=i} x_{\vin(h)} \to 
    x_i'\otimes x_i \to f_i \to 0
  \end{equation*}
  and we replace the middle term by $x_i\otimes x_i'$ if $i\in I_1$.
\end{Remark}

\begin{NB}
\begin{Remark}
  In \cite[\S4]{HerLec} the modules corresponding to $x_i$, $x_i'$, $f_i$
  are denoted by $S[-\alpha_i]$, $S[\alpha_i]$, $F_i$ respectively.
\end{Remark}
\end{NB}

We have an algebra embedding
\begin{equation*}
  \bfR_{\ell=1} = \Z[x_i,x_i']_{i\in I} \to {\mathscr F} = \Q(x_i,f_i)_{i\in I}.
\end{equation*}
We now put the cluster algebra structure on the right hand side.
It is enough to specify the initial seed. We take $x_i$, $f_i$ as
cluster variables of the initial seed. 
We make $f_i$ as a frozen variable.
We call the quiver for the initial seed the {\it $\mathbf
  x$-quiver}. It looks almost the same as the decorated quiver in
\defref{def:decorated}, but is a little different, and is given as
follows.

\begin{Definition}\label{def:x-quiver}
  Suppose that a finite graph $\cG = (I,E)$ together with a bipartite
  partition $I = I_0\sqcup I_1$ is given. We define the {\it
    $\bx$-quiver\/} $\widetilde\cQ_{\bx} = (\widetilde
  I,\widetilde\Omega_\bx)$ by the following two steps.

  (1) The underlying graph is the same as one of the decorated quiver:
  $\cG = (I\sqcup I_\fr, E \sqcup \bigcup \{ i
  \nolinebreak[4]-\nolinebreak[4]i'\})$. The variable $x_i$
  corresponds to the vertex $i$ in the original quiver, while $f_i$
  corresponds to the new vertex $i'$.

  (2) The rule for drawing arrows is
\begin{equation}\label{eq:rule}
   f_i\to x_i\ (i\in I_0), \qquad
   x_i\to f_i\ (i\in I_1), \qquad
   x_{\vout(h)} \xrightarrow{h} x_{\vin(h)} 
   \ (\text{if $\vout(h)\in I_0$, $\vin(h)\in I_1$}).
\end{equation}
\end{Definition}

For our favorite example, $A_3$ with $I_0 = \{1,3\}$, we get the
following quiver.
\begin{equation*}
  \begin{CD}
    x_1 @>>> x_2 @<<< x_3
\\
    @AAA
    @VVV
    @AAA
\\
    f_1 @. f_2 @. f_3
  \end{CD}
\end{equation*}
Note that the orientation differs from the decorated quiver
\eqref{eq:decorated} and the principal decoration
\eqref{eq:ori'}. Also the vertex $f_i$ corresponds to $W_{i'}$, and
$x_i$ corresponds to $W_i$. This is different from the identification
\eqref{eq:name}.
If we look at the principal part, the orientation is reversed.

If we make a mutation in direction $x_i$, the new variable given
by the exchange relation \eqref{eq:exchange} is nothing but
\begin{equation*}
  x_i' = \frac{f_i + \prod_{h\in H:\vout(h)=i} x_{\vin(h)}}{x_i}
\end{equation*}
from \eqref{eq:T-system}. Note the exchange relation is correct for
the $\bx$-quiver given by our rule \eqref{eq:rule}, but {\it wrong\/}
for the decorated quiver. Thus this confusion cannot be avoided.

We thus have
\begin{Proposition}
  The Grothendieck ring $\bfR_{\ell=1}$ is a subalgebra of the cluster
  algebra ${\mathscr A}(\widetilde{\bB})$.
\end{Proposition}

The argument in \cite[4.4]{HerLec} (based on \cite[1.21]{FomZel3})
implies that $\bfR_{\ell=1}\cong{\mathscr A}(\widetilde{\bB})$, but we
will see that all cluster monomials come from simple modules in
$\bfR_{\ell=1}$, so we have a different proof later.

We also need the seed obtained by applying the sequence of mutations
$\prod_{i\in I_1} \mu_i$. (See \cite[\S7.1]{HerLec}.)
Then (1) $x_i$ ($i\in I_1$) is replaced by
$x'_i$, (2) the orientation of arrows are reversed in the principal part
and $i\to i'$ ($i\in I_1$), and (3) add $a_{ij}$ arrows from $i$ to $j'$.
In our $A_3$ example, we obtain
\begin{equation}\label{eq:z-quiver}
 \xymatrix{
  x_1 \ar[rd] & \ar[l] x_2' \ar[r] & \ar[ld] x_3
\\
  \ar[u] f_1 & \ar[u] f_2 & f_3 \ar[u]
}
\end{equation}
We set
\begin{equation}\label{eq:z}
  z_i \defeq 
  \begin{cases}
    x_i & \text{if $i\in I_0$},
\\
    x_i'  & \text{if $i\in I_1$}.
  \end{cases}
\end{equation}
We call above one the {\it $\mathbf z$-quiver}.

\section{Cluster character and prime factorizations of simple modules}
\label{sec:prime}

\subsection{An almost simple module}

Fix an $\widetilde I$-graded vector space $W$. Let $\Psi$ be the
Fourier-Sato-Deligne functor for the vector space $\bE_W\cong\N_0(W)$
(\cite{KaSha,Lau}).
We define a subset $\mathscr L_W\subset\mathscr P_W$ by
\begin{equation*}
   L\in \mathscr L_W \Longleftrightarrow
   \text{the support of $\Psi(L)$ is the whole space $\bE_W^*$}.
\end{equation*}
If $L\in \mathscr L_W$, $\Psi(L)$ is an IC complex associated with a
local system defined over an open set in $\bE_W$. We denote
its rank by $r_W(L)\in\Z_{> 0}$.

Since the Fourier transform of $IC_W(0) = 1_{\{0\}}$ is
$1_{\bE_W^*}[\dim \bE_W^*]$, we always have $IC_W(0)\in\mathscr L_W$.
We have $r_W(IC_W(0)) = 1$.

We extend this definition for a condition on simple modules $L(W')$.
Recall that $IC_W(V)$ is identified with $IC_{W^\perp}(0)$ such that
$\dim W^\perp = \dim W-\bC_\vq \dim V$. We say $L(W')\in\mathscr L_W$
if $IC_W(V)\in\mathscr L_W$ with $W' = W^\perp$. We similarly define
$r_W(L(W'))$.

We define the {\it almost simple module\/} associated with $W$ by
\begin{equation*}
   {\mathbb L}(W) = \sum_{L(W')\in \mathscr L_W} r_W(L(W')) L(W').
\end{equation*}
This is an element in $\bfR_t$.

From the definition of $L(W')\in\mathscr L_W$ we have $W'\le
W$. Therefore almost simple modules $\{ {\mathbb L}(W) \}$ form a
basis of $\bfR_t$ such that the transition matrix between it and $\{
L(W)\}$ is upper triangular with diagonal entries $1$.

We will see that an almost simple module is not necessarily simple
later. There will be also a simple sufficient condition guaranteeing
an almost simple module is simple.

\begin{Remark}
  As we will see soon, almost simple modules are given in terms of
  quiver Grassmannian for a general representation of $\bE_W^*$. This,
  at first sight, looks similar to the set of generic variables
  considered by Dupont \cite{Dupont}. (See also \cite{DXX}.)
  But there is a crucial difference. We consider the total sum of
  Betti numbers of the quiver Grassmannian, while Dupont consider
  Euler numbers.
  There is an example with nontrivial odd degree cohomology groups
  \cite[Ex.~3.5]{DWZ2}, so this is really different.

  Note that from the representation theory of $\Ulq$, it is natural to
  specialize as $t=1$, since $t$-analog becomes the ordinary
  ${\vq}$-character (and the positivity is preserved). This difference
  cannot be seen for cluster monomials, thanks to
  \remref{rem:negative}.
  \begin{NB}
    I need to consider $\chi_{\vq,t=-1}({\mathbb L}(W))_{\le 2}
    = \chi_{\vq,t=1}({\mathbb L}'(W))_{\le 2})$ for some virtual 
    module ${\mathbb L}'(W)$.
  \end{NB}%

  In fact, we can also consider a specialization at $t=-1$, but then
  the positivity is lost and the proof of the factorization
  (\propref{prop:Schur}) breaks.
\end{Remark}

\subsection{Truncated $q$-character}

In \cite[\S6]{HerLec} Hernandez-Leclerc introduced the {\it truncated
  $\vq$-cha\-ra\-cter\/} $\chi_\vq(M)_{\le 2}$ from the ordinary
$q$-character $\chi_\vq(M)$ by setting variables $V_{i,\vq^r} = 0$ for
$r \ge 3$. From the geometric definition of the $q$-character reviewed
in \subsecref{subsec:q-char}, it just means that we only consider
nonsingular quiver varieties $\N(V,W)$ satisfying \eqref{eq:V1},
i.e.~those studied in \propref{prop:ast_1}(2). In particular, its
$t$-analog also makes sense:
\begin{equation}
  \label{eq:trunc}
  \begin{split}
  & \chi_{\vq,t}(M(W))_{\le 2} \defeq \sum_{\text{$V$ satisfies \eqref{eq:V1}}}
  \sum_k t^{-k} \dim H^k(i_0^! \pi_W({V})) e^W e^V,
\\
  & \chi_{\vq,t}(L(W))_{\le 2} =  \sum_{\text{$V$ satisfies \eqref{eq:V1}}}
     a_{V,0;W}(t) e^W e^V,
  \end{split}
\end{equation}
where $a_{V,0;W}(t)$ is the coefficient of $IC_{W}(0)=1_{\{0\}}$ in
$\pi_{W}(V)$ in $\mathcal K(\mathscr Q_W)$.
Since $V$ satisfies \eqref{eq:V1} if $(V,W)$ is {\it l\/}-dominant,
the truncated $q$-character still embed $\bfR_{\ell=1}$ to $\mathscr
Y_t$. (See \cite[Prop.~6.1]{HerLec} for an algebraic proof.)

The following is one of main results in this paper.

\begin{Theorem}\label{thm:main}
  Suppose $W$ satisfies $(\ast_\ell)$ with $\ell = 1$. Then the
  truncated $t$-analog of $q$-character of an almost simple module is
  given by
  \begin{equation*}
    \chi_{\vq,t}(\mathbb L(W))_{\le 2} 
    = \sum_V P_t(\Gr_{V}(\lsp{\sigma}{W})) e^W e^V, 
  \end{equation*}
  where the summation runs over all $I\times\C^*$-graded vector spaces
  $V$ with \eqref{eq:V1} and $P_t(\ )$ is the normalized Poincar\'e
  polynomial for the Borel-Moore homology group
  \begin{equation*}
     P_t(\Gr_{V}(\lsp{\sigma}{W}))
     = \sum_i t^{i-\dim \N(V,W)} \dim H_i(\Gr_{V}(\lsp{\sigma}{W})).
  \end{equation*}
\end{Theorem}

\begin{NB}
The following is wrong ! It is only true for standard modules.

Note that the virtual Poincar\'e polynomial and the actual Poincar\'e
polynomial coincide as there are no odd cohomology groups by
\cite[\S5]{MR2144973}.
\end{NB}

Since $\Gr_V(\lsp{\sigma}W)$ is a fiber of $\pi^\perp\colon
\Tilde{\mathscr F}(\nu,W)^\perp\to \bE_W^*$ over a general point in
$\bE_W^*$ and $\Tilde{\mathscr F}(\nu,W)^\perp$ is nonsingular,
$\Gr_V(\lsp{\sigma}W)$ is nonsingular by the generic smoothness
theorem. Since $\pi^\perp$ is projective, it is also projective.
Therefore the Poincar\'e polynomial is essentially equal to the
virtual one defined by Danilov-Khovanskii \cite{DanKho} using a mixed
Hodge structure of Deligne \cite{Deligne3}:
\begin{equation*}
  P^{\mathrm{vert}}_t(X)
  \defeq
    \sum_k (-1)^k t^{p+q}
    h^{p,q}(H^k_c(X)).
\end{equation*}
(See \cite{DanKho} for the notation $h^{p,q}(H^k_c(X))$.)
Since our Poincar\'e polynomial is normalized, we have
\begin{equation*}
  P_t(\Gr_V(\lsp{\sigma}W)) =
  t^{-\dim\Nreg(V,W)}P^{\mathrm{vert}}_{-t}(\Gr_V(\lsp{\sigma}W)).
\end{equation*}

\begin{Remark}\label{rem:E8}
  Recall that $\chi_{\vq,t}(L(W))$ was computed in \cite{MR2144973}.
  More precisely, a purely combinatorial algorithm to compute
  $\chi_{\vq,t}(L(W))$ was given in [loc.\ cit.].
  If we are interested in simple modules in $\mathscr C_1$, the same
  algorithm works by replacing every `$\chi_{\vq,t}(\ )$' in [loc.\
  cit.] by '$\chi_{\vq,t}(\ )_{\le 2}$'.
  Thus the computation is drastically simplified.
  The algorithm consists of 3 steps. The first step is the computation
  of $\chi_{\vq,t}$ for {\it l\/}-fundamental representation. The
  actual computation of $\chi_{\vq,t}$ was performed by a
  supercomputer \cite{Na:E8}. But this is certainly unnecessary for
  $\chi_{\vq,t}(\ )_{\le 2}$. The second step is the computation of
  $\chi_{\vq,t}$ for the standard modules. This is just a twisted
  multiplication of $\chi_{\vq,t}$'s given in the first step. This
  step is simple. The third step is analog of the definition of
  Kazhdan-Lusztig polynomials. It is still hard computation if we take
  large $W$. It is probably interesting to compare this algorithm with
  one given by the mutation, e.g., for $W$ corresponding to the
  highest root of $E_8$. In this case we have ${\mathbb L}(W) = L(W)$
  as we will see soon in \propref{lem:real}.

  In general, if ${\mathbb L}(W) \neq L(W)$, we need to compute
  $r_W(L(W'))$.
\end{Remark}

\begin{Example}\label{ex:KR}
  For the Kirillov-Reshetikhin module $f_i$, we have $\dim W_{i} = 1 =
  \dim W_{i'}$.  If $i\in I_1$, we have $\lsp{\sigma}W_i = 0$. Therefore
  $\Gr_V(\lsp{\sigma}{W})$ is a point if $V=0$ and $\emptyset$
  otherwise.
  If $i\in I_0$, a general $\lsp{\sigma}\bx_i^{*}\colon W_i\to W_{i'}$
  is an isomorphism. Therefore $\Gr_V(\lsp{\sigma}{W})$ is again a
  point if $V=0$ and $\emptyset$ otherwise.
  Thus we must have ${\mathbb L}(W) = L(W)$ in this case, and
  $\chi_{q,t}(f_i)_{\le 2}$ contains only the first term:
\begin{equation*}
  \chi_{q,t}(f_i)_{\le 2} = Y_{i,\vq^{\xi_i}} Y_{i,\vq^{\xi_i+2}}.
\end{equation*}
This can be shown in many ways, say using the main result of
\cite{MR1993360}.

Next consider $x_i$. If $i\in I_0$, then $\lsp{\sigma}{W}$ is
$1$-dimensional with nonzero entry at $\lsp{\sigma}{W}_{i'}$.  But
since we can put only $0$-dimensional space $X_{i'}$, we only allow
$V = 0$.
Thus ${\mathbb L}(W) = L(W)$ and $\chi_{\vq,t}(x_i)_{\le 2} = Y_{i,\vq^2}$.

If $i\in I_1$, then $\lsp{\sigma}{W}$ is $2$-dimensional with nonzero
entries at $\lsp{\sigma}{W}_i$ and $\lsp{\sigma}{W}_{i'}$. Therefore
we have either $V=0$ or $1$-dimensional $V$ with nonzero entry at
$V_{i'}$. The corresponding varieties are a single point in both
cases. Thus ${\mathbb L}(W) = L(W)$ and $\chi_{\vq,t}(x_i)_{\le 2} =
Y_{i,\vq}(1 + V_{i,\vq^2})$.

Similarly we can compute $x_i'$. We have ${\mathbb L}(W) = L(W)$
always and the $\vq$-character is
\begin{equation*}
  \chi_{\vq,t=1}(x_i')_{\le 2} =
  \begin{cases}
    Y_{i,1}(1 + V_{i,\vq} \prod_{j}(1 + V_{j,\vq^2})^{a_{ij}})
    & \text{if $i\in I_0$},
\\
    Y_{i,\vq^3} & \text{if $i\in I_1$}.
  \end{cases}
\end{equation*}
This gives an answer to the exercise we mentioned after
\eqref{eq:T-system}.
\begin{NB}
  Recall $x_i'$ corresponds to $S_i$.  First suppose $i\in I_1$. Then
  $\lsp{\sigma}S_i = 0$. Therefore
  \begin{equation*}
    \chi_{\vq,t}(x_i')_{\le 2} = Y_{i,\vq^3}.
  \end{equation*}
Hence
\begin{equation*}
   \chi_{\vq,t=1}(x_i')_{\le 2} \chi_{\vq,t=1}(x_i)_{\le 2} 
   = Y_{i,\vq}Y_{i,\vq^3}(1 + V_{i,\vq^2})
   = Y_{i,\vq} Y_{i,\vq^2} + \prod_{j} Y_{j,\vq^2}^{a_{ij}}
   = \chi_{\vq,t=1}(f_i)_{\le 2} + \chi_{\vq,t=1}(\bigotimes_j x_j^{a_{ij}})_{\le 2}.
\end{equation*}

Next suppose that $i\in I_0$. Then
  $\lsp{\sigma}S_i = S_i \oplus \bigoplus_{h:\vin(h)=i}
  S_{\vout(h)}$. A general representation $\lsp{\sigma}\by_h^*$ is surjective
  to $S_i$, so the only possible subspaces $X$ are
  \begin{itemize}
  \item $0$ for all vertexes,
  \item $X_i = \C$, $X_j = \C^n$ with $n=0,\dots, \# \{ i\leftarrow
    j\} = a_{ij}$.
  \end{itemize}
Therefore
\begin{equation*}
   \chi_{\vq,{t=1}}(x_i')_{\le 2} = Y_{i,1}(1 
   + V_{i,\vq} \prod_{j}(1 + V_{j,\vq^2})^{a_{ij}}).
\end{equation*}
We have
\begin{equation*}
  \begin{split}
  \chi_{\vq,{t=1}}(x_i)_{\le 2} \chi_{\vq,{t=1}}(x_i')_{\le 2} 
  = Y_{i,1} Y_{i,\vq^2}(1 
   + V_{i,\vq} \prod_{j}(1 + V_{j,\vq^2})^{a_{ij}})
  &= Y_{i,1} Y_{i,\vq^2} + \prod_j Y_{j,\vq}^{a_{ij}}(1 + V_{j,\vq^2})^{a_{ij}})
\\
  &= \chi_{\vq,{t=1}}(f_i)_{\le 2} + \chi_{\vq,t=1}(\bigotimes_j x_j^{a_{ij}})_{\le 2}.
  \end{split}
\end{equation*}
\end{NB}%
\end{Example}

\begin{proof}[Proof of \thmref{thm:main}]
  Since $\bE_W$ is a vector space by \propref{prop:ast_1} and
  $IC_W(V)$'s are monodromic (i.e.\ $H^j(IC_W(V))$ is locally constant
  on every $\C^*$-orbit of $\bE_W$), we can apply the
  Fourier-Sato-Deligne functor $\Psi$ (\cite{KaSha,Lau}). For example,
  we have
  \begin{equation*}
    \Psi(IC_W(0)) = 1_{\bE_W^*}[\dim \bE_W].
  \end{equation*}
  Other $\Psi(IC_W(V))$ are simple perverse sheaves on $\bE_W^*$.

  Recall that $\Tilde{\mathcal F}(\nu,W)$ is a vector subbundle of the
  trivial bundle $\mathcal F(\nu,W)\times \bE_W$ by
  \propref{prop:ast_1}. Let $\Psi'$ be the Fourier-Sato-Deligne
  functor for this trivial bundle. We have
  \begin{equation*}
    \Psi'(1_{\Tilde{\mathcal F}(\nu,W)}[\dim \Tilde{\mathcal F}(\nu,W)])
    = 1_{\Tilde{\mathcal F}(\nu,W)^\perp}[\dim\Tilde{\mathcal F}(\nu,W)^\perp],
  \end{equation*}
  where $\Tilde{\mathcal F}(\nu,W)^\perp$ is the annihilator in the
  dual trivial bundle $\mathcal F(V,W)\times\bE_W^*$ as in
  \subsecref{subsec:sigma}. Moreover we have
  \begin{equation*}
     \pi^\perp_! \circ\Psi' = \Psi\circ\pi_!.
  \end{equation*}
  Therefore if we decompose the pushforward as
  \begin{equation*}
     \pi^\perp_!(1_{\Tilde{\mathcal F}(\nu,W)^\perp}[\dim\Tilde{\mathcal
       F}(\nu,W)^\perp])
     \cong \bigoplus_{V',l} L_{V',l}\otimes \Psi(IC_W(V'))[l],
  \end{equation*}
  we have $\sum_l t^l \dim L_{V',l} = a_{V,V';W}(t)$.

  Take a general point of $\bE_W^*$ and consider the Poincar\'e
  polynomial of the stalk of the above. In the left hand side we get
  the Poincar\'e polynomial of $\Gr_V(\lsp{\sigma}{W})$ by
  \lemref{lem:reflect}. On the other hand, in the right hand side the
  factor $\Psi(IC_W(V'))$ with $IC_W(V')\notin \mathscr L_W$
  disappears as its support is smaller than $\bE_W^*$. For
  $IC_W(V')\in\mathscr L_W$, we get $ r_W(IC_W(V'))\times
  a_{V,V';W}(t)$, as $\Psi(IC_W(V'))$ is the IC complex associated
  with a local system of rank $r_W(IC_W(V'))$ defined over an open
  subset of $\bE_W^*$. Thus we have
  \begin{equation}\label{eq:Poincare}
    P_t(\Gr_V(\lsp{\sigma}{W}))
    = \sum_{IC_W(V')\in \mathscr L_W}
    r_W(IC_W(V'))\, a_{V,V';W}(t).
  \end{equation}
  We get the assertion by recalling that
  $a_{V,V';W}(t)$ is the coefficient of $e^{W^\perp} e^{V^\perp} =
  e^{W} e^{V}$ in the $\vq$-character of $L(W^\perp)$, where $\dim
  W^\perp = \dim W - \bC_\vq \dim V'$, $\dim V^\perp = \dim V - \dim
  V'$ (\subsecref{subsec:slice}).
  \begin{NB}
    Therefore we know that $a_{V,V';W}(t)\in t^{\dim
      \N(V^\perp,W^\perp)}\Z_{\ge 0}[t^2]$ by \remref{rem:negative}.
  \end{NB}%
\end{proof}

\begin{NB}
Before studying applications to modules in $\mathscr
R_{\mathrm{conv}}$, we give an application to geometry of $\N_0(W)$.

It seems likely that the open stratum in $\N_0(W)$ is $\Nreg(V,W)$
with something close to
\begin{equation*}
  \dim V_i(\vq^{\xi_i+1}) = \lsp{\sigma}W_i.
\end{equation*}
But in general, it seems $(V,W)$ is not {\it l\/}-dominant
though $\pi \colon \N(V,W) \to \bE_W$ is generically 1 to 1.

Let us first check that $(V,W)$ is {\it l\/}-dominant. We have
{\allowdisplaybreaks
  \begin{equation*}
    \begin{split}
      \rank C^\bullet_{i,1} &= \dim W_i - \dim V_i(\vq) = 0 \quad
      (i\in I_0),
      \\
      \rank C^\bullet_{i,\vq} 
      &= \dim W_{i'} + \sum_j a_{ij} \dim
      V_j(\vq) - \dim V_i(\vq^2)
      \\
      &= \min( \dim W_i, \dim W_{i'} +
      \sum_j a_{ij} \dim V_j(\vq)) \ge 0 \quad (i\in I_1),        
\\
      \rank C^\bullet_{i,\vq^2} 
      &= \sum_j a_{ij} \dim V_j(\vq^2) - \dim V_i(\vq) 
\\
      &= \sum_j a_{ij} \dim \lsp{\sigma} W_j
      - \dim \lsp{\sigma} W_i
      \quad (i\in I_0),
\\
      \rank C^\bullet_{i,\vq^3} 
      &= \dim W_i - \dim V_i(\vq^2) = \dim W_i - \dim\lsp{\sigma}W_i
      \quad (i\in I_1).
    \end{split}
  \end{equation*}
Thus} the only question is $\rank C^\bullet_{i,\vq^2} \ge 0$,
$\rank C^\bullet_{i,\vq^3}$.
If we assume $W$ is indecomposable and not equal to $S_i$, $S_{i'}$,
the map
\(
   \sum_{\vin(h)=i} \lsp{\sigma}\by_h 
   \colon\bigoplus \lsp{\sigma}W_{\vout(h)} \to \lsp{\sigma}W_i
\)
is surjective. Hence we have the inequality in $\rank C^\bullet_{i,\vq^2}
\ge 0$. But how we get $\rank C^\bullet_{i,\vq^3} \ge 0$ ?
\end{NB}

\subsection{Factorization of KR modules}
In the remainder of this section, we give several simple applications
of \thmref{thm:main}.

\begin{Proposition}\label{prop:fac}
  \begin{equation*}
     {\mathbb L}(W) \cong
     {\mathbb L}(\lsp{\varphi}{W}) \otimes \bigotimes_{i\in I}
      f_i^{\min(\dim W_i,\dim W_{i'})},
  \end{equation*}
  \begin{NB}
    Earlier version:
  \begin{equation*}
     L(W) \cong
     L(\lsp{\varphi}{W}) \otimes \bigotimes_{i\in I}
      x_i^{\max(\dim W_{i'}- \dim W_i,0)}
      \otimes
      f_i^{\min(\dim W_i,\dim W_{i'})},
  \end{equation*}
  \end{NB}
where $\lsp{\varphi}{W}$ is given by
\begin{equation*}
   \dim \lsp{\varphi}{W}_i = 
    \max(\dim W_i - \dim W_{i'},0),
\quad
   \dim \lsp{\varphi}{W}_{i'} = 
     \max(\dim W_{i'}-\dim W_i,0)
\end{equation*}
\begin{NB}
  Earlier version
\begin{equation*}
   \dim \lsp{\varphi}{W}_i = 
   \dim W_i - \min(\dim W_i,\dim W_{i'}),
\quad
   \dim \lsp{\varphi}{W}_{i'} = 0
\end{equation*}
\end{NB}%
The right hand side is independent of the order of the tensor product.
\end{Proposition}

From this proposition it becomes enough to understand
${\mathbb L}(\lsp{\varphi}{W})$.
Notice that either $\lsp{\varphi}W_i$ or $\lsp{\varphi}W_{i'}$ is zero
for each $i\in I$. If $\lsp{\varphi}W_i = 0$, then
$\lsp{\varphi}W_{i'}$ is not connected to any other vertexes, and is
easy to factor out it.
Thus we eventually reduce to study the case when all
$\lsp{\varphi}{W}_{i'}= 0$, i.e.\ $\bE_{\lsp{\varphi}{W}}$ is the
vector space of representations of the principal part of the decorated
quiver obtained by deleting all frozen vertexes $i'$.

\begin{NB}
Earlier version:

From this proposition it becomes enough to understand
$L(\lsp{\varphi}{W})$. We have $\lsp{\varphi}{W}_{i'}= 0$, i.e.\
$\bE_{\lsp{\varphi}{W}}$ is the vector space of representations of the
principal part of the decorated quiver obtained by deleting all frozen
vertexes $i'$.
Moreover, if $x_i$ actually appears in the factorization, we have
$\dim W_{i'} > \dim W_i$. Then we have $\dim \lsp{\varphi}{W}_i =
0$. It means that we can further delete the vertex $i$ from the
quiver.
This will be important in the further analysis in \secref{sec:exchange}.
\end{NB}

\begin{proof}
  From the definition of $\lsp{\sigma}{W}$ in the formula
  \eqref{eq:sigmaW} it is clear that $\lsp{\sigma}{W}_i$ is unchanged
  even if we add $\pm (1,1)$ to $(\dim W_i,\dim W_{i'})$ for $i\in
  I_1$. And the change of $\dim \lsp{\sigma}{W}_{i'}$ does not affect
  the quiver Grassmannian.
  Therefore we can subtract $\min(\dim W_i, \dim W_{i'})$ from both
  $\dim W_i$ and $\dim W_{i'}$ for each $i\in I_1$. Let $\Tilde W$ be
  the resulting $(I\sqcup I_\fr)$-graded vector space. We have
\begin{equation*}
  \chi_{\vq,t}({\mathbb L}(W))_{\le 2} = 
  \chi_{\vq,t}({\mathbb L}(\Tilde W))_{\le 2}\;
  \prod_{i\in I_1} (Y_{i,\vq} Y_{i,\vq^3})^{\min(\dim W_i,\dim W_{i'})}.
\end{equation*}
Since the truncated $q$-character of the Kirillov-Reshetikhin module
is equal to $Y_{i,\vq} Y_{i,\vq^3}$ by Example~\ref{ex:KR}, we have
\begin{equation*}
  \chi_{\vq,t}({\mathbb L}(W))_{\le 2} = 
  \chi_{\vq,t}({\mathbb L}(\Tilde W))_{\le 2}\;
  \prod_{i\in I_1} f_i^{\min(\dim W_i,\dim W_{i'})}.
\end{equation*}

\begin{NB}
Note that either $\Tilde W_i$ or $\Tilde W_{i'}$ is $0$ for each
$i$. Assume the former
\begin{NB2}
  $\Tilde W_i(\vq^3) = \Tilde W_i$
\end{NB2}%
is $0$. Then $\lsp{\sigma}{\Tilde W}_i \cong \Tilde W_{i'}\oplus
\bigoplus_{h:\vout(h)=i} W_{\vin(h)}$, and $\lsp{\sigma}{\bx}_i^{*}$,
$\lsp{\sigma}{\by}_h^{*}$ are projections. We consider the
$\C^*$-action with weight $1$ on $\Tilde W_{i'}$ and its copy in
$\lsp{\sigma}{\Tilde W}_i$.  Then $\lsp{\sigma}{\bx}_i^{*}$,
$\lsp{\sigma}{\by}_h^{*}$ are preserved and hence we have the induced
action on $\Gr_V(\lsp{\sigma}{\Tilde W})$.
Recall that the Euler characteristic of a variety is the same as that
of the $\C^*$-fixed point. The $\C^*$-fixed points of
$\Gr_V(\lsp{\sigma}{\Tilde W})$ is the disjoint union of
$\Gr_{V^1}(\lsp{\sigma}{\Tilde{\Tilde W}})\times \Gr_{V^2}(\Tilde
W_{i'})$ for various $V^1$, $V^2$. Here (1) $\Tilde{\Tilde W}$ is
obtained from $\Tilde W$ by replacing $\Tilde W_{i'}$ by $0$, (2)
$\Tilde W_{i'}$ is considered an $I\sqcup I_\fr$-vector space by
setting $0$ on other entries.
Therefore we have
\[
   \chi_q(L(\Tilde W))_{\le 2} 
   = \chi_q(L(\Tilde{\Tilde W}))_{\le 2}\; x_i^{\dim \Tilde W_i(\vq)}.
\]
This can be proved more directly in terms of the convolution diagram
by observing that the corresponding variety $\bE_{\Tilde W}$ is
isomorphic to $\bE_{\Tilde{\Tilde W}}$.
\begin{NB2}
  \begin{equation*}
    \chi_{\vq}(L(W))_{\le 2} = \chi_{\vq}(L(\Tilde{\Tilde W}))_{\le 2}\;
    \prod_{i\in I_1} x_i^{\max(\dim W_i(\vq)- \dim W_i(\vq^3),0)}
    f_i^{\min(W_i(\vq),\dim W_i(\vq^3))}.
  \end{equation*}
Note also that $\Tilde{\Tilde W}_i(\vq^3) = 0$ if
$\dim W_i(\vq) \ge \dim W_i(\vq^3)$.
\end{NB2}
\end{NB}

Next we study a similar but slightly different reduction for $i\in
I_0$. We consider the variety $(\pi^\perp)^{-1}(\bigoplus
\bx^*_i,\bigoplus \by^*_{\overline{h}})$ as in the statement of
\lemref{lem:reflect}. Here $(\bigoplus \bx^*_i,\bigoplus
\by^*_{\overline{h}})$ is the representation considered in 
the proof of \lemref{lem:reflect} before applying the
reflection functors.
From the condition $X_i\subset \Ker x_i^*$ it is isomorphic to
$(\bar\pi^\perp)^{-1}(\bigoplus
\bar\bx^*_i,\bigoplus\bar\by^*_{\overline{h}})$, where (1) $\Tilde W$
obtained from $W$ by replacing $W_i$ by $\Ker x_i^*$,
\begin{NB}
  Earlier version: and $W_{i'}$ by $0$,
\end{NB}%
(2) $\bar\by^*_{\overline{h}}$ is the restriction of
$\by^*_{\overline{h}}$ and other maps are obvious ones. We have
\begin{equation*}
   \dim \Tilde W_{i} = \max(\dim W_{i} - \dim W_{i'},0).
\end{equation*}
Therefore we have
\begin{equation*}
   \chi_{\vq,t}({\mathbb L}(W))_{\le 2} =
   \chi_{\vq,t}({\mathbb L}(\Tilde W))_{\le 2}\;
   \prod_{i\in I_0} (Y_{i,1} Y_{i,\vq^2})^{\min (\dim W_{i}, \dim W_{i'})}.
\end{equation*}
\begin{NB}
Earlier version:
\begin{equation*}
   \chi_{\vq}(L(W))_{\le 2} =
   \chi_{\vq}(L(\Tilde W))_{\le 2}\;
   \prod_{i\in I_0} (Y_{i,1} Y_{i,\vq^2})^{\min (\dim W_{i}, \dim W_{i'})}
    \,  Y_{i,\vq^2}^{\max(\dim W_{i'} - \dim W_i,0)}.
\end{equation*}
\end{NB}%
Note again that $Y_{i,1} Y_{i,\vq^2}$ is the truncated $q$-character
of the Kirillov-Reshetikhin module $f_i$.
\begin{NB}
Earlier version :
 We also note that
$Y_{i,\vq^2}$ is the truncated $q$-character of $x_i$.   
\end{NB}%
Therefore the above equality can be written as
\begin{equation*}
   \chi_{\vq,t}({\mathbb L}(W))_{\le 2} =
   \chi_{\vq,t}({\mathbb L}(\Tilde W))_{\le 2}\;
   \prod_{i\in I_0} f_i^{\min (\dim W_{i}, \dim W_{i'})}.
\end{equation*}
\begin{NB}
Earlier version:
\begin{equation*}
   \chi_{\vq}(L(W))_{\le 2} =
   \chi_{\vq}(L(\Tilde W))_{\le 2}\;
   \prod_{i\in I_0} f_i^{\min (\dim W_{i}, \dim W_{i'})}
   \,  x_{i}^{\max(\dim W_{i'}- \dim W_i,0)}.
\end{equation*}
\end{NB}%
Combining these two reductions we obtain the assertion.
\end{proof}

\subsection{Factorization and canonical decomposition}

Take a general representation $(\bigoplus\by_h)$ of
$\bE_{\lsp{\varphi}{W}}$. We decompose it into a sum of indecomposable
representations.
We have a corresponding decomposition
\begin{equation*}
   \lsp{\varphi}{W} = 
   \begin{NB}
   \bigoplus_{i\in I_1} S_i^{\oplus m_i} \oplus 
   \end{NB}%
   W^1\oplus W^2\oplus\cdots \oplus W^s
\end{equation*}
of the $I$-graded graded vector space.
\begin{NB}
, where we assume $W^1$, \dots,
$W^s$ are not isomorphic to $S_i$.
\end{NB}%
It is known \cite[p.85]{Kac-quiver} that $W^1$, \dots, $W^s$ are independent
of a choice of general representation of $\bE_{\lsp{\varphi}{W}}$ up to
permutation. This is called the {\it canonical decomposition\/} of
$\lsp{\varphi}{W}$ (or $\dim \lsp{\varphi}{W}$).
It is known that all $\dim W^\alpha\in\Z_{\ge 0}^I$ are {\it Schur
  roots\/} and $\ext^1(W^k, W^l) = 0$ for $k\neq l$
\begin{NB}
$\ext^1(S_i, W^k) = 0 = \ext^1(W^k,S_i)$   
\end{NB}%
(\cite[Prop.~3]{Kac-quiver2}).
Here $\dim W^k$ is a Schur root if a general representation in
$\bE_{W^k}$ has only trivial endomorphisms, i.e.\ scalars. It is known
that this is equivalent to a general representation is indecomposable
([loc.~cit., Prop.~1]).
And $\ext^1(W^k, W^l)$ is the dimension of $\Ext^1$ between general
representations in $\bE_{W^k}$ and $\bE_{W^l}$.
Basic results on the canonical decomposition were obtained by
Schofield \cite{Schofield}, which will be used in part below.

Note that the frozen part play no role in the canonical decomposition,
as $\lsp{\varphi}W_{i'}\neq 0$ implies $\lsp{\varphi}W_i =
0$. Therefore we simply have factors
$\underbrace{S_{i'}\oplus\dots\oplus S_{i'}}_{\text{$\dim
    \lsp{\varphi}W_{i'}$ factors}}$ in the canonical decomposition.
If $\lsp{\varphi}W$ contains a factor $S_i^{\oplus m_i}$ for $i\in
I_1$, it is killed by $\lsp{\sigma}(\ )$.
We thus have

\begin{Proposition}\label{prop:fac-1}
  Suppose that the canonical decomposition of $\lsp{\varphi}W$
  contains factors as
  \begin{equation*}
     \lsp{\varphi}W = \lsp{\psi}W \oplus 
     \bigoplus_{i\in I} S_{i'}^{\oplus \dim \lsp{\varphi}W_{i'}}
     \oplus
     \bigoplus_{i\in I_1} S_{i}^{\oplus m_i}.
  \end{equation*}
  Then we have a factorization
  \begin{equation*}
    {\mathbb L}(\lsp{\varphi}W) = {\mathbb L}(\lsp{\psi}W)
    \otimes 
    \bigotimes_{i\in I}   L(S_{i'})^{\otimes \dim \lsp{\varphi}W_{i'}}\otimes
    \bigotimes_{i\in I_1} L(S_{i})^{\otimes m_{i}}.
  \end{equation*}
\end{Proposition}

We consider the following condition (C):
\begin{equation}
  \text{The canonical decomposition of $\lsp{\varphi}W$
    contains only {\it real\/} Schur roots.}
  \tag{C}
\end{equation}

\begin{Proposition}\label{lem:real}
\textup{(1)} Assume the condition {\rm (C)}.
Then $\mathscr L_{\lsp{\varphi}W} = \{ IC_{\lsp{\varphi}W}(0)\}$ and
hence ${\mathbb L}(W) = L(W)$.

\textup{(2)} If $\mathscr L_{\lsp{\varphi}W} = \{ IC_{\lsp{\varphi}W}(0)\}$,
$\Gr_V(\lsp{\sigma}W)$ has no odd cohomology.
\end{Proposition}

\begin{proof}
  (1) From the definition, $\bE_{\lsp{\varphi}W}$ contains the
  $\prod_i \GL(W_i)\times \GL(W_{i'})$ orbit of a general
  representation as a Zariski open subset. The same is true for
  $\bE_{\lsp{\varphi}W}^*$.  Since all $\Psi(IC_W(V))$ are $\prod_i
  \GL(W_i)\times \GL(W_{i'})$-equivariant, we cannot have IC complexes
  associated with nontrivial local systems as stabilizers are always
  connected.
  Therefore we only have $\mathscr L_{\lsp{\varphi}W} = \{
  IC_{\lsp{\varphi}W}(0)\}$.

  (2) Since \eqref{eq:Poincare} is a single sum, the assertion follows
  from \remref{rem:negative}.
\end{proof}

\begin{Proposition}\label{prop:canfac}
\begin{NB}
Assume {\rm (C)}. Then
\begin{equation}
   {L}(\lsp{\varphi}{W}) \cong
   \begin{NB2}
   \bigotimes_{i\in I_1} x_{i'}^{\otimes m_i} \otimes 
   \end{NB2}%
   {L}(W^1)\otimes\cdots \otimes {L}(W^s).     
\end{equation}
\end{NB}
\begin{equation}\label{eq:Schur}
   {\mathbb L}(\lsp{\varphi}{W}) \cong
   \begin{NB}
   \bigotimes_{i\in I_1} x_{i'}^{\otimes m_i} \otimes 
   \end{NB}%
   {\mathbb L}(W^1)\otimes\cdots \otimes {\mathbb L}(W^s).     
\end{equation}
\end{Proposition}

\begin{proof}
  We assume $s=2$. Since we do not use the assumption that $W^1$,
  $W^2$ are Schur roots, the proof also gives the proof for general
  case.

  Consider the convolution diagram in \subsecref{subsec:conv}.
  By \cite[10.1]{Lu-book} the restriction functor commutes with the
  Fourier-Sato-Deligne functor up to shift. Therefore we consider
  perverse sheaves defined over $\bE_{W^1}^*$, $\bE_{W^2}^*$,
  $\bE_W^*$.

  We take open subsets $U^1$, $U^2$ in $\bE_{W^1}^*$, $\bE_{W^2}^*$ so
  that perverse sheaves not in $\mathscr L_{W^1}$, $\mathscr L_{W^2}$
  have support outside of $U^1$, $U^2$. Similarly we take an open
  subset $U\subset\bE_W^*$ consisting of modules isomorphic to direct
  sum of modules in $U^1$ and $U^2$, and perverse sheaves not in
  $\mathscr L_W$ have support outside of $U$.

  We may assume that $\Ext$-groups between modules in $U^1$, $U^2$
  vanish. Therefore any module in $\kappa^{-1}(U^1\times U^2)$ is
  isomorphic to direct sum of modules from $U^1$ and $U^2$. Therefore
  $\kappa^{-1}(U^1\times U^2)\subset U$ and $\kappa$ is an
  isomorphism.
  Therefore for $L\in \mathscr P_W\setminus \mathscr L_W$,
  $\operatorname{Res}L$ does not have factors in
  $IC_{W^1}(V^1)\boxtimes IC_{W^2}(V^2)$ with
  $IC_{W^\alpha}(V^\alpha)\in\mathscr L_{W^\alpha}$
  ($\alpha=1,2$). Therefore the product of $L(W^{\prime 1})\in\mathscr
  L_{W^1}$ and $L(W^{\prime 2})\in\mathscr L_{W^2}$ is a linear
  combination of elements in $\mathscr L_W$.

  If $IC_W(V)\in \mathscr L_W$, the restriction of $\kappa_!\iota^*
  \Psi(IC_W(V))$ to $U^1\times U^2$ is a local system of rank
  $r(IC_W(V))$. Thus if we write
  \begin{multline*}
     \operatorname{Res} IC_W(V)
     = \sum_{IC_{W^1}(V^1)\in\mathscr L_{W^1}, IC_{W^2}(V^2)\in \mathscr L_{W^2}}
     a^{V^1,V^2}_V IC_{W^1}(V^1)\boxtimes IC_{W^2}(V^2)
\\
     + (\text{linear combination of
       $L\in \mathscr P_W\setminus \mathscr L_W$}),
  \end{multline*}
then $a^{V^1,V^2}_V$ is an integer (up to shift). And we have
\[
   r(IC_W(V)) = \sum_{V^1,V^2} a^{V^1,V^2}_V r(IC_{W^1}(V^1))r(IC_{W^2}(V^2)).
\]
\begin{NB}
  Let $L_{V^1}$, $L_{V^2}$, $L_V$ be the dual elements of
  $IC_{W^1}(V^1)$, $IC_{W^2}(V^2)$, $IC_W(V)$. Then
  \begin{equation*}
    L_{V^1}\otimes L_{V^2} = \sum_V a^{V^1,V^2}_V L_V.
  \end{equation*}
Therefore
\begin{equation*}
   \sum_{L_V\in \mathscr L_W} r(L_V) L_V
   = \sum_{L_V\in \mathscr L_W} \sum_{L_{V^1},L_{V^2}} a^{V^1,V^2}_V 
   r(L_{V^1}) r(L_{V^2}) L_V
   = \sum_{L_{V^1},L_{V^2}} r(L_{V^1}) r(L_{V^2}) L_{V^1}\otimes L_{V^2}.
\end{equation*}
\end{NB}%
From this we have $\mathbb L(W^1)\otimes \mathbb L(W^2) = \mathbb L(W)$.
\begin{NB}
Original argument:  

  From the definition, we have ${\mathbb L}(W) = L(W)$ and
  ${\mathbb L}(W^k) = L(W^k)$ for any $k$.
  Then we have
  \begin{equation*}
    P_t(\Gr_V(\lsp{\sigma\varphi}W)) = a_{V,0;W}(t)
  \end{equation*}
  from \eqref{eq:Poincare}. By \remref{rem:negative} the right hand
  side is in $t^{\dim \N(V,W)}\Z[t^{-2}]$.
  Therefore $\Gr_V(\lsp{\sigma\varphi}W)$ has no odd cohomology groups, and
  hence $P_{t=1}(\Gr_V(\lsp{\sigma}W))$ is the Euler number.
  \begin{NB2}
    The degrees of $t$ in the left hand side are in
    $-\dim\N(V,W)$ to $2\dim \Gr_V(\lsp{\sigma}W) -\dim \N(V,W)$.
    Since $\dim \Gr_V(\lsp{\sigma}W)\le \dim \N(V,W)$, we have a
    little better bound.
  \end{NB2}%
  Thus the truncated $\vq$-character is the generating function of
  Euler numbers of quiver Grassmannian.
  The same applies to $\lsp{\sigma\varphi}W^k$.

  If $\lsp{\varphi}W$ contains a factor of the form $S_{i}$ ($i\in
  I_1$), it is killed by $\sigma$. It means that we have a factor
  $Y_{i,\vq^3}$ in $\chi_{\vq}({\mathbb L}(\lsp{\varphi}W))_{\le 2}$
  for each $S_i$.
Since we have
\(
   \chi_{\vq}({L}(S_i))_{\le 2} =
   \chi_{\vq}(x_{i'})_{\le 2} = Y_{i,\vq^3},
\)
we have the corresponding factor in the tensor product decomposition.
Therefore we may assume that $\lsp{\varphi}W$ does not contain a factor
$S_i$ in its canonical decomposition.

Then we have the canonical corresponding decomposition
\(
  \lsp{\sigma}{\lsp{\varphi}{W}} = 
  \lsp{\sigma}{W}^1\oplus \lsp{\sigma}{W}^2\oplus\cdots 
  \oplus \lsp{\sigma}{W}^s
\)
of $\lsp{\sigma}{\lsp{\varphi}{W}}$.
Let us consider the torus $T^s$ action on $\lsp{\varphi\sigma}{W}$
given by $t = t_1\id_{\lsp{\sigma}{W}^1}\oplus\cdots \oplus
t_s\id_{\lsp{\sigma}{W}^s}$ for $t = (t_1,\dots,t_s)\in T^s$. It
induces an action on $\bE_{\lsp{\sigma\varphi}W}$. Since it fixes
$(\bigoplus\by_h)$, we have an induced action on
$\Gr_{V}(\lsp{\sigma\varphi}{W})$. Recall that the Euler
characteristic of a variety is the same as that of the torus fixed
point. We have
\[
   \Gr_{V}(\lsp{\sigma}{\lsp{\varphi}{W}})^{T^s}
   = \bigsqcup \prod_{k=1}^s \Gr_{V^k}(\lsp{\sigma}{W^k}),
\]
where the summation is over various decompositions $V = V^1 \oplus
\cdots \oplus V^s$. This induces the corresponding factorization of
truncated $q$-characters.
\end{NB}%
\end{proof}

Let us show the converse.
\begin{Proposition}\label{prop:Schur}
\textup{(1)}
  Suppose that ${\mathbb L}(\lsp{\varphi}W)$ decomposes as
  \begin{equation*}
     {\mathbb L}(\lsp{\varphi}W) 
     \cong {\mathbb L}(W^1)\otimes {\mathbb L}(W^2).
  \end{equation*}
Then we have $\ext^1(W^1, W^2) = 0 = \ext^1(W^2,W^1)$.

\textup{(2)} The same assertion is true even if the almost simple
modules $\mathbb L(\ )$ are replaced by simple modules $L(\ )$.

\textup{(3)} The factorization of an almost simple module ${\mathbb
  L}(W)$ is exactly given by the canonical decomposition of
$\lsp{\varphi}W$, and we have the bijection
\begin{equation*}
  \left\{ \txt{prime almost simple modules
    \\ with {\rm (C)}}\right\}
  \setminus \{ x_i, f_i \mid i\in I\}
  \longleftrightarrow
  \left\{ \txt{Schur roots of the principal\\
      part $\cQ$ of the decorated quiver}\right\}
\end{equation*}
given by ${\mathbb L}(W)\leftrightarrow\dim W$.
\end{Proposition}

Here an almost prime simple module $\mathbb L(W)$ means that it does
not factor as $\mathbb L(W^1)\otimes \mathbb L(W^2)$ of almost simple
modules.

\begin{proof}
(1)  Let us first consider the case $x_{i'} = \mathbb L(W^2) =
  L(W^2)$. Taking the truncated $q$-character, we have
  \begin{equation*}
     \sum_V P_t(\Gr_V(\lsp{\sigma}W)) e^W e^V
     = 
     Y_{i,\vq^3} \ast
     \left(
       \sum_{V^1} P_t(\Gr_{V^1}(\lsp{\sigma}W^1)) e^{W^1} e^{V^1}
     \right), 
  \end{equation*}
where $\ast$ is the twisted multiplication \eqref{eq:twist}.

Since $i\in I_1$ is a source, we have $\ext^1(W^1,S_i) = 0$.
If we have $\ext^1(S_i,W^1)\neq 0$, then
$\dim \lsp{\sigma}W^1 = \dim \lsp{\sigma}W + \dim S_i$.
Therefore the right hand side contains the term for $V^1$ with $\dim
V^1 = \dim \lsp{\sigma}W + \dim S_i$, as the corresponding quiver
Grassmannian $\Gr_{\lsp{\sigma}W^1}(\lsp{\sigma}W^1)$ is a single
point.

But the left hand side obviously cannot contain the corresponding term.
Therefore we must have $\ext^1(S_i,W^1) = 0$.

\begin{NB}
  \begin{Example}
  Consider an example $A_3$. We take $W^1 = S_3$, $S_i = S_2$. Then
  $\Ext^1(S_2,W^1) = \C$, as we have an arrow $2 \to 3$.

  A representation of $W = W^1\oplus S_i$ is $(0 \leftarrow
  \C\rightarrow \C)$. Then $\lsp{\sigma}W = (0 \leftarrow 0
  \rightarrow \C)$. We have
  \begin{equation*}
     \chi_{\vq}(L(W))_{\le 2} = Y_{2,\vq^3} Y_{3,1}(1 + V_{3,\vq}).
  \end{equation*}
But $\lsp{\sigma}W^1 = (0 \leftarrow \C\rightarrow \C)$ and we have
\begin{equation*}
   \chi_{\vq}(L(W^1))_{\le 2} = Y_{3,1}(1 + V_{3,\vq} + V_{2,\vq^2} V_{3,\vq}).
\end{equation*}
  \end{Example}
\end{NB}

Now we suppose general representations of $W^1$ and $W^2$ do not
contain the direct summand $S_i$ for any $i\in I_1$.
Then the vanishing of $\ext^1$ is equivalent to the corresponding
statement after applying the functor $\sigma$. (Since $\sigma$ starts
with taking the {\it dual}, we need to exchange the first and the
second entries $A$, $B$ of $\ext^1(A,B)$, but we are studying both
$\ext^1(A,B)$ and $\ext^1(B,A)$, so it does not matter.)

We again consider the equality for the truncated $\vq$-character:
\begin{equation*}
     \sum_V P_t(\Gr_V(\lsp{\sigma}W)) e^W e^V
     = 
     \left(
       \sum_{V^1} P_t(\Gr_{V^1}(\lsp{\sigma}W^1)) e^{W^1} e^{V^1}
     \right) 
     \ast
     \left(
       \sum_{V^2} P_t(\Gr_{V^2}(\lsp{\sigma}W^2)) e^{W^2} e^{V^2}
     \right).
  \end{equation*}
  The right hand side contain the terms with $V^1 = \lsp{\sigma}W^1$,
  $V^2 = 0$ and $V^1 = 0$, $V^2 = \lsp{\sigma}W^2$, as both
  $\Gr_{V^1}(\lsp{\sigma}W^1)$ and $\Gr_{V^2}(\lsp{\sigma}W^2)$ are
  points in these cases.
  These survive thanks to the positivity
  $P_t(\Gr_{V^1}(\lsp{\sigma}W^1))$,
  $P_t(\Gr_{V^2}(\lsp{\sigma}W^2))\in \Z_{\ge 0}[t]$.
  Therefore the corresponding quiver Grassmannian varieties
  $\Gr_V(\lsp{\sigma}W)$ (two cases) are nonempty in the left hand
  side also.
  Therefore a general representation of $\bE_W$ contains two
  subrepresentations of $\dim \lsp{\sigma}W^1$, $\dim \lsp{\sigma}W^2$
  respectively.
  By \cite[Th.~3.3]{Schofield}, it implies that we have both
  $\ext^1(\lsp{\sigma}W^1,\lsp{\sigma}W^2) = 0$ and
  $\ext^1(\lsp{\sigma}W^2,\lsp{\sigma}W^1) = 0$.
This proves the first assertion.

\begin{NB}
A first attempt of the proof:  

Take the open stratum $\Nreg(V,W)$ of $\bE_W = \N_0(W)$. Then 
\[
   \pi\colon \N(V,W)=\Tilde{\mathcal F}(\nu,W)\to
   \N_0(W) = \bE_W
\]
is an isomorphism on $\pi^{-1}(\Nreg(V,W))$. In particular, we have
$\dim \N(V,W) = \dim \mathcal F(\nu,W) + \rank \Tilde{\mathcal
  F}(\nu,W) = \dim \bE_W$. Therefore
\(
   \dim \Tilde{\mathcal F}(\nu,W)^\perp
   = \dim \mathcal F(\nu,W) + \dim \bE_W - 
   \rank \Tilde{\mathcal F}(\nu,W)
   = 2 \dim \mathcal F(\nu,W).
\)
\begin{quote}
  $\Tilde{\mathcal F}(\nu,W)^\perp = 
  T^*\mathcal F(\nu,W)$ in this case ? 
\end{quote}
And similarly $\Nreg(V^1,W^1)$, $\Nreg(V^2,W^2)$ of $\bE_{W^1} =
\N_0(W^1)$, $\bE_{W^2} = \N_0(W^2)$ respectively.
\end{NB}

(2) We make a closer look to the above argument: Recall that
$P_t(\Gr_V(\lsp{\sigma}W))$ is the sum of contributions from $L(W)$
and other perverse sheaves, as \eqref{eq:Poincare}, and
$r_W(IC_W(V'))\in\Z_{> 0}$, $a_{V,V';W}(t)\in \Z_{\ge 0}[t]$.

For the case $V^1 = \lsp{\sigma}W^1$, $V^2 = 0$ or $V^1 = 0$, $V^2 =
\lsp{\sigma}W^2$, the corresponding subspace is uniquely determined,
and the projection $\pi^\perp\colon \Tilde{\mathcal
  F}(\nu,W)^\perp\to\bE_W^*$ becomes an isomorphism. Therefore other
perverse sheaves do not appear in $\pi_!^\perp(1_{\Tilde{\mathcal
    F}(\nu,W)^\perp})$. Therefore we must have
$\Gr_V(\lsp{\sigma}W)\neq\emptyset$ in the two cases. The remaining
argument is the same.

(3) The assertion follows from the first and the characterization of
the canonical decomposition:
  $\alpha = \sum \beta^i$ is the canonical decomposition if and only
  if each $\beta^i$ is a Schur root and $\ext^1(\beta^i,\beta^j) = 0$
  for $i\neq j$.
(See \cite[Prop.~3]{Kac-quiver2}.)
\end{proof}

\begin{Corollary}
  If $L(W)$ satisfies {\rm (C)}, it is real, i.e.\ $L(W)\otimes L(W)$
  is simple.
\end{Corollary}

At this moment, we do not know the converse is true or not.

Next suppose $\cG$ is of type $ADE$. Then all positive roots are real
and Schur. Let $\Delta_+$ be the set of positive roots. Following
\cite{FomZel2} we introduce the set $\Phi_{\ge -1}$ of {\it almost 
positive\/} roots:
\begin{equation*}
   \Phi_{\ge -1} = \Delta_+ \sqcup \{ -\alpha_i \mid i\in I\},
\end{equation*}
where $\alpha_i$ is the simple root for $i$.

\begin{Corollary}\label{cor:ADE}
  \textup{(1)} There are only finitely many prime simple modules in
  $\bfR_{\ell=1}$ if and only if the underlying graph $\cG$ of the
  principal part is of type $ADE$.

  \textup{(2)} Suppose that $\cG$ is of type $ADE$. Then all simple
  modules are real, and there is a bijection
\begin{equation*}
  \{ \text{prime simple modules}\}
  \setminus \{ f_i \mid i\in I\}
  \xrightarrow[\mathrm{1:1}]{\dim(\bullet)}
  \Phi_{\ge -1}.
\end{equation*}
Here the bijection is given by \propref{prop:Schur}(3) together with
$x_i\mapsto (-\alpha_i)$.
\end{Corollary}

The first assertion is a simple consequence of the fact that there are
infinitely many real Schur roots for non $ADE$ quivers. This can be
shown for example, by observing non $ADE$ graph always contains an
affine graph.
\begin{NB}
\cite[VII.2.1]{ASS}.
\end{NB}%
Then for an affine graph, real roots $\alpha$ with the defects
\(
   \chi(\delta, \alpha)
   = \dim \Hom(\delta,\alpha) - \dim \Ext^1(\delta,\alpha)
\)
are nonzero are Schur. Here $\delta$ is the generator of positive
imaginary roots and the above is Euler form for a representation $N$
with $\dim N = \delta$ and $M$ with $\dim M=\alpha$, which is
independent of the choice of $M$, $N$.

This Corollary is nothing but \cite{FomZel2} after identifying prime
simple modules with cluster variables in the next section.

Now we consider the affine case.

\begin{Example}
  Suppose that $(I,E)$ of type $A_1^{(1)}$. The corresponding quiver
  $(I,\Omega)$ is called the Kronecker quiver. Positive roots are $(n
  \rightrightarrows n+1)$, $(n+1\rightrightarrows n)$,
  $(n\rightrightarrows n)$ ($n\in\Z_{\ge 0}$). The vector
  $(1\rightrightarrows 1)$ is the generator of positive imaginary
  roots, and denoted by $\delta$ as above.

  For $n\in \Z_{>0}$ let $nW$ denote an $(I\sqcup I_\fr)$-graded
  vector space with $\C^n$ at the entry $i$ and $0$ at ${i'}$
  ($i=0,1$): $(nW)_0 = \C^n \rightrightarrows
  (nW)_1 = \C^n$.
  Thus $\dim (nW) = n\delta$.
  Then $nW = W\oplus\cdots \oplus W$ is the canonical decomposition of
  $nW$, where $W$ means $1 W$: It is well-known that a general
  representation in $\bE_{W}$ corresponds to a point in
  $\proj^1(\C)$. And a general representation in $\bE_{nW}$
  corresponds to distinct $n$ points in $\proj^1(\C)$.

  For a real positive root $(n\rightrightarrows n+1)$ or $(n+1
  \rightrightarrows n)$, there is the unique indecomposable module
  $M$. It is known that either $\Ext^1(M,W)$ or $\Ext^1(W,M)$ are
  nonvanishing. Therefore $M$ and $W$ cannot appear in a canonical
  decomposition simultaneously.
  It is also known that extensions between $(n\rightrightarrows n +
  1)$ and $((n+1)\rightrightarrows (n+2))$ vanish. It is also true for
  $((n+1)\rightrightarrows n)$ and $((n+2)\rightrightarrows (n+1))$.
  All other pairs, one of extensions does not vanish.

  From these observations, the canonical decomposition only have real
  Schur roots, except the case $nW$.
  \begin{NB}
    For example, consider $W$ with $\dim W = (w_0\rightrightarrows
    w_1)$, $w_1 > w_0$. Put
    \begin{equation*}
      n \defeq \left[ \frac{w_0}{w_1 - w_0} \right],
\quad
      a \defeq (n+1) (w_1 - w_0) - w_0,
\quad
      b \defeq w_0 - n(w_1 - w_0),
    \end{equation*}
where $[\ ]$ denotes the integral part. From the definition,
$n\le w_0/(w_1 - w_0) < n + 1$, we have $a > 0$, $b\ge 0$. We also have
\begin{equation*}
  (w_0\rightrightarrows w_1)
  = a   (n\rightrightarrows (n+1)) + b ((n+1)\rightrightarrows (n+2)).
\end{equation*}
  \end{NB}%
  We consider the case $n=2$. If we consider $\pi^\perp \colon
  \Tilde{\mathcal F}(\nu,2W)^\perp\to\bE_{2W}^*$ in
  \subsecref{subsec:sigma}, the perverse sheaves appearing (up to shift) in
the pushforward
\(
     \pi^\perp_!(1_{\Tilde{\mathcal F}(\nu,2W)^\perp}[\dim\Tilde{\mathcal
       F}(\nu,2W)^\perp])
\)
was studied in \cite{Lu-affine}. If we take $\nu = (1,1)\in\Z_{\ge
  0}^I$, then $\pi^\perp$ is the principal $\{ \pm 1\}$ cover over the
open set $\bE_{2W}^{*\mathrm{reg}}$ corresponding to distinct pairs of
points in $\proj^1(\C)$.
Then from [loc.\ cit.] we have
\begin{equation*}
  {\mathscr L}_{2W} = \{ 1_{\{0\}}, \Psi^{-1}(IC(\bE_{2W}^*,\rho)) \},
\end{equation*}
where $IC(\bE_{2W}^*,\rho)$ is the IC complex associated with the
nontrivial local system $\rho$ corresponding to the nontrivial
representation of $\{ \pm 1\}$.
\begin{NB}
  The following is probably true:
  $\Psi^{-1}(IC(\bE_{2W}^*,\rho))$ is the IC complex associated with
  the stratum $\{ \text{regular} \oplus S_1\oplus S_2 \}$.
\end{NB}%
In particular, the almost simple module $\mathbb L(2W)$ is not
the simple module $L(2W)$. 
On the other hand $\mathscr L_W = \{ 1_{\{0\}}\}$.

The coefficient of $\chi_\vq(L(2W))$ at $Y_{1,1}^2 Y_{2,\vq^3}^2\times
V_{1,\vq} V_{2,\vq^2}$ is $1$. The coefficients of $\chi_\vq(L(W))$ at
$Y_{1,1}Y_{2,\vq^3}V_{1,\vq}$, $Y_{1,1}Y_{2,\vq^3}V_{2,\vq^2}$ are
both $1$.  Therefore $L(2W)\not\cong L(W)\otimes L(W)$, i.e.\ $L(W)$
is {\it not\/} real. On the other hand, we have
$\mathbb L(2W) \cong \mathbb L(W)\otimes \mathbb L(W)$.

There are many attempts to construct a base for the cluster algebra
corresponding to this example in the cluster algebra literature
(\cite{ShermanZelevinsky,CalderoZelevinsky,Dupont,DXX} and \cite{GLS2}
in a wider context). The problem is how to understand imaginary root
vectors, and the solution is not unique.
Relationship between various bases are studied by
Leclerc~\cite{Leclerc}.
\end{Example}

More generally if $W$ corresponds to an indivisible isotropic
imaginary root (i.e.\ in the Weyl group orbit of $\delta$ of a
subdiagram of affine type in $\cG$) in an arbitrary $\cQ$, we have
\begin{equation*}
    {\mathbb L}(nW) \cong {\mathbb L}(W)^{\otimes n}. 
\end{equation*}
This can be generalized thanks to the results by Schofield
\cite{Schofield}. First we have if $\alpha$ is a non-isotropic
imaginary Schur root, $n\alpha$ is also a Schur root for $n\in\Z_{>
  0}$ ([loc.\ cit., Th.~3.7]).
It is also known that an isotropic Schur root must be indivisible
([loc.\ cit., Th.~3.8].)
Therefore we introduce the following notation: For a $W$ as above and
$n\in\Z_{>0}$ let $nW$ be an $I$-graded vector space with $\dim (nW)_i
= n \dim W_i$.
For a factor $\mathbb L(W^k)$ in \eqref{eq:Schur} let $(n\mathbb
L)(W^k)$ be $\mathbb L(nW^k)$ if $\dim W^k$ is a non-isotropic Schur
imaginary root, and $\mathbb L(W^k)^{\otimes n}$ otherwise, i.e.\
$\dim W^k$ is a real or indivisible isotropic Schur root.

\begin{Corollary}\label{cor:real}
  Let $W$ be as above. Let $\lsp{\varphi}W = W^1\oplus W^2\oplus\cdots
  \oplus W^s$ be the canonical decomposition.
   Then we have
  \begin{equation*}
     {\mathbb L}(nW) \cong 
     (n\mathbb L)(W^1)\otimes\cdots \otimes (n\mathbb L)(W^s)
     \otimes \bigotimes_{i\in I}
      \begin{NB}
      x_i^{n \max(\dim W_{i'}- \dim W_i,0)}
      \otimes
      \end{NB}
      f_i^{n \min(\dim W_i,\dim W_{i'})}.
  \end{equation*}
\end{Corollary}
\begin{NB}
The following is wrong.

  Therefore $L(W)$ is \/{\rm real} in the sense of \subsecref{subsec:HerLec}
  if and only if there are no non-isotropic imaginary Schur roots in the
  canonical decomposition of $\lsp{\varphi}W$.
\end{NB}

Mimicking the definition in \subsecref{subsec:HerLec}, we say $\mathbb
L(W)$ is {\it real\/} if $\mathbb L(2W)\cong \mathbb L(W)\otimes
\mathbb L(W)$.
The above implies $\mathbb L(W)$ is {\it real\/} in this sense if and
only if there are no non-isotropic imaginary Schur roots in the
canonical decomposition of $\lsp{\varphi}W$.

If $L(2W) \cong L(W)\otimes L(W)$ (i.e.\ $L(W)$ is real), we have
$\ext^1(W,W) = 0$ by \propref{prop:Schur}(2). By the result of
Schofield \cite{Schofield} used above, this can happen only when the
canonical decomposition does not contain non-isotropic imaginary Schur
root.
This is a step towards proving that $\mathscr C_1$ is a monoidal
categorification.

\section{Cluster algebra structure}\label{sec:cluster}

In this section we prove that cluster monomials are dual canonical
base elements after some preparation.

In the previous sections, we use the notation $W$ for an $(I\sqcup
I_\fr)$-graded representation. In this section we also use it for its
general representation.
Or if we first take a representation, its underlying $(I\sqcup
I_\fr)$-graded vector space will be denoted by the same notation.

\subsection{Tilting modules}

We first review the theory of tilting modules. (See \cite[VI]{ASS} and
\cite{HappelUnger}.)

Let $\cQ = (I,\Omega)$ be a quiver as in \secref{sec:cluster_pre}. Let
$\C\cQ$ be its path algebra defined over $\C$.
We consider the category $\rep\cQ$ of finite dimensional
representations of $\cQ$ over $\C$, which is identified with
the category of finite dimensional $\C\cQ$-modules.

A module $M$ of the quiver is said to be a {\it tilting module\/} if
the following two conditions are satisfied:
\begin{enumerate}
\item $M$ is {\it rigid}, i.e.\ $\Ext^1(M,M) = 0$.
\item There is an exact sequence $0\to \C\cQ \to M_0 \to M_1\to 0$
  with $M_0$, $M_1\in\add M$, where $\add M$ denotes the additive
  category generated by the direct summands of $M$.
\end{enumerate}
We usually assume $M$ is multiplicity free.

It is known that the number of indecomposable summands of $M$
equals to the number of vertexes $\# I$, i.e.\ rank of $K_0(\C\cQ)$.

A rigid module $M$ always has a module $X$ so that $M\oplus X$ is a
tilting module.

A module $M$ is said to be an {\it almost complete tilting module\/}
if it is rigid and the number of indecomposable summands of $M$
is $\# I - 1$. We say an indecomposable module $X$ is {\it
  complement\/} of $M$ if $M\oplus X$ is a tilting module.

We have the following structure theorem:

\begin{Theorem}[\protect{Happel-Unger~\cite{HappelUnger}}]
  Let $M$ be an almost complete tilting module.

  \textup{(1)} If $M$ is sincere, there exists two nonisomorphic
  complements $X$, $Y$ which are related by an exact sequence
  \begin{equation*}
     0 \to X \to E\to Y \to 0
  \end{equation*}
  with $E\in\add M$. Moreover, we have $\Ext^1(Y,X) \cong \C$,
  $\Ext^1(X,Y) = 0$, $\Hom(Y,X) = 0$.

  \begin{NB}
    I have thought that the above is an almost splitting sequence, but
    it is not correct.
  \end{NB}

  \textup{(2)} If $M$ is not sincere, there exists only one
  complement $X$ up to isomorphism.
\end{Theorem}

Here a module $M$ is said to be {\it sincere\/} if $M_i\neq 0$ for any
vertex $i$.

\subsection{Cluster tilting sets}\label{subsec:tilt}

When the quiver $\cQ$ does not contain an oriented cycle (i.e.\
acyclic quiver), combinatorics of the cluster algebra can be
understood from the cluster category theory. Since we only need the
statement, we explain the theory only very briefly following
\cite{Hubery}. We only consider the case when there are no frozen
variables.

Let $n = \# I$. A collection $\mathbf L = \{ W^1,\dots, W^n\}$ is said
to be a {\it cluster-tilting set\/} if the following conditions are
satisfied:
\begin{enumerate}
\setcounter{enumi}{-1}
\item $W^i$ is either an indecomposable representation of the quiver
  $\cQ$ or a vertex. Let $\mathbf L_{\mathrm{mod}}$ be the subset of
  indecomposable representations, $\mathbf L_{\mathrm{ver}} = \mathbf
  L\setminus \mathbf L_{\mathrm{mod}}$.

\item $W^k\in\mathbf L_{\mathrm{mod}}$ are pairwise
  nonisomorphic. $W^i\in\mathbf L_{\mathrm{ver}}$ are pairwise
  distinct.

\item Delete all arrows incident to a vertex $W^i\in\mathbf
  L_{\mathrm{ver}}$. Remove the vertex $W^i$. Let $\lsp{\psi}\cQ$ be
  the resulting quiver.

\item The entry for $W^k\in\mathbf L_{\mathrm{mod}}$ is $0$ for a
  vertex $W^j\in\mathbf L_{\mathrm{ver}}$. Hence $W^k$ is a
  representation of $\lsp{\psi}\cQ$.

\item $\lsp{\psi}W \defeq \bigoplus_{W^k\in\mathbf
    L_{\mathrm{mod}}}W^k$ is a tilting module as a representation of
  $\lsp{\psi}\cQ$.
\end{enumerate}

Note that $\# \mathbf L_{\mathrm{mod}} = \#
(\lsp{\psi}{I})$. Therefore $\lsp{\psi}W$ is tilting if and only if
$\ext^1(W^k,W^l) = 0$ for any $k$, $l$ (including the case
$k=l$). Thus this is stronger than the canonical decomposition and
means that $\dim W^k$ is a real Schur root.

The initial cluster-tilting set is the collection $\mathbf L = I$ with
$\mathbf L_{\mathrm{mod}}=\emptyset$. In this case $\lsp{\psi}{I} =
\emptyset$ and the condition is trivially satisfied.

If we identify $W^i\in\mathbf L_{\mathrm{ver}}$ with $P_{W^i}[1]$ the
shift of the indecomposable projective module associated corresponding
to the vertex $W^i$, the above definition is nothing but the
definition of a cluster-tilting set for the cluster category
\cite{BMRRT}.

For $k\in \{1,\dots, n\}$ we define the {\it mutation\/} $\mu_k(\bL)$
of $\bL$ in direction $k$ as follows:
\begin{enumerate}
\item Suppose $W^k$ is a vertex. We add it again, together with all
  arrows incident to it, to the quiver $\lsp{\psi}\cQ$. Let
  $\lsp{+\psi}\cQ$ be the resulting quiver. Since $\lsp{\psi}W$ is an
  almost tilting non-sincere module as a representation of
  $\lsp{+\psi}\cQ$, we can add the unique indecomposable $\lsp{*}W^k$
  to $\lsp{\psi}W$ to get a tilting module.

\item Next suppose $W^k$ is a module. We consider an almost tilting
  module $\lsp{-\psi}W$ which is obtained from $\lsp{\psi}W$ by
  subtracting the summand $W^k$.
  \begin{aenume}
  \item If it is sincere, there is another indecomposable module
    $\lsp{*}W^k\neq W^k$ such that $\lsp{*}W^k\oplus \lsp{-\psi}W$ is a
    tilting module.

  \item If it is not sincere, there exists the unique simple module
    $S_i$, not appearing in the composition factors of
    $\lsp{-\psi}W$. Then we set $\lsp{*}W^k = {i}$.
  \end{aenume}
\end{enumerate}
Let
  \begin{equation*}
     \mu_k(\bL) \defeq \bL \cup \{\lsp{*}W^k\}
     \setminus \{ W^k \}.
  \end{equation*}

  In all cases $\mu_k(\bL)$ is again a cluster-tilting set.
  We can iterate this procedure and obtain new clusters starting from
  the initial cluster $\bL = I$.

\subsection{Cluster character}\label{subsec:cc}
We still continue to assume that the quiver $\cQ$ does not contain an
oriented cycle.
It is known that cluster monomials can be expressed in terms of
generating functions of Euler numbers of quiver Grassmannian
varieties.
This important result was first proved by Caldero-Chapoton in type
$ADE$ \cite{CalderoChapoton}. Later it was generalized to any acyclic
quiver by Caldero-Keller~\cite{Caldero-Keller} using various results
in the cluster category theory (see \cite{Keller} for the
reference). We recall the formula in this subsection.

Let $(\bx,\bB)$ be the initial seed of the cluster algebra $\mathscr
A(\bB)$. We assume there is no frozen part for simplicity.
Let $W$ be a representation of the quiver $\cQ$ corresponding to
$\bB$. Let $\Gr_V(W)$ be the corresponding quiver Grassmannian
variety, where $V$ is an $I$-graded vector space. Though we soon
assume $W$ is a general representation in $\bE_W$, it is not necessary
for the definition. Let $e(\Gr_V(W))$ be its Euler number. We define
\begin{equation*}
   X_W \defeq
   \frac1{\bx^{\dim W}}
   \sum_V e(\Gr_V(W)) \bx^{\dim V\cdot R}\bx^{(\dim W - \dim V)R'},
\end{equation*}
where
\begin{equation*}
  \begin{gathered}
  \bx^{\dim W} = \prod_i x_i^{\dim W_i},
\\
  \bx^{\dim V\cdot R} = \prod_{h\in\Omega} x_{\vin(h)}^{\dim V_{\vout(h)}},
\quad
  \bx^{(\dim W - \dim V)R'}
  = \prod_{h\in\Omega} x_{\vout(h)}^{(\dim W_{\vin(h)} - \dim V_{\vin(h)})}.
  \end{gathered}
\end{equation*}

For a vertex $i$, we set $X_i = x_i$.

Then it is known that the correspondence $W\to X_W$ gives the
followings:
\begin{itemize}
\item the correspondence $W\to X_W$ defines a bijection between the
  set of isomorphism classes of rigid indecomposable modules with
  cluster variables minus $\{ x_i\}$;
\item the correspondence $\mathbb L\to \{ X_{W^1},\dots, X_{W^n}\}$ gives
  a bijection between cluster tilting sets and clusters;
\item the mutation on cluster tilting sets corresponds to the cluster
  mutation.
\end{itemize}

\subsection{Piecewise-linear involution}

We give one more preparation before applying results from the cluster
category theory to our setting. This last preliminary is not necessary
for our argument, but helps to make a relation to
\cite[\S12.3]{HerLec}.

We recall the piecewise-linear involution $\tau_-$ on the root lattice
considered in \cite[\S7]{HerLec}: for $\gamma = \sum_i \gamma_i
i\in\Z^I$, we define $\tau_-(\gamma) = \sum_i \tau_-(\gamma)_i i$ by
\begin{equation}\label{eq:tau_-}
  \tau_-(\gamma)_i =
  \begin{cases}
    - \gamma_i - \sum_{j\neq i} c_{ij} \max(0, \gamma_j) & \text{if $i\in I_1$},
\\
    \gamma_i & \text{if $i\in I_0$},
  \end{cases}
\end{equation}
where $(c_{ij})$ is the Cartan matrix.

Let
\begin{equation}\label{eq:gamma}
    \gamma = \sum_i (\dim W_i - \dim W_{i'}) i.  
\end{equation}
If $i\in I_0$, we have
\begin{equation*}
   \tau_-(\gamma)_i = \dim W_i - \dim W_{i'}
   = \dim \lsp{\varphi}W_i - \dim \lsp{\varphi}W_{i'}.
\end{equation*}
If $i\in I_1$, we have
\begin{equation*}
  \begin{split}
   \tau_-(\gamma)_i & = 
   \dim W_{i'} - \dim W_i - \sum_{j\neq i} c_{ij} \max(\dim W_j - \dim W_{j'},0)
\\
   & = \dim \lsp{\varphi}W_{i'} - \dim \lsp{\varphi}W_i
   - \sum_{j\neq i} c_{ij} \dim \lsp{\varphi}W_j.
  \end{split}
\end{equation*}
Therefore we have
\begin{equation*}
  \dim \lsp{\sigma\varphi}W_i
  = \max(\tau_-(\gamma)_i, 0).
\end{equation*}
where $\lsp{\sigma\varphi}W = \lsp{\sigma}(\lsp{\varphi}W)$ is obtained by
applying $\sigma$ to $\lsp{\varphi}W$. 
\begin{NB}
We do not have a simple relation between
\(
  \dim \lsp{\sigma\varphi}W_{i'}
\)
and
$\tau_-(\gamma)$.
\end{NB}

\begin{Remark}
  In \cite[\S12.3]{HerLec} the quiver Grassmannian
  $\Gr_V(M[\tau_-(\gamma)])$ was considered where $M[\tau_-(\gamma)]$
  is a general representation with
\begin{equation*}
  \dim M[\tau_-(\gamma)]_i = \max(\tau_-(\gamma)_i, 0).
\end{equation*}
Here the quiver is the principal part $\cQ$ of our decorated quiver.
From the above computation
\(
  M[\tau_-(\gamma)]
\)
is nothing but the principal quiver part of $\lsp{\sigma\varphi}W$.
The frozen part of $\lsp{\sigma\varphi}W$ does not play any role in
the quiver Grassmannian, by \propref{prop:fac-1}. Therefore
\(
   \Gr_V(M[\tau_-(\gamma)])
\)
in [loc.\ cit., \S12.3] is isomorphic to our
\(
   \Gr_V(\lsp{\sigma\varphi}W)
\)
under \eqref{eq:gamma}.
\end{Remark}

\subsection{Cluster monomials}

We start to put the cluster algebra structure on $\bfR$ from this
subsection. 

\begin{Proposition}\label{prop:variable}
  \textup{(1)} Let $W$ be an $I$-graded vector space such that $\dim
  W$ is a real Schur root of the principal part of the decorated
  quiver. Then $L(W)$ is a cluster variable.

  \textup{(2)} This correspondence defines a bijection between the set
  of real Schur roots and the set of cluster variables except
  variables in the initial seed, i.e.\ $x_i$, $f_i$ \textup(${i\in
    I}$\textup).
\end{Proposition}

For type $ADE$, this together with \corref{cor:ADE} shows the
condition (2) in the monoidal
categorification~\ref{def:categorification}.

\begin{proof}
  Roughly this is a consequence of results reviewed in
  \subsecref{subsec:cc}. However, our quiver Grassmannian is for
  $\lsp{\sigma}W$, not for $W$. Correspondingly we need to replace the
  initial seed of $\mathscr A(\widetilde \bB)$ by the $\mathbf
  z$-quiver in \eqref{eq:z}.
  When we mutate from $\bx$-quiver to $\mathbf z$-quiver, the set of
  cluster variables does not change by definition, but variables in
  the initial seed change. So let us first consider this effect.
  The functor $\lsp{\sigma}(\bullet)$ induces an involution on the
  set 
\[
  \{ \text{real Schur roots}\}\setminus \{ \alpha_i \mid i\in I_1\}.
\]
Therefore we only need to study cluster variables corresponding to
$\alpha_i$ in either $\bx$-quiver or $\mathbf z$-quiver.

\begin{itemize}
\item 
In $\bx$-quiver, $\alpha_i$ corresponds to $W = S_i$. We have $L(S_i)
= x_i' = z_i$. This is a cluster variable of the seed for the $\mathbf
z$-quiver, but not for the original $\bx$-quiver.
Note also that $\lsp{\sigma}W = 0$ in this case.

\item
In $\mathbf z$-quiver, $\alpha_i$ corresponds to the cluster variable
obtained as $z_i^*$. But this is nothing but $x_i$. The corresponding
simple module is $L(S_{i'})$. We do not consider since it has support
in the frozen part.
\end{itemize}
We now may assume $\dim \lsp{\sigma}W$ is a real Schur root different
from $\alpha_i$ ($i\in I_1$).

We cannot apply the formula in \subsecref{subsec:cc} directly as the
$\mathbf z$-quiver contains an oriented cycle in general. (See
\eqref{eq:z-quiver}.) We thus first consider the quiver with principal
coefficients, and write down $F$-polynomials and $\mathbf g$-vectors
by using the formula in \subsecref{subsec:cc}. Then we apply the
result in \subsecref{subsec:F-pol} to get the formula for cluster
variables in the original cluster algebra.

We take $\mathbf u$, $\mathbf f$ as cluster variables for the initial
seed of $\mathscr A_\pr$ and define
\begin{equation*}
  X_{\lsp{\sigma}W}(\mathbf u,\mathbf f) \defeq
  \frac1{
  \prod_{i\in I} u_i^{\lsp{\sigma}w_i}}
    \sum_V e(\Gr_V(\lsp{\sigma}W))
  \prod_{i\in I_0} u_i^{\sum_j a_{ij} v_j}
  \prod_{i\in I_1} u_i^{\sum_j a_{ij} (\lsp{\sigma}w_j - v_j)}
  \prod_{i\in I} f_i^{v_i},
\end{equation*}
where $v_i = \dim V_i$, $w_i = \dim W_i$, $\lsp{\sigma}w_i = \dim
\lsp{\sigma}W_i$.
By \subsecref{subsec:cc} this is a cluster variable $\alpha$ for
$\mathscr A_\pr$, and hence above gives the Laurent polynomial
$X_\alpha(\mathbf u,\mathbf f)$ in \subsecref{subsec:F-pol}.

Hence the $F$-polynomial is
\begin{equation*}
   F_{\lsp{\sigma}W}(\mathbf f) = \sum_V e(\Gr_V(\lsp{\sigma}W))
  \prod_{i\in I} f_i^{v_i}.
\end{equation*}
And the $\mathbf g$-vector is
\begin{NB}
\begin{equation*}
  \sum_{i\in I_0} \sum_j a_{ij} v_j i 
  + \sum_{i\in I_1} \sum_j a_{ij} (\lsp{\sigma} w_j - v_j)i 
  - \sum_{i,j} \varepsilon_i a_{ij} v_j i
  - \sum_i \lsp{\sigma}w_i i =
\end{equation*}
\end{NB}%
\begin{equation*}
   \mathbf g_{\lsp{\sigma}W} =
   - \sum_{i\in I_0} \lsp{\sigma}w_i i
   - \sum_{i\in I_1} \left(\lsp{\sigma} 
     w_i - \sum_j a_{ij} \lsp{\sigma}w_j\right) i
   \begin{NB}
   = 
   \begin{cases}
   0 & \text{if $W = S_i$ for $i\in I_1$},
     \\    
   - \sum_i w_i \varepsilon_i i & \text{otherwise}
   \end{cases}
   \end{NB}
   = - \sum_i w_i \varepsilon_i i,
\end{equation*}
where $\ve_i = (-1)^{\xi_i}$.

Now we return back to our original cluster algebra. Since our initial
seed is given by the $\mathbf z$-quiver, we change the notation in 
\subsecref{subsec:F-pol} and use $z$-variables instead of $x$-variables.
We denote the cluster variable corresponding to above
$X_{\lsp{\sigma}W}$ by $z[{\lsp{\sigma}W}]$. We have
\begin{equation*}
  z[{\lsp{\sigma}W}] = 
  \frac{F_{\lsp{\sigma}W}(\widehat{\mathbf y})}
  {\left.F_{\lsp{\sigma}W}\right|_{\mathbb P}(\mathbf y)}
  {\mathbf z}^{\mathbf g_\alpha},
\end{equation*}
where
\begin{equation*}
  y_j = 
  \begin{cases}
  f_j^{-1} \prod_{i\in I} f_i^{a_{ij}} & \text{if $j\in I_0$},
\\
  f_j^{-1} & \text{if $j\in I_1$},
  \end{cases}
\qquad
  \widehat y_j = y_j \prod_{i\in I} z_i^{\ve_i a_{ij}}
\qquad
  (j\in I).
\end{equation*}
in this situation.
A direct calculation shows (see \cite[Lem.~7.2]{HerLec})
\begin{equation*}
  \chi_{\vq}(\widehat y_j) = V_{j,\vq^{\xi_j+1}}.
\end{equation*}
\begin{NB}
First suppose $j\in I_1$. We have
  \begin{equation*}
     \widehat y_j = f_j^{-1} \prod_i z_i^{a_{ij}}
     = Y_{j,\vq^1}^{-1} Y_{j,\vq^3}^{-1} \prod_j Y_{i,\vq^2}^{a_{ij}}
     = V_{i,\vq^2}.
  \end{equation*}
Next suppose $j\in I_0$. We have
\begin{equation*}
     \widehat y_j = f_j^{-1} \prod_i f_i^{a_{ij}} z_i^{-a_{ij}}
     = Y_{j,1}^{-1} Y_{j,\vq^2}^{-1} \prod_j Y_{i,\vq}^{a_{ij}}
     = V_{i,\vq}.
\end{equation*}
\end{NB}%
We note that $F_{\lsp{\sigma}W}$ contains the monomial $\prod_i
f_i^{\lsp{\sigma}w_i}$ for $V = \lsp{\sigma}W$ with the coefficient
$1$, and all other terms are its factor. If we evaluate it
at $y_j$, we have
\begin{equation*}
  \prod_{i\in I} f_i^{-\lsp{\sigma}w_i}
  \prod_{i\in I_1} f_i^{\sum a_{ij} \lsp{\sigma}w_j}
  = \prod_{i\in I_0} f_i^{-\lsp{\sigma}w_i}
  \prod_{i\in I_1} f_i^{-\lsp{\sigma}w_i+\sum a_{ij} \lsp{\sigma}w_j}
  = \prod_{i\in I_0} f_i^{-w_i} \prod_{i\in I_1} f_i^{w_i}.
\end{equation*}
We also have the constant term $1$ for $V = 0$. Therefore
\begin{equation*}
  {\left.F_{\lsp{\sigma}W}\right|_{\mathbb P}(\mathbf y)}
  = \prod_{i\in I_0} f_i^{-w_i}.
\end{equation*}
Thus combining with the above calculation of $\mathbf
g_{\lsp{\sigma}W}$, we get (\cite[Lem.~7.3]{HerLec})
\begin{equation*}
  \frac{\mathbf z^{\mathbf g_{\lsp{\sigma}W}}}
  {\left.F_{\lsp{\sigma}W}\right|_{\mathbb P}(\mathbf y)}
  = \prod_{i\in I_0} f_i^{w_i} \prod_{i\in I} z_i^{-w_i\ve_i}.
\end{equation*}
Its $\vq$-character is
\begin{equation*}
  \chi_{\vq}\left(
    \frac{\mathbf z^{\mathbf g_{\lsp{\sigma}W}}}
  {\left.F_{\lsp{\sigma}W}\right|_{\mathbb P}(\mathbf y)}\right)
   = \prod_{i\in I_0} Y_{i,1}^{w_i} \prod_{i\in I_1} Y_{i,\vq^3}^{w_i}.
\end{equation*}
We thus get
\begin{equation*}
   \chi_{\vq}(z[\lsp{\sigma}W])_{\le 2}
   = \sum_V e(\Gr_V(\lsp{\sigma}W))\,
   e^W e^V.
\end{equation*}
Hence we have $z[\lsp{\sigma}W] = {\mathbb L}(W) = L(W)$, where the
first equality follows from \thmref{thm:main} and the second equality
from \propref{lem:real}.
\end{proof}

\begin{Proposition}\label{prop:monomial}
  Let $L(W^1)$, \dots, $L(W^s)$ be simple modules corresponding to
  cluster variables $w_1$, \dots, $w_s$ \textup(either via
  \propref{prop:variable} or $x_i$, $f_i$\textup). Then $L(W^1)\otimes
  \cdots\otimes L(W^s)$ is simple if and only if all $w_1$, \dots,
  $w_s$ live in a common cluster.
\end{Proposition}

For type $ADE$, this shows the condition (1) in the monoidal
categorification~\ref{def:categorification}.

\begin{proof}
  The assertion is trivial for the factor $f_i$ by
  \propref{prop:fac}. So we may assume any $W^1$,\dots, $W^s$ is not
  $f_i$. Therefore we have $W^1 = \lsp{\varphi}W^1$, \dots, $W^s =
  \lsp{\varphi}W^s$.

  By Propositions~\ref{prop:canfac},\ref{prop:Schur} $L(W^1)\otimes
  \cdots\otimes L(W^s)$ is simple if and only if $\ext^1(W^k,W^l) =
  \ext^1(W^l,W^k) = 0$ for $k\neq l$.
  Thus we need to show that this is equivalent to the condition that
  the corresponding $w_k$ and $w_l$ are in a common cluster.
  Therefore we may assume $k=1$, $l=2$.

  When $W^1 = W^2$, then $w_1 = w_2$ is in a common cluster. But
  $\ext^1(W^1,W^2) = 0$ is also true since $\dim W^1 = \dim W^2$ is a
  real Schur root.
  
  If neither $L(W^1)$ nor $L(W^2)$ is one of $x_i$ and $x_i'$, then
  $L(W^1) = z[\lsp{\sigma}W^1]$, $L(W^2) = z[\lsp{\sigma}W^2]$ as in
  the proof of \propref{prop:variable}. We have $\ext^1(W^1,W^2) =
  \ext^1(W^2,W^1) = 0$ if and only if
  $\ext^1(\lsp{\sigma}W^1,\lsp{\sigma}W^2) =
  \ext^1(\lsp{\sigma}W^2,\lsp{\sigma}W^1) = 0$.
  This happens if and only if $\lsp{\sigma}W^1\oplus \lsp{\sigma}W^2$
  is rigid, and hence can be extended to a tilting module.
  From \S\S\ref{subsec:tilt},\ref{subsec:cc}, this is equivalent to that
  the corresponding cluster variables live in a common cluster.

  If $L(W^1) = x_i$, $L(W^2) = x'_i$, then $L(W^1)\otimes L(W^2)$ is
  not simple by the $T$-system \eqref{eq:T-system}. They are not in
  any cluster simultaneously.
  Any other pairs from $x_i$, $x'_j$, they are always in a common
  cluster. It is also clear that $L(W^1)\otimes L(W^2)$ is always
  simple.
  Therefore we may assume $L(W^1)$ is one of $x_i$, $x'_i$, and
  $L(W^2)$ is not.

  Consider the case $L(W^1) = x_i$ with $i\in I_0$. We have $W^1 =
  S_{i'}$. From Propositions~\ref{prop:fac},\ref{prop:fac-1}
  $L(S_{i'})\otimes L(W^2)$ is simple if and only if $W^2_i = 0$.  In
  this case $x_i = z_i$ is a cluster variable from the seed for
  $\mathbf z$-quiver. From \subsecref{subsec:tilt} the cluster
  variable $w_2$ is in a common cluster with $z_i$ if and only if
  $\lsp{\sigma}W^2_i = 0$. This is equivalent to $W^2_i = 0$, since
  $i\in I_0$.

  The case $L(W^1) = x'_i$ with $i\in I_0$ is not necessary to
  consider since we have $L(W^1) = L(S_{i}) = z[S_i]$, which is
  already studied.

  Next suppose $L(W^1) = x_i$ with $i\in I_1$. We have $W^1 =
  S_{i'}$. From Propositions~\ref{prop:fac},\ref{prop:fac-1}
  $L(S_{i'})\otimes L(W^2)$ is simple if and only if $W^2_i = 0$ as
  above. Since $i$ is a source, this is equivalent to $\Hom(W^2,S_i) =
  0$. From the definition of the reflection functor, it is equivalent
  to $\Ext^1(S_i, \lsp{\sigma}W^2) = 0$.
  \begin{NB}
    We have $\lsp{\sigma}S_i = S_i[-1]$. Therefore
    $\Hom(W^2,S_i) = D\Hom(\lsp{\sigma}S_i, \lsp{\sigma}W^2)
    = D\Ext^1(S_i,\lsp{\sigma}W^2)$.
  \end{NB}%
  Since we have $x_i = z_i^*$, the corresponding rigid module for the
  $\mathbf z$-quiver is $S_i$. Therefore $x_i$ and $w$ is in a common
  cluster if and only $\Ext^1(S_i,\lsp{\sigma}W^2) = 0 =
  \Ext^1(\lsp{\sigma}W^2, S_i)$ by \subsecref{subsec:tilt}. But the
  latter equality is trivial since $i$ is source. Thus we have checked
  the assertion in this case.

  Finally suppose $L(W^1) = x'_i$ for $i\in I_1$. This is $z_i$ and
  corresponds to a vertex $i$ in the cluster-tilting set for $\mathbf
  z$-quiver. Therefore $w$ is in a same cluster with $z_i$ if and only
  if $\lsp{\sigma}W^2_i = 0$. By the same argument as above, 
  this is equivalent to $\Ext^1(S_i,W^2) = 0 = \Ext^1(W^2,S_i)$.
  Thus we have checked the final case.
\end{proof}

\begin{Remark}\label{rem:derived}
  As indicated in the proof, it is more natural to define
  $\lsp{\sigma}S_i$ as $S_i[-1]$, an object in the derived category
  $\mathscr D(\rep\lsp{\sigma}{\!\widetilde\cQ^{\mathrm{op}}})$. This
  is also compatible with the cluster category theory, as $S_i[-1] =
  I_i[-1]$ for $i\in I_1$, where $I_i$ is the indecomposable injective
  module corresponding to the vertex $i$.
\end{Remark}

\subsection{Exchange relation}

Consider an exchange relation \eqref{eq:exchange2}. Thanks to
Propositions~\ref{prop:variable}, \ref{prop:monomial} we have the
corresponding equality in $\bfR_{\ell=1}$:
\begin{equation*}
   L(x_k)\otimes L(x_k^*) = L(m_+) + L(m_-).
\end{equation*}

Since $L(m_\pm)$ are simple, this inequality in the Grothendieck group
implies either of the followings:
\begin{equation*}
  0 \to L(m_+) \to L(x_k)\otimes L(x_k^{*}) \to L(m_-)\to 0,
\end{equation*}
or
\begin{equation*}
  0 \to L(m_-) \to L(x_k)\otimes L(x_k^{*}) \to L(m_+)\to 0
\end{equation*}
in the level of modules.
It is natural to conjecture that we always have the above one.
For the $T$-system, this is true thanks to \remref{rem:T-system}.

This conjecture follows from a refinement of the exchange relation:
\begin{equation*}
 \chi_{\vq,t}(L(x_k)\otimes L(x_k^*))
 = t^{-l+n} \chi_{\vq,t}(L(m_+)) + t^n\chi_{\vq,t}(L(m_-))
\end{equation*}
for some $l > 0$, $n\in\Z$.
If we write the corresponding perverse sheaves by $P(x_k)$,
$P(x_k^*)$, $P(m_+)$, $P(m_-)$, the above means that 
\begin{equation*}
  \begin{split}
  \operatorname{Res}(P(m_+)) 
  & = P(x_k)\boxtimes P(x_k^*)[l-n] \oplus \cdots,
\\
  \operatorname{Res}(P(m_-)) 
  &= P(x_k)\boxtimes P(x_k^*)[-n] \oplus \cdots,
  \end{split}
\end{equation*}
where $\cdots$ means sum of (shifts of) other perverse sheaves.
Since $\Hom(P(x_k)\boxtimes P(x_k^*)[l], P(x_k)\boxtimes P(x_k^*))$
vanishes for $l>0$ by a property of perverse sheaves
\cite[8.4.4]{CG}, we see that $L(m_+)$ is a submodule
of $L(x_k)\otimes L(x_k^*)$.

This refinement of the exchange relation might be proved directly, but
it should be proved naturally if we make an isomorphism of the quantum
cluster algebra \cite{BerZel} with $\bfR_{t,\ell=1}$.

\begin{NB}
\begin{equation*}
  \begin{split}
   & d(\lsp{\sigma}W^2,W^2;0,W^1) - d(0,W^1;\lsp{\sigma}W^2,W^2)
\\
   =\;&
   \boldsymbol\langle \dim \lsp{\sigma}W^2,\vq^{-1}\dim W^1\boldsymbol\rangle
   - \boldsymbol\langle \dim \lsp{\sigma}W^2,\vq \dim W^1\boldsymbol\rangle
\\
   =\; &
   \sum_{i\in I_0} \dim W^2_i(1) \dim W^1_i(1)
   - \sum_{i\in I_1} \dim \lsp{\sigma}W^2_i(\vq^3) \dim W^1_i(\vq^3)
\\   
   =\; &
   \sum_{i\in I_0} \dim W^1_i(1)\dim W^2_i(1) 
   + \sum_{i\in I_1} \dim W^1_i(\vq^3) \dim W^2_i(\vq^3) 
   - \sum_{i\in I_1} a_{ij} \dim W^1_i(\vq^3) \dim W^2_j(1) 
\\
   =\; &
   \chi(W^1,W^2)
  \end{split}
\end{equation*}
\end{NB}

\begin{NB}
  \begin{NB2}
    The following section is far from complete.....
  \end{NB2}

\section{Exchange relation}\label{sec:exchange}

In this section we use several results from the representation theory
of quivers to prove the exchange relation and construct a cluster
structure on $\bfR_{\ell=1}$.

\subsection{Rigid modules}

Let $W$ be a representation of the decorated quiver in
\defref{def:decorated}. Let $\bE_W$ be the corresponding vector space
as in \propref{prop:ast_1}. We have an action of $\prod \GL(W_i(a))$
on $\bE_W$. Let $\mathcal O_W$ be the orbit through $W$ in
$\bE_W$. The following property holds for arbitrary quiver and
well-known:

\begin{Proposition}
The followings are equivalent:
\begin{enumerate}
\item The orbit $\mathcal O_W$ is a Zariski open subset of $\bE_W$;

\item $W$ is rigid, i.e.\ $\Ext^1(W,W) = 0$.

\item $W$ is a general representation of $\bE_W$ such that the
  canonical decomposition contains only \/{\rm real} Schur roots.
\end{enumerate}
\end{Proposition}

In fact, we have
\begin{equation*}
  \codim \mathcal O_W = \dim \Ext^1(W,W).
\end{equation*}
\begin{NB2}
  This holds for an arbitrary quiver. We have
  \begin{equation*}
    \begin{split}
    & \dim \mathcal O_W = \dim \prod \GL(W_i(a)) 
      - \dim \operatorname{Stab}(W)
      = \dim \prod \GL(W_i(a)) - \dim \Hom(W,W),
\\
    & \dim \bE_W = \sum_h \dim W_{\vin(h)} \dim W_{\vout(h)}
    = \Ext^1(\{ B=0\}, \{B = 0\})
    = - \chi(W,W) + \dim \prod \GL(W_i(a)),
    \end{split}
  \end{equation*}
  where $\chi(W,W)$ is the Euler characteristic $\dim \Hom(W,W) - \dim
  \Ext^1(W,W) = \dim \Hom(\{B=0\},\{B=0\})$. Therefore
\begin{equation*}
  \codim\mathcal O_W = -\chi(W,W) + \dim \Hom(W,W) = \dim \Ext^1(W,W).
\end{equation*}
\end{NB2}%
Since $\bE_W$ is irreducible and smooth, the equivalence between (1)
and (2) is clear. If $\shfO_W$ is a Zariski open subset in $\bE_W$,
$W$ is a general representation in $\bE_W$. If $W^k$ is a factor of
the canonical decomposition of $W$, we have $\Ext^1(W^k,W^k) = 0$.
It can be shown $\Hom(W^k,W^k) \cong \C$. *** Therefore
\begin{equation*}
  \dim \Hom(W^k,W^k) - \dim \Ext^1(W^k,W^k) = 1.
\end{equation*}
Since this is equal to the half of the length of the root $\dim W^k$,
it means that $\dim W^k$ is real.
Converse is also clear.


\subsection{Multiplication formulas}

Suppose that
\begin{equation}\label{eq:extmp}
  0 \to W^1 \xrightarrow{\alpha} W \xrightarrow{\beta} W^2 \to 0
\end{equation}
is a nonsplit exact sequence of representations of the quiver in
\eqref{eq:ori'}. Let $\Gr_{V}(W)$ be the quiver Grassmannian
consisting of submodules $X$ of $W$ with $\dim V = \dim X$, and the
same for $\Gr_{V^1}(W^1)$, $\Gr_{V^2}(W^2)$.

We divide the quiver Grassmannian $\Gr_{V}(W)$ as
\begin{equation*}
  \begin{split}
   &\Gr_{V}(W) = \bigsqcup_{V=V^1\oplus V^2}
   \Gr_{V^1,V^2}(W),
\\
   & \Gr_{V^1,V^2}(W) = \{ X\in \Gr_{V}(W) \mid
   \dim (X\cap \Ima\alpha) = \dim V^1\},
  \end{split}
\end{equation*}
We have a morphism
\begin{equation*}
  \begin{split}
   \Pi\colon \Gr_{V^1,V^2}(W) &\to 
   \Gr_{V^1}(W^1)\times \Gr_{V^2}(W^2),
\\
   X&\mapsto (\alpha^{-1}(X),\beta(X)).
  \end{split}
\end{equation*}

\begin{Lemma}[\protect{\cite[Lemma~3.11]{CalderoChapoton}}]\label{lem:CC}
  Assume \eqref{eq:extmp} is an almost split sequence and $W^1$ is
  rigid, i.e.\ $\Ext^1(W^1,W^1) = 0$.
  Then $\Gr_{V^1,V^2}(W)$ is the empty set if $(V^1,V^2) = (0,W^2)$,
  and $\Pi$ is a fiber bundle whose fiber at $(X^1,
  X^2)\in\Gr_{V^1}(W^1)\times \Gr_{V^2}(W^2)$ is
  \(
    \Hom(X^2, W^1/X^1).
  \)
\end{Lemma}

We give a proof for a sake of the reader.

\begin{proof}
If $(V^1,V^2) = (0,W^2)$, then $\Pi^{-1}(0,W^2) = \emptyset$,
since \eqref{eq:extmp} splits otherwise.

Next suppose $(V^1,V^2)\neq (0,W^2)$. If $X^2 \subsetneq W^2$, then
from the almost splitting assumption, the short exact sequence $0\to
W^1 \to \beta^{-1}(X^2)\to X^2\to 0$ split (see \subsecref{subsec:almost}).
The assertion is clear.

So suppose $X^2 = W^2$. For a submodule $X^1\subsetneq W^1$, the
homomorphism $\Ext^1(W^2,W^1)\to \Ext^1(W^2,W^1/X^1)$ vanishes
(\subsecref{subsec:almost}). Therefore we have an exact sequence
\begin{equation*}
  \Hom(W^2, W^1) \to \Hom(W^2, W^1/X^1) \to
  \Ext^1(W^2,X^1)\to \Ext^1(W^2,W^1) 
  \to 0.
\end{equation*}
We have $\Hom(W^2, W^1) = D\Ext^1(W^1,\tau W^2) = D\Ext^1(W^1,W^1) =
0$ from the assumption.
Therefore \eqref{eq:extmp} can be lifted to an exact sequence $0 \to
X^1 \to X \to W^2\to 0$ and the choice is unique up to $\Hom(X^2,
W^1/X^1)$.
\end{proof}

If we have an exact sequence \eqref{eq:extmp} which is an almost
splitting sequence for a full subquiver, we have the same assertion.
Here a {\it full subquiver\/} $\cQ'\subset\cQ$ is a subset $I'\subset
I$ together with $\Omega' \defeq \{ h\in \Omega \mid
\vout(h),\vin(h)\in I'\}$.

\begin{NB2}
  In the following two propositions we still omit the $t$-deformation.
\end{NB2}

\begin{Proposition}\label{prop:ex1}
  Suppose that 
  \begin{equation*}
      0 \to W^1 \xrightarrow{\alpha} W \xrightarrow{\beta} W^2 \to 0
  \end{equation*}
  is an almost split sequence of representations of the principal
  part $\cQ$ of the decorated quiver.
  We assume $W^2$ is rigid and $W^2\neq S_i$ for any $i\in I_1$.
  Then
\begin{equation*}
   L(W^1)\otimes L(W^2) = L(W) + L(C^\bullet(\lsp{\sigma}W^1,W)),
\end{equation*}
where $C^\bullet(\ ,\ )$ is defined as in \eqref{eq:C(V,W)}.

\begin{NB2}
\textup{(2)} We assume $W^2 = S_i$ for $i\in I_1$. The the same
formula holds.
\begin{equation*}
   L(W^1)\otimes L(W^2) = L(W) + L(C^\bullet(\lsp{\sigma}W^1,W)),
\end{equation*}
\end{NB2}%
\end{Proposition}

The explicit computation of $L(C^\bullet(\lsp{\sigma}W^1,W))$ will be given
in \subsecref{subsec:compute}.

\begin{NB2}
\begin{Example}
  Consider type $A_3$. We write $L(1_0)$, $W(1_0)$, etc for the simple
  module in $\mathcal K(\mathscr Q_W)$ and a general representation of
  $\widetilde\cQ$.

(1) Consider the cluster mutation
\begin{equation*}
  \{ \alpha_3, \alpha_2 + \alpha_3, \alpha_1+\alpha_2+\alpha_3 \}
  \longleftrightarrow
  \{ \alpha_1+\alpha_2, \alpha_2 + \alpha_3, \alpha_1+\alpha_2+\alpha_3 \}.
\end{equation*}
We have the corresponding almost split sequence $0\to W(3_0)\to W\to
W(1_0 2_3) \to 0$.
Then
\begin{gather*}
  \chi_{\vq}(L(3_0))_{\le 2} = 3_0 + 3_2^{-1} 2_1 + 1_2 2_3^{-1},
\quad
  \chi_{\vq}(L(1_02_3))_{\le 2}  = 1_0 2_3 + 1_2^{-1} 2_1 2_3,
\\
  \chi_{\vq}(L(1_0 2_3 3_0))_{\le 2} 
  = 1_0 2_3 3_0 + 1_2^{-1} 2_1 2_3 3_0 + 1_0 2_1 2_3 3_2^{-1}
   + 1_2^{-1} 2_1^2 2_3 3_2^{-1} + 2_1,
\quad
  \chi_{\vq}(L(1_0 1_2))_{\le 2} = 1_0 1_2.
\end{gather*}
Thus we check
\begin{equation*}
   L(3_0) \otimes L(1_02_3) = L(W) + L(1_0 1_2).
\end{equation*}

(2) Consider the cluster mutation
\begin{equation*}
  \{ -\alpha_3,\alpha_1+\alpha_2,\alpha_1 \}
  \longleftrightarrow 
  \{ -\alpha_3,\alpha_1+\alpha_2,\alpha_2 \}.
\end{equation*}
We have $0\to W(1_0)\to W \to W(2_3) \to 0$. This is almost splitting
sequence for the quiver $1 \leftarrow 2$ ($3$ deleted). But it is not
for $1\leftarrow 2 \rightarrow 3$ as 
$W(1_0)\oplus W(2_3 3_0) \to W(2_3)$ does not factor through
$W = W(1_0 2_3)$.
Then
\begin{gather*}
  \chi_\vq(L(1_0))_{\le 2}  = 1_0 + 1_2^{-1} 2_1 + 2_3^{-1} 3_2,
\quad
  \chi_\vq(L(2_3))_{\le 2}  = 2_3,
\\
  \chi_\vq(L(1_02_3))_{\le 2}  = 1_0 2_3 + 1_2^{-1} 2_1 2_3,
\quad
  \chi_\vq(L(3_2))_{\le 2}  = 3_2.
\end{gather*}
Thus we check
\begin{equation*}
   L(1_0)\otimes L(2_3) = L(W) + L(3_2).
\end{equation*}

(3) Consider the cluster mutation
\begin{equation*}
  \{ \alpha_1 + \alpha_2, \alpha_2, \alpha_2+\alpha_3 \}
  \longleftrightarrow
  \{ \alpha_1+\alpha_2, \alpha_2 + \alpha_3, \alpha_1+\alpha_2+\alpha_3 \}.
\end{equation*}
We have $0\to W(1_02_33_0)\to W \to W(2_3) \to 0$. Then
\begin{gather*}
  \chi_{\vq}(L(1_0 2_3 3_0))_{\le 2} 
  = 1_0 2_3 3_0 + 1_2^{-1} 2_1 2_3 3_0 + 1_0 2_1 2_3 3_2^{-1}
   + 1_2^{-1} 2_1^2 2_3 3_2^{-1} + 2_1
\quad
  \chi_\vq(L(2_3))_{\le 2}  = 2_3,
\\
  \chi_\vq(L(1_02_3^2 3_0))_{\le 2}  = 1_0 2_3^2 3_0 + 1_2^{-1} 2_1 2_3^2 3_0
  + 1_0 2_1 2_3^2 3_2^{-1} + 1_2^{-1} 2_1^2 2_3^2 3_2^{-1},
\quad
  \chi_\vq(L(2_1 2_3))_{\le 2}  = 2_1 2_3.
\end{gather*}
Thus we check
\begin{equation*}
   L(1_0 2_3 3_0)\otimes L(2_3)
   = L(1_02_3^2 3_0) + L(2_1 2_3).
\end{equation*}
\end{Example}
\end{NB2}

\begin{proof}
  Note first that  $W^1$ cannot be $S_i$ since $i\in I_1$ is a source.

  We apply the functor $\sigma$ to the exact sequence. From the
  assumption, we have
  \begin{equation*}
    0\to \lsp{\sigma} W^2\to \lsp{\sigma}W \to \lsp{\sigma}W^1\to 0.
  \end{equation*}
  This is again an almost splitting sequence (\cite[VI.5.3]{ASS}). Now
  we apply \lemref{lem:CC} to get
  \begin{equation*}
     \chi_{\vq}(L(W^1))_{\le 2}\,\chi_{\vq}(L(W^2))_{\le 2}
     = \chi_{\vq}(L(W))_{\le 2} + e^W e^{\lsp{\sigma}W^1}.
  \end{equation*}
  Therefore $e^W e^{\lsp{\sigma}W^1}$ must be {\it l\/}-dominant, and
  it is the $\chi_{\vq}$ of the corresponding simple module, which
  is $L(C^\bullet(\lsp{\sigma}W^1,W))$.
\begin{NB2}
  (2) We next suppose $W^2 = S_i$ for $i\in I_1$. Then the almost
  splitting sequence is of the form
  \begin{equation*}
    0 \to W^1 = \tau S_i \to 
    W = \bigoplus_{\vout(h)=i} I_{\vin(h)} \to S_i \to 0,
  \end{equation*}
where $I_j$ is the indecomposable injective module corresponding to $j$.
(See \cite[VII.5.3]{ASS}.)
We apply $\sigma$ to get
\begin{equation*}
  0 \to \lsp{\sigma}W = \bigoplus_{\vout(h)=i} \lsp{\sigma}I_{\vin(h)}
  \xrightarrow{\alpha} \lsp{\sigma}W^1 \xrightarrow{\beta} S_i \to 0.
\end{equation*}
Note here that $\lsp{\sigma}S_i = 0$, but returns in the rightmost term.
From this exact sequence, we have an embedding
$\Gr_V(\lsp{\sigma}W)\subset \Gr_V(\lsp{\sigma}W^1)$.

We have
\( 
  \bigoplus_{\vout(h)=i} \lsp{\sigma}I_{\vin(h)}
  = \bigoplus_{\vout(h)=i} S_{\vin(h)}.
\)
From the above description, it is clear that if a submodule
$X\subset\lsp{\sigma}W^1$ satisfies $\beta(X) \neq 0$ (i.e.\ $\beta(X)
= S_i$), we have $X = \lsp{\sigma}W^1$. Therefore the complement
$\Gr_V(\lsp{\sigma}W^1)\setminus \Gr_V(\lsp{\sigma}W)$ is a
single point $\{ \lsp{\sigma}W^1 \}$ if $V = \lsp{\sigma}W^1$
and empty otherwise.
Using $\chi_{\vq}(L(S_i))_{\le 2} = \chi_{\vq}(x_i')_{\le 2} = Y_{i,\vq^3}$, we have
the assertion also in this case.
\end{NB2}%
\end{proof}

\begin{Remarks}\label{rem:almostsplit}
(1)
  It is clear from the proof that it is enough to assume
\(
    0\to \lsp{\sigma} W^2\to \lsp{\sigma}W \to \lsp{\sigma}W^1\to 0
\)
is an almost splitting sequence for a full subquiver $\cQ'\subset\cQ$.

(2)
  It is also possible to prove the same formula even in the case $W^2
  = S_i$, using the explicit construction of the almost splitting
  sequence for $S_i$. (See \cite[VII.5.3]{ASS}.) The argument is
  similar to one for the following proposition.
\end{Remarks}

\begin{Proposition}
  Let $i\in I_1$.
  Let $W^1$ be the representation of the decorated quiver such that
  $\lsp{\sigma}W^1$ is the indecomposable projective module $P_i$
  corresponding to $i$ of a full subquiver $\cQ'$ of the principal
  part $\cQ$ containing $i$.
  Let $W$ be the representation of $\widetilde\cQ$ given by
  \begin{equation*}
    0 \to W^1 \to W \to S_{i} \to 0.
  \end{equation*}
  Then
  \begin{equation*}
    L(W^1)\otimes x'_i = L(W) + 
    f_i^{\otimes (-1+\sum_{j\in I'} a_{ij}^2)}
    \otimes
    \bigotimes_{j\notin I'} x_j^{\otimes a_{ij}}
    \otimes
    \bigotimes_{k\in I_1\setminus\{ i\}}
    f_k^{\otimes \sum_{j\in I'} a_{ij}a_{jk}}
    ,
  \end{equation*}
where  $I'$ is the set of vertexes in the subquiver $\cQ'\subset\cQ$.
\end{Proposition}

Explicitly we have
\begin{equation*}
   \dim W^1_j = 
   \begin{cases}
     \sum_{j\in I'} a_{ij}^2 - 1 & \text{if $j=i$},
\\
    \sum_{j\in I'} a_{ij}a_{jk} & \text{if $k\in I_1\setminus\{i\}$},
\\
     a_{ij} & \text{if $j\in I'$},
\\
     0 & \text{if $j\notin I'$}.
   \end{cases}
\end{equation*}

\begin{proof}
We apply $\sigma$ to the exact sequence to get
\begin{equation*}
  0 \to \lsp{\sigma}W
  \xrightarrow{\alpha} \lsp{\sigma}W^1 = P_i \xrightarrow{\beta} S_i \to 0.
\end{equation*}
Note here that $\lsp{\sigma}S_i = 0$, but returns in the rightmost term.
From this exact sequence, we have an embedding
$\Gr_V(\lsp{\sigma}W)\subset \Gr_V(\lsp{\sigma}W^1)$.

Since $P_i$ is the projective module, if a submodule
$X\subset\lsp{\sigma}W^1$ satisfies $\beta(X) \neq 0$ (i.e.\ $\beta(X)
= S_i$), we have $X = \lsp{\sigma}W^1$. Therefore the complement
$\Gr_V(\lsp{\sigma}W^1)\setminus \Gr_V(\lsp{\sigma}W)$ is a
single point $\{ \lsp{\sigma}W^1 \}$ if $V = \lsp{\sigma}W^1$
and empty otherwise.
Therefore we have
\begin{equation*}
   \chi_{\vq}(L(W^1)\otimes x_i')_{\le 2}  = 
   \chi_{\vq}(L(W))_{\le 2} + e^{W} e^{\lsp{\sigma}W^1},
\end{equation*}
as $\chi_{\vq}(x_i')_{\le 2} = Y_{i,\vq^3}$.

Now we calculate everything explicitly:
\begin{equation*}
  \begin{split}
  & e^W = \prod_{j\in I'} Y_{j,1}^{a_{ij}} 
    \prod_{k\in I_1} Y_{k,\vq^3}^{\sum_{j\in I'} a_{ij}a_{jk}},
\\
  & e^{\lsp{\sigma}W^1} = V_{i,\vq^2} \prod_{j\in I'} V_{j,\vq}^{a_{ij}}
  = 
  Y_{i,\vq}^{-1} Y_{i,\vq^3}^{-1} \prod_{j\in I} Y_{j,\vq^2}^{a_{ij}}
  \times
  \prod_{j\in I'}
  \left(
  Y_{j,1}^{-1} Y_{j,\vq^2}^{-1} \prod_{k\in I} Y_{k,\vq}^{a_{jk}}
  \right)^{a_{ij}}.
  \end{split}
\end{equation*}
Therefore
\begin{equation*}
  e^W e^{\lsp{\sigma}W^1}
  = \left( Y_{i,\vq} Y_{i,\vq^3} \right)^{-1 + \sum_{j\in I'}a_{ij}^2}
   \prod_{j\notin I'} Y_{j,\vq^2}^{a_{ij}}
   \prod_{k\in I_1\setminus\{i\}}
     (Y_{k,\vq}Y_{k,\vq^3})^{\sum_{j\in I'} a_{ij}a_{jk}}.
\end{equation*}
Thus we have the assertion.
\end{proof}

\begin{Remark}\label{rem:derived}
As indicated in the proof of the above proposition, it is more natural
to define $\lsp{\sigma}S_i$ as $S_i[1]$, an object in the derived
category $\mathscr D(\rep\lsp{\sigma}{\!\widetilde\cQ^{\mathrm{op}}})$. When
$W^1$, $W^2$, $W$ are modules in $\rep\cQ$, we say 
\begin{equation*}
  0\to \lsp{\sigma}W^2\to \lsp{\sigma}W \to\lsp{\sigma}W^1\to 0 
\end{equation*}
is an exact sequence, if it is so in the usual sense if $W^2\neq S_i$
and if $0\to \lsp{\sigma}W \to\lsp{\sigma}W^1 \to S_i\to 0$ is so if
$W^2 = S_i$. This is also compatible with the cluster category theory,
as $S_i[1] = I_i[1]$ for $i\in I_1$, where $I_i$ is the indecomposable
injective module corresponding to the vertex $i$.
\end{Remark}

\begin{Proposition}
  Let $i\in I_0$.
  Let $W^1$ be the representation of the decorated quiver such that
  $\lsp{\sigma}W^1$ is the indecomposable injective module $I_i$
  corresponding to $i$ of a full subquiver $\cQ'$ of the principal
  part $\cQ$ containing $i$.
  Let $W$ be the representation of $\widetilde\cQ$ given by
  \begin{equation*}
    0 \to W^1 \to W \to S_{i'} \to 0.
  \end{equation*}
  Then
  \begin{equation*}
    L(W^1)\otimes x_i = L(W) + 
      \bigotimes_{j\notin I'} f_j^{\otimes a_{ij}}
      \otimes \bigotimes_{j\in I'} x_j^{\otimes a_{ij}},
  \end{equation*}
where  $I'$ is the set of vertexes in the subquiver $\cQ'\subset\cQ$.
\end{Proposition}

Explicitly we have
\begin{equation*}
   \dim W^1_j = 
   \begin{cases}
     1 & \text{if $j=i$},
\\
     a_{ij} & \text{if $j\notin I'$},
\\
     0 & \text{if $j\in I'$}.
   \end{cases}
\end{equation*}

\begin{NB2}
The following example is for the decorated quiver $\widetilde\cQ$, not
for $\lsp{\sigma}{\!\widetilde\cQ}$.
  \begin{Example}
    First note
  \begin{equation*}
     \chi_{\vq,t}(x_i)_{\le 2} =
     \begin{cases}
       Y_{i,\vq^2} & \text{if $i\in I_0$},
\\
       Y_{i,\vq}(1+V_{i,\vq^2}) & \text{if $i\in I_1$}.
     \end{cases}
  \end{equation*}

Consider type $A_3$ with $\dim W = (1 \leftarrow 1 \rightarrow 1)$.
($\dim \lsp{\sigma}W = (1 \leftarrow 1 \rightarrow 1)$.)
\begin{equation*}
  \begin{split}
   &\chi_{\vq,t}(L(W))_{\le 2}
   = Y_{1,1}Y_{2,\vq^3} Y_{3,1} \{ ( 1 + V_{1,\vq})(1 + V_{2,\vq^2})(1 + V_{3,\vq})
    - (V_{2,\vq^2} + V_{2,\vq^2} (V_{1,\vq} + V_{3,\vq}) \}
\\
   = \; & Y_{1,1}Y_{2,\vq^3} Y_{3,1} + \cdots + Y_{2,\vq}.
  \end{split}
\end{equation*}
If $i=1$, or $3\in I_0$, we multiply $Y_{i,\vq^2}$. If $i=2$, we
multiply $Y_{2,\vq}(1 + V_{2,\vq^2}) = Y_{2,\vq^3}^{-1}
Y_{1,\vq^2}Y_{3,\vq^2}$. The {\it l\/}-highest weights of the simple
modules different from $L(W\oplus S_{i'})$ are
\begin{equation*}
  i=1,3 : Y_{2,\vq} Y_{i,\vq^2},
\quad
  i=2 : Y_{1,1}Y_{1,\vq^2} Y_{3,1} Y_{3,\vq^2}.
\end{equation*}

Next consider $\dim W = (1 \leftarrow 1 \rightarrow 0)$.
($\dim \lsp{\sigma}W = (1 \leftarrow 0 \rightarrow 0)$.)
 We have
\begin{equation*}
   \chi_{\vq,t}(L(W))_{\le 2}
   = Y_{1,1}Y_{2,\vq^3} ( 1 + V_{1,\vq})
   = Y_{1,1}Y_{2,\vq^3} + Y_{1,\vq^2}^{-1} Y_{2,\vq}Y_{2,\vq^3}.
\end{equation*}
If $i=1$, we multiply $Y_{1,\vq^2}$. If $i=2$, we multiply
$Y_{2,\vq}(1 + V_{2,\vq^2}) = Y_{2,\vq^3}^{-1} Y_{1,\vq^2}Y_{3,\vq^2}$.
We cannot multiply $i=3$, since $W_3 = 0$. In fact, we cannot kill
$Y_{1,\vq^2}^{-1}$.

Finally consider $\dim W = (1 \leftarrow 0 \rightarrow 0)$.
($\dim \lsp{\sigma}W = (1 \leftarrow 1 \rightarrow 0)$.)
We have
\begin{equation*}
   \chi_{\vq,t}(L(W))_{\le 2}
   = Y_{1,1} ( 1 + V_{1,\vq})
   = Y_{1,1} + Y_{1,\vq^2}^{-1} Y_{2,\vq}.
\end{equation*}
If $i=1$, we multiply $Y_{1,\vq^2}$ to get $Y_{2,\vq}$.
\end{Example}
\end{NB2}

\begin{proof}
From the given exact sequence, we have
\begin{equation*}
  0\to S_{i'} \to \lsp{\sigma}W \to \lsp{\sigma}W^1 \to 0.
\end{equation*}
Since we must put $X_{i'} = 0$ for a submodule $X\subset
\lsp{\sigma}W$, we have an embedding
\(
  \Gr_V(\lsp{\sigma}W)\subset \Gr_V(\lsp{\sigma}W^1)
\)
giving by the image under $\lsp{\sigma}W\to \lsp{\sigma}W^1$.
Conversely if $X\in \Gr_V(\lsp{\sigma}W^1)$ is given, we consider
it an $I$-graded subspace of $\lsp{\sigma}W$ by setting $0$ on $X_{i'}$.
Then it is a submodule if and only if $\lsp{\sigma}\bx_i(X_i) = 0$. This
is true if and only if $X_i = 0$, since $\lsp{\sigma}W_i = \C$ and
$\bx_i$ is an isomorphism.
Therefore the complement $\Gr_V(\lsp{\sigma}W^1)\setminus
\Gr_V(\lsp{\sigma}W)$ is empty if $V_i = 0$, and the whole space if
$V_i = \C$. Hence
\begin{equation*}
  \chi_{\vq}(L(W^1)\otimes x_i)_{\le 2}
  = \chi_{\vq}(L(W))_{\le 2} + 
  Y_{i,\vq^2}\sum_{V_i=\C} \Gr_V(\lsp{\sigma}W^1) e^{W^1} e^V,
\end{equation*}
as $\chi_{\vq}(x_i)_{\le 2} = Y_{i,\vq^2}$.
A submodule $X$ of $\lsp{\sigma}W^1$ with $X_i = \C$ is nothing but a
submodule of $\lsp{\sigma}W^1/S_i = \bigoplus_{j\in I'} S_j^{\oplus
  a_{ij}}$.
Since we put the whole space as $X_{j'}$ at the frozen vertex $j'$,
the corresponding Grassmannian is the same as that for
$\bigoplus_{j\in I'} F_j^{\oplus a_{ij}}$, where $F_j$ is the
representation corresponding to $f_j$. Then the latter module is the
image under $\sigma$ of $\bigoplus_{j\in I'} S_{j'}^{\oplus a_{ij}}$.
Therefore the second term is equal to
\begin{equation*}
  \begin{split}
  & Y_{i,\vq^2} \sum_{V_i = \C} \Gr_V(\lsp{\sigma}W^1) e^{W^1} e^V
\\
  =\; & \prod_{j\notin I'} (Y_{j,\vq} Y_{j,\vq^3})^{a_{ij}}
  \sum_{V/S_i} \Gr_{V/S_i}(\lsp{\sigma}{\left(
  \bigoplus_{j\in I'} S_{j'}^{\oplus a_{ij}}\right)})
  e^{\bigoplus_{j\in I'} S_{j'}^{\oplus a_{ij}}}
  e^{V/S_i}
\\
  =\; &
   \prod_{j\notin I'} \chi_{\vq}(f_j)_{\le 2}^{a_{ij}}
   \prod_{j\in I'} \chi_{\vq}(x_j)_{\le 2}^{a_{ij}},
  \end{split}
\end{equation*}
as $\chi_{\vq}(f_j)_{\le 2} = Y_{j,\vq}Y_{j,\vq^3}$.
\begin{NB2}
  \begin{equation*}
      \begin{split}
   Y_{i,\vq^2} e^{W^1} e^V &=
  Y_{i,1} Y_{i,\vq^2} \prod_{j\notin I'} Y_{j,\vq^3}^{a_{ij}}
    V_{i,\vq} e^{V/S_i}
\\
  &= \prod_{j\notin I'} (Y_{j,\vq}Y_{j,\vq^3})^{a_{ij}}
  \prod_{j\in I} Y_{j,\vq}^{a_{ij}}
  \,
  = \prod_{j\notin I'} (Y_{j,\vq}Y_{j,\vq^3})^{a_{ij}}
  \,
  e^{\bigoplus_{j\in I'} S_{j'}^{\oplus a_{ij}}} e^{V/S_i}
  \end{split}
  \end{equation*}
\end{NB2}%
\end{proof}

The following proof contains the mistake. But I keep it for the record.
\begin{proof}
\begin{NB2}
For completeness, we first consider the case $\dim W_i = 0$,
though it is already contained in \propref{prop:fac}, as
$\lsp{\varphi}(W\oplus S_{i'}) = W$.

  Recall that $x_i$ corresponds to the simple module $S_{i'}$ at the
  frozen vertex $i$ of the decorated quiver.
  
  We consider the $(I\sqcup I_\fr)$-graded vector space $W\oplus
  S_{i'}$. Let $W' \defeq \lsp{\varphi}(W\oplus S_{i'})$. Then
  \begin{equation*}
    \dim W'_j =
    \begin{cases}
      \max(\dim W_i - 1,0) & \text{if $j=i$},
      \\
      \dim W_j & \text{if $j\neq i$}.
    \end{cases}
  \end{equation*}
  If $\dim W_i = 0$, we clearly have $\Gr_V(\lsp{\sigma\varphi}W) =
  \Gr_V(\lsp{\sigma}W')$, as $\lsp{\varphi}W = W = W'$. Therefore
\begin{equation*}
  L(W\oplus S_{i'}) \underset{\text{Prop.\ref{prop:fac}}}{\cong}
  L(W')\otimes x_i \cong L(W)\otimes x_i.
\end{equation*}

So suppose $\dim W_i \neq 0$ and $\dim W'_i = \dim W_i - 1$.
\end{NB2}%
Recall that $x_i$ corresponds to the simple module $S_{i'}$ at the
frozen vertex $i$ of the decorated quiver.
If $W = S_i$ (and $L(W) = x'_i$), the assertion is contained in
\eqref{eq:T-system}. Therefore we may assume $W\neq S_i$.

We consider the $(I\sqcup I_\fr)$-graded vector space $W' = W\oplus
S_{i'}$.

First suppose that $i\in I_0$.
\begin{NB2}
  $i$ is sink, $i'$ is source.
\end{NB2}
Taking a general representation for $W$ and add an isomorphism $\C =
W'_{i'} \to W'_{i} = \C$ to get a general representation for $W'$. We
have an exact sequence
\begin{equation*}
   0 \to  W \xrightarrow{} W' \xrightarrow{} S_{i'} \to 0.
\end{equation*}
We have $\C\cong\Hom(W,W) \cong \Hom(W,W') \cong \Hom(W',W')$.
\begin{NB2}
  A $\xi\in \Hom(W,W')$ is a collection $\xi_i\colon W_i\to W'_i$
and $\xi_{i'}\colon W_{i'}\to W'_{i'}$ commuting with $\by_h$ and $\by_h'$.
But as $i'$ is source and $W_{i'} = 0$, it is nothing but $\Hom(W,W)$.

For the second isomorphism, we use
\begin{equation*}
   \C = \Hom(W,W')
   \leftarrow \Hom(W',W')
   \leftarrow \Hom(S_{i'},W') = 0.
\end{equation*}
\end{NB2}%
We apply $\sigma$ to get
\begin{equation*}
   0 \to \lsp{\sigma}S_{i'}=S_{i'} \xrightarrow{\alpha}
   \lsp{\sigma}W'\xrightarrow{\beta} \lsp{\sigma}W\to 0.
\end{equation*}
Since $\End(\lsp{\sigma}W')=\C$, $\beta$ is right minimal.
It is also right almost split as $\lsp{\sigma}W'_{i'} = \C$.
\begin{NB2}
  Take a non-retraction homomorphism $v\colon M\to \lsp{\sigma}W$. If
  $M_{i'}\neq 0$, then $M_i\to M_{i'}$ is nonzero, since $v$ is not a
  retraction. From $v$, we have $M_i\to \lsp{\sigma}W_i =
  \lsp{\sigma}W'_i$. We compose $\lsp{\sigma}W'_i\to
  \lsp{\sigma}W'_{i'}$ and define $M_{i'}\to \lsp{\sigma}W'_{i'}$ so
  that the following diagram is commutative.
  \begin{equation*}
    \begin{CD}
      M_i @>{v_i}>> \lsp{\sigma}W'_i
\\
      @VVV           @VVV
\\
      M_{i'} @>>\exists> \lsp{\sigma}W'_{i'} = \C
    \end{CD}
  \end{equation*}
\end{NB2}%
\begin{NB2}
This is not correct, as the following example for $A_2$ shows:
\begin{equation*}
  0\to
  \begin{pmatrix}
    0 & & 0
\\
    \downarrow && 
\\
   \C &&
  \end{pmatrix}
\to
  \begin{pmatrix}
    \C & \leftarrow & \C
\\
    \downarrow &&
\\
   \C &&
  \end{pmatrix}
\to
  \begin{pmatrix}
    \C & \leftarrow & \C
\\
    \downarrow && 
\\
   0 &   
  \end{pmatrix}
\to 0.
\end{equation*}
Then we cannot lift a homomorphism
\begin{equation*}
  \begin{pmatrix}
    \C & & 0
\\
    \downarrow && 
\\
   0 &   
  \end{pmatrix}
\to  
  \begin{pmatrix}
    \C & \leftarrow & \C
\\
    \downarrow && 
\\
   0 &   
  \end{pmatrix}
\end{equation*}
to the middle term.
\end{NB2}%
Therefore the above is almost split (\cite[IV, Th.~1.13]{ASS}).
Hence we apply \lemref{lem:CC} to get the assertion.
\begin{NB2}
  The only missing piece in $\Gr_V(\lsp{\sigma}W')$ is
  $\Gr_{0}(\lsp{\sigma}S_{i'})\times \Gr_{\lsp{\sigma}W}(\lsp{\sigma}W)$.
\end{NB2}

Next suppose $i\in I_1$.
  \begin{NB2}
    So $i$ is source, $i'$ is sink.
  \end{NB2}%
  Taking a general representation for $W$ and add an isomorphism $\C =
  W'_{i} \to W'_{i'} = \C$ to get a general representation for
  $W'$. We have an exact sequence
\begin{equation*}
   0 \to  S_{i'} \xrightarrow{} W' \xrightarrow{} W \to 0.
\end{equation*}
We have $\C\cong\Hom(W,W) \cong \Hom(W',W) \cong \Hom(W',W')$.
\begin{NB2}
The first isomorphism is obvious as
$\Hom(S_{i'},W) = 0$.
For the second, use $0 = \Hom(W',S_{i'}) \to \Hom(W',W')\to \Hom(W',W)$.
\end{NB2}%
This is almost split as above.

We apply $\sigma$ to get
\begin{equation*}
  0\to \lsp{\sigma}W\to \lsp{\sigma}W' \to \lsp{\sigma}S_{i'}\to 0,
\end{equation*}
which is still an almost split sequence by \cite[VI.5.3]{ASS}. Then we
apply \lemref{lem:CC} to get the assertion.
\begin{NB2}
  The only missing piece in $\Gr_V(\lsp{\sigma}(W'))$
is $\Gr_{0}(\lsp{\sigma}W)\times \Gr_{\lsp{\sigma}S_i}(\lsp{\sigma}S_{i'})$.
\end{NB2}%
\end{proof}

\subsection{Computation}\label{subsec:compute}

We compute $L(C^\bullet(\lsp{\sigma}W^1,W))$ in \propref{prop:ex1}
in this subsection.

Recall that the Euler characteristic is defined by
\(
   \chi(M,N) = \dim \Hom(M,N) - \dim \Ext^1(M,N).
\)
We consider the principal part $\cQ$ of the decorated quiver. We have
\begin{equation*}
  \chi(S_i, S_j) = 
  \begin{cases}
    - a_{ij} & \text{if $i\in I_1$, $j\in I_0$},
\\
    0 & \text{if $i\in I_0$, $j\in I_1$},
\\
    \delta_{ij} & \text{if $i,j\in I_0$ or $i,j\in I_1$}.
  \end{cases}
\end{equation*}

Since we are assuming $W_{i'} = 0$, the nontrivial terms in
$C^\bullet_{i,\vq^n}$ are as follows:
\begin{equation*}
  \begin{CD}
   C_{i,\vq^3}^\bullet(\lsp{\sigma}W^1,W)\  (i\in I_1):\qquad
   @. @. W_i@>>> \lsp{\sigma}W_i^1
\\
   C_{i,\vq^2}^\bullet(\lsp{\sigma}W^1,W)\  (i\in I_0):\qquad
   @. @.
   \displaystyle{\bigoplus_{h:\vin(h)=i}}
     \lsp{\sigma}W^1_{\vout(h)}
     @>>> \lsp{\sigma}W^1_i
\\
   C_{i,\vq}^\bullet(\lsp{\sigma}W^1,W)\  (i\in I_1) :\qquad
   @. \lsp{\sigma}W_i^1 @>>> 
  \displaystyle{\bigoplus_{h:\vin(h)=i}}
     \lsp{\sigma}W_{\vout(h)}^1
   @.
\\
   C_{i,1}^\bullet(\lsp{\sigma}W^1,W)\ (i\in I_0): \qquad @.
   \lsp{\sigma}W^1_i @>>>  W_i @.
  \end{CD}
\end{equation*}
Since we are in the situation of \propref{prop:ex1}, we may further
assume that $W$, $W^1$, $W^2$ do not contain $S_i$ ($i\in I_1$) as
direct summands, $W^1$, $W^2$ are indecomposable, and $W^2$ is
nonprojective. Therefore
\begin{NB2}
\begin{gather*}
  \rank C^\bullet_{i,\vq^3}(\lsp{\sigma}W^1,W)
  = \dim W_i - \dim \lsp{\sigma}W^1_i,
\\
  \rank C^\bullet_{i,\vq^2}(\lsp{\sigma}W^1,W)
  = \sum_{j} a_{ij} \dim \lsp{\sigma}W^1_j - \dim \lsp{\sigma}W^1_i,
\\
  \rank C^\bullet_{i,\vq}(\lsp{\sigma}W^1,W)
  = \sum_{j} a_{ij} \dim \lsp{\sigma}W^1_j - \dim \lsp{\sigma}W^1_i,
\\
  \rank C^\bullet_{i,1}(\lsp{\sigma}W^1,W)
  = \dim W_i - \dim \lsp{\sigma}W^1_i.
\end{gather*}
and hence
\end{NB2}%
\begin{NB2}
$i\in I_1$ : source  
\end{NB2}
\begin{equation*}
  \begin{split}
  & \rank C^\bullet_{i,\vq^3}(\lsp{\sigma}W^1,W)
  -   \rank C^\bullet_{i,\vq}(\lsp{\sigma}W^1,W)
  = \dim W_i - \sum_{j} a_{ij} \dim W^1_j
\\
  =\; & \dim W^2_i + \dim W^1_i - \sum_{j} a_{ij} \dim W^1_j
  = \chi(W^2,S_i) +  \chi(S_i,W^1)
\\
  =\; & \chi(W^2,S_i) + \chi(S_i,\tau W^2) = 0
  \end{split}
\end{equation*}
and \begin{NB2}
$i\in I_0$ : sink
\end{NB2}
\begin{equation*}
  \begin{split}
    & \rank C^\bullet_{i,1}(\lsp{\sigma}W^1,W) - \rank
    C^\bullet_{i,\vq^2}(\lsp{\sigma}W^1,W) = \dim W_i - \sum_{j}
    a_{ij} \dim \lsp{\sigma}W^1_j
    \\
    =\; & \dim \lsp{\sigma}W^2_i + \dim \lsp{\sigma}W^1_i - \sum_{j}
    a_{ij} \dim \lsp{\sigma}W^1_j = \chi(\lsp{\sigma}W^2,S_i) +
    \chi(S_i,\lsp{\sigma}W^1)
    \\
    =\; & \chi(\lsp{\sigma}W^2,S_i) + \chi(S_i, \tau \lsp{\sigma}W^2)
    = 0,
  \end{split}
\end{equation*}
where we have used \eqref{eq:Serre_dual}. Therefore
\begin{equation*}
  \begin{split}
   L(C^\bullet(\lsp{\sigma}W^1,W))
   & \cong \bigotimes_{i\in I_1}n
   f_i^{\otimes \rank C^\bullet_{i,\vq^3}(\lsp{\sigma}W^1,W)}
   \otimes
   \bigotimes_{i\in I_0}
   f_i^{\otimes \rank C^\bullet_{i,1}(\lsp{\sigma}W^1,W)}
\\
   & \cong \bigotimes_{i\in I_1}
   f_i^{\otimes \dim W^1_i}
   \otimes
   \bigotimes_{i\in I_0}
   f_i^{\otimes \dim W^2_i}
  \end{split}
\end{equation*}
by \propref{prop:fac}.
\begin{NB2}
  \begin{equation*}
    \rank C^\bullet_{i,\vq^3}(\lsp{\sigma}W^1,W)
    = \dim W_i - \dim \lsp{\sigma}W^1_i 
    = \dim W_i + \dim W^1_i 
    - \sum_j a_{ij} \dim W^1_j
    = \dim W^1_i
  \end{equation*}
from the computation above.
\end{NB2}

When we only assume $0\to W^1\to W\to W^2\to 0$ is an almost splitting
sequence for a full subquiver $\cQ'\subset \cQ$, the above computation
is true only if $i\in I'$, a vertex in the subquiver.
For a vertex $i\notin I'$, we have
\begin{equation*}
  \rank C^\bullet_{i,\vq} \text{ or }
  \rank C^\bullet_{i,\vq^2}
  = \sum_j a_{ij} \dim \lsp{\sigma}W^1_j,
\qquad
  \rank C^\bullet_{i,\vq^3} \text{ or }
  \rank C^\bullet_{i,1}
  = 0
\end{equation*}
from the form of the complexes above. In summary we have

\begin{Proposition}
  \begin{equation*}
       L(C^\bullet(\lsp{\sigma}W^1,W))
       \cong
       \bigotimes_{i\in I_1\cap I'}
   f_i^{\otimes \dim W^1_i}
   \otimes
   \bigotimes_{i\in I_0\cap I'}
   f_i^{\otimes \dim W^2_i}
   \otimes
   \bigotimes_{i\notin I'}
   (x_i')^{\otimes  \sum_j a_{ij} \dim \lsp{\sigma}W^1_j}.
  \end{equation*}
\end{Proposition}

\subsection{Cluster-tilting sets}

\begin{NB2}
Earlier version:  

Let $n = \# I$. A collection of ($2n$) simple modules 
\[
  \bL = \{ L(W^1),\dots, L(W^n)\} \sqcup \{ f_i\}_{i\in I}
\]
is said to be a {\it cluster-tilting set\/} if the following
conditions are satisfied:
\begin{enumerate}
\setcounter{enumi}{-1}
\item Each $W^k$ satisfies $(\ast_{\ell=1})$.
\item $L(W^k)$'s are distinct simple modules different from any $f_i$.
\item Each $W^k$ is an {\it indecomposable\/} module of the
  decorated quiver.
 
\item Remove $L(W^k)$ from $\bL$ if $L(W^k) = x_i$, i.e.\ $W^k =
  S_{i'}$ the simple module for a new vertex $i'\in I_\fr$. Delete the
  corresponding vertex $i$ from the principal part of the decorated
  quiver $(I,\Omega)$. (It is an {\it old\/} vertex.)

\item Remove also $f_i$ from $\bL$. Let $\lsp{\varphi}{\bL}$ be the
  resulting collection. Let $\lsp{\varphi}\cQ =
  (\lsp{\varphi}{I},\lsp{\varphi}{\Omega})$ be the resulting quiver:
  \begin{equation*}
    \begin{split}
    & \lsp{\varphi}{\bL} = \bL\setminus \bigcup \{ f_i\}_{i\in I}
    \setminus
    \{ L(W^k) \mid \text{$L(W^k) = x_i$ for some $i$}\},
\\
   & \lsp{\varphi}{I} = I \setminus \{ i \mid 
   \text{$x_i = L(W^k)$ for some $k$}\},
\quad
   \lsp{\varphi}\Omega = \Omega\setminus
   \{ h \mid \text{$\vout(h)$ or $\vin(h)\in I\setminus\lsp{\varphi}{I}$}\}
    \end{split}
  \end{equation*}

\item If $L(W^k)\in \lsp{\varphi}{\bL}$, then $W^k_{i'} = 0$
  for all new vertexes $i'$ and $W^k_i = 0$ if $i\in I\setminus
  \lsp{\varphi}{I}$. Therefore $W^k$ is a module of the quiver
  $\lsp{\varphi}{(I,\Omega)}$.

\item $\bigoplus_{L(W^k)\in\lsp{\varphi}{\bL}} W^k$ is a tilting module as
  a representation of $\lsp{\varphi}(I,\Omega)$.
\end{enumerate}
\end{NB2}

Let $n = \# I$. A collection of $n$ indecomposable representations
of the modified quiver $\lsp{\sigma}{\!\widetilde\cQ}$ plus
$n$ KR modules
\[
  \bL = \{ \lsp{\bar{\sigma}}W^1,\dots, \lsp{\bar{\sigma}}W^n\}
  \sqcup \{ f_i\}_{i\in I}
\]
is said to be a {\it cluster-tilting set\/} if the following
conditions are satisfied:
\begin{enumerate}
\setcounter{enumi}{-1}
\item Each $\lsp{\bar{\sigma}}W^k$ satisfies $(\ast_{\ell=1})$.
\item $\lsp{\bar{\sigma}}W^k$'s are pairwise nonisomorphic.
 
\item Remove $\lsp{\bar{\sigma}}W^k$ from $\bL$ if $\lsp{\bar{\sigma}}W^k =
  S_{i'}$. Delete the corresponding vertex $i$ from the principal part
  of the decorated quiver $(I,\Omega)$. (It is an {\it old\/} vertex.)

\item Let $\lsp{\psi}{\bL}$ be the resulting collection. Let
  $\lsp{\psi}\cQ = (\lsp{\psi}{I},\lsp{\psi}{\Omega})$ be the resulting
  principal part of the quiver:
  \begin{equation*}
    \begin{split}
    & \lsp{\psi}{\bL} = \bL\setminus
    \{ \lsp{\bar{\sigma}}W^k \mid \text{$\lsp{\bar{\sigma}}W^k = S_{i'}$ for some
      $i$}\},
    \\
    & \lsp{\psi}{I} = I \setminus \{ i \mid \text{$S_{i'} =
      \lsp{\bar{\sigma}}W^k$ for some $k$}\}, \quad \lsp{\psi}\Omega =
    \Omega\setminus \{ h \mid \text{$\vout(h)$ or $\vin(h)\in
      I\setminus\lsp{\psi}{I}$}\}.
    \end{split}
  \end{equation*}

\item If $\lsp{\bar{\sigma}}W^k\in \lsp{\psi}{\bL}$, then
  $\lsp{\bar{\sigma}}W^k_{i'} = 0$ for all new vertexes $i'$ and
  $\lsp{\bar{\sigma}}W^k_i = 0$ if $i\in I\setminus
  \lsp{\psi}{I}$. Therefore $\lsp{\bar{\sigma}}W^k$ is a module of the
  quiver $\lsp{\psi}{(I,\Omega)}$.

\item $\bigoplus_{\lsp{\bar{\sigma}}W^k\in\lsp{\psi}{\bL}} \lsp{\bar{\sigma}}W^k$
  is a tilting module as a representation of $\lsp{\psi}(I,\Omega)$.
  
\end{enumerate}
Note that $\# (\lsp{\psi}{\bL}) = \# (\lsp{\psi}{I})$. Therefore
$\bigoplus_{\lsp{\bar{\sigma}}W^k\in\lsp{\psi}{\bL}}
\lsp{\bar{\sigma}}W^k$ is tilting if and only if
$\ext^1(\lsp{\bar{\sigma}}W^k,\lsp{\bar{\sigma}}W^l) = 0$ for any
$\lsp{\bar{\sigma}}W^k,\lsp{\bar{\sigma}}W^l\in \bL$ (including the
case $k=l$).

If we identify $\lsp{\bar{\sigma}}W^k = S_{i'}\notin \lsp{\psi}\bL$
with $I_i[1]$ the shift of the indecomposable injective module
associated corresponding to the vertex $i$, the above definition is
nothing but the definition of a cluster-tilting set for the cluster
category \cite{BMRRT}.

Here $\bar\sigma$ is not a functor, and $\lsp{\bar{\sigma}}W^k$ is
just a notation for an indecomposable representation of the modified
quiver $\lsp{\sigma}{\!\widetilde\cQ}$. We choose a representation
$W^k$ of the decorated quiver $\widetilde\cQ$ by the following rule:
\begin{aenume}
\item if $\lsp{\bar{\sigma}}W^k = S_{i'}$ with $i\in I_1$, then we put
  $W^k = S_i$;
\item if $\lsp{\bar{\sigma}}W^k = S_{i}$ with $i\in I_1$, then we put
  $W^k = S_{i'}$;
\item otherwise we take $W^k$ so that $\lsp{\sigma}W^k =
  \lsp{\bar{\sigma}}W^k$.
\end{aenume}
We have $\lsp{\sigma}W^k = 0$ in case (a), and $\lsp{\sigma}W^k = F^i$
(representation corresponding to $f_i$) in case (b). They are
different from $\lsp{\bar\sigma}W^k$. But the quiver Grassmannian
varieties $\Gr_V(\lsp{\sigma}W^k)$, $\Gr_V(\lsp{\bar\sigma}W^k)$ are
isomorphic (they are either points or empty sets), so we have no
problem.
Note that $W^k$ are pairwise nonisomorphic also.

Here is the table of the exceptional rules including the corresponding
cluster variables.

\begin{center}
\begin{tabular}{c|c|c|c}
cluster variables & $W$ & $\lsp{\bar{\sigma}}W$ & $\lsp{\sigma}W$
\\
\hline
$z_i = x_i'$ ($i\in I_1$) & $S_{i}$ & $S_{i'}$ & 0
\\
$x_i$ ($i\in I_1$) & $S_{i'}$ & $S_i$ & $F^i$
\\
$f_i$ ($i\in I_1$) & $F_{i}$ & $\times$ & $S_{i'}$
\end{tabular}
\end{center}
The bottom line is just a remark. We do not assign a representation of
$\lsp{\sigma}{\!\widetilde\cQ}$ to $f_i$.

We identify $\bL$ with 
$\{ L(W^1),\dots, L(W^n)\} \sqcup \{ f_i\}_{i\in I}$ a collection
of $(2n)$ simple modules.

\begin{Proposition}
  If $\bL = \{ L(W^1),\dots, L(W^n)\} \sqcup \{ f_i\}_{i\in I}$ is a
  cluster-tilting set, the tensor product
\begin{equation*}
   L(W^1)^{\otimes m_1}\otimes\cdots\otimes
   L(W^n)^{\otimes m_n} \otimes \bigotimes_{i\in I} f_i^{\otimes n_i}
\end{equation*}
is simple for any $m_k, n_i\in \Z_{\ge 0}$.
\end{Proposition}

\begin{proof}
  Let $F^i$ be the $(I\sqcup I_\fr)$-graded vector space corresponding
  to the Kirillov-Reshetikhin module $f_i$. We consider
  \begin{equation*}
    W \defeq (W^1)^{\oplus m_1}\oplus\cdots\oplus (W^n)^{\oplus m_n}
    \oplus \bigoplus_{i\in I} (F^i)^{\oplus n_i}
  \end{equation*}
  and the corresponding simple module $L(W)$. Let us apply
  \propref{prop:fac} to $L(W)$. 
  We fix a vertex $i\in I$ and want to show $\min(\dim W_i, \dim
  W_{i'}) = n_i$. We separate the situation into two cases:

  Case 1. $\bL$ contains $S_{i'}$, say $\lsp{\bar{\sigma}}W^1$.

  All other $\lsp{\bar{\sigma}}W^k$ ($k\in \{2,\dots,n\}$) do not have
  entries at $i$ and $i'$ from the condition~(4). Thus $\dim
  \lsp{\bar{\sigma}}W_i = n_i$, $\dim \lsp{\bar{\sigma}}W_{i'} = n_i + m_1$.
  If $i\in I_0$, we have $\dim W_i = n_i$, $\dim W_{i'} = n_i + m_1$.
  Therefore $\min(\dim W_i, \dim W_{i'}) = n_i$.
  If $i\in I_1$, we have $\dim W_i \ge n_i + m_1$, $\dim W_{i'} = n_i$.
  We also have $\min(\dim W_i, \dim W_{i'}) = n_i$.

  Case 2. $\bL$ does not contain $z_i$

  We have $\dim W_{i'} = n_i$ and $\dim
  W_i \ge n_i$. Therefore we have $\min(\dim W_i, \dim W_{i'}) = n_i$.

Thus \propref{prop:fac} implies
  \begin{equation*}
    L(W)\cong L(\lsp{\varphi}{W})\otimes
    \bigotimes_{i\in I} f_i^{\otimes n_i}
    ,
\qquad
   \lsp{\varphi}{W} = \bigoplus_{k=1}^n (W^k)^{\oplus m_k}.
  \end{equation*}

  In case 1 with $i\in I_0$, we have $\dim \lsp{\varphi}W_i = 0$,
  $\dim \lsp{\varphi}W_{i'} = m_1$. Then we have the factor
  $S_{i'}^{\oplus m_1}$ in the canonical decomposition of
  $\lsp{\varphi}W$. Therefore $x_i^{\otimes m_1} = L(W^1)^{\otimes
    m_1}$ factors out from $L(\lsp{\varphi}{W})$ by
  \propref{prop:canfac}.

  In case 1 with $i\in I_1$, we have $\Ext^1_{\cQ}(S_i, W^k) = 0$
  ($k\neq 1$) as $\lsp{\sigma}W^k_i = \lsp{\bar\sigma}W^k_i = 0$ by
  the condition~(5).
  \begin{NB2}
    A nonsplit exact sequence
\(
    0 \to W^k \to W \to S_i \to 0
\)
induces
\(
    0  \to \lsp{\sigma} W\to \lsp{\sigma}W^k \to S_i \to 0.
\)
But the third arrow must be zero as $\lsp{\sigma}W^k_i = 0$.
  \end{NB2}%
  We also have $\Ext^1_{\cQ}(W^k,S_i) = 0$ since $i$ is
  sink. Therefore $W = S_i^{\oplus m_1}\oplus \bigoplus_{k\ge 2}
  W_k^{\oplus m_k} = W_1^{\oplus m_1}\oplus \bigoplus_{k\ge 2}
  W_k^{\oplus m_k}$ is the canonical decomposition. Therefore
  $z_i^{\otimes m_1} = x_{i'}^{\otimes m_1} = L(W_1)^{\otimes m_1}$
  factors out from $L(\lsp{\varphi}W)$ also in this case.

  The remaining part also factors as desired, since
  $\bigoplus_{L(W^k)\in\lsp{\psi}{\bL}} \lsp{\bar\sigma}W^k$ is a
  tilting module.
  More precisely we need to have $\bigoplus_{L(W^k)\in\lsp{\psi}{\bL}}
  \lsp{\sigma}W^k$ to be a tilting module, but $\lsp{\sigma}W^k$ and
  $\lsp{\bar\sigma}W^k$ possibly differ only by $S_i$ and $F_i$, as we
  already factored $z_i = x_i'$. (See the table.) Then there is no
  difference for the condition.
\end{proof}

The initial cluster-tilting set $\bL$ is the collection $\bL = \{
S_{i'}, f_i\}_{i\in I}$ or $\{ z_i, f_i\}_{i\in I}$. In this case
$\lsp{\psi}{I} = \emptyset$ and all conditions are satisfied.

For $k\in \{1,\dots, n\}$ we define the {\it mutation\/} $\mu_k(\bL)$
of $\bL$ in direction $k$ as follows:
\begin{enumerate}
\item Suppose $\lsp{\bar\sigma}W^k = S_{i'}$ for $i\in I$. We return
  back the vertex $i$ and all arrows incident to $i$ to the quiver
  $\lsp{\psi}{(I,\Omega)}$. Let $\lsp{+\varphi}(I,\Omega)$ be the
  resulting quiver. Since $\bigoplus_{W^k\in\lsp{\psi}{\bL}}
  \lsp{\bar\sigma}W^k$ is an almost tilting non-sincere module as a
  representation of $\lsp{+\varphi}(I,\Omega)$, we can add the unique
  indecomposable $\lsp{\bar\sigma*}W^k$ to
  $\bigoplus_{W^k\in\lsp{\psi}{\bL}} \lsp{\bar\sigma}W^k$ to get a
  tilting module.

\item Next suppose $\lsp{\bar\sigma}W^k$ is not $S_{i'}$ for any $i\in
  I$. We consider an almost tilting module
  $\bigoplus_{W^l\in\lsp{\psi}{\bL}, l\neq k} \lsp{\bar\sigma}W^l$.
  \begin{aenume}
  \item If it is sincere, there is another indecomposable
  module $\lsp{\bar\sigma*}W^k\neq \lsp{\bar\sigma}W^k$ such that 
  $\lsp{\bar\sigma*}W^k\oplus \bigoplus_{W^l\in\lsp{\psi}{\bL},
    l\neq k} \lsp{\bar\sigma}W^l$ is a tilting module.

\item If it is not sincere, there exists the unique simple module
  $S_i$, not appearing in the composition factors of
  $\bigoplus_{W^l\in\lsp{\psi}{\bL}, l\neq k} \lsp{\bar\sigma}W^l$. Then we set
  $\lsp{\bar\sigma*}W^k = S_{i'}$.
  \end{aenume}
\end{enumerate}
Let
  \begin{equation*}
     \mu_k(\bL) \defeq \bL \cup \{\lsp{\bar\sigma*}W^k\}
     \setminus \{ \lsp{\bar\sigma}W^k \}.
  \end{equation*}

  In all cases $\mu_k(\bL)$ is again a cluster-tilting set. The
  resulting quiver $\lsp{\psi}{(I,\Omega)}$ for $\mu_k(\bL)$ is
  $\lsp{+\varphi}(I,\Omega)$ in case (1), $\lsp{\psi}{(I,\Omega)}$
  in case (2a), and the quiver obtained from
  $\lsp{\psi}{(I,\Omega)}$ by removing $i$ and all incident arrows
  in case (2b).

We can iterate this procedure and obtain new clusters starting from
the initial cluster $\bL = \{x_i,f_i\}_{i\in I}$.

****** formula for $\widetilde\bB$ *********

Suppose that we have arrows $i\to k\to l$.
Since there is no $2$-cycles, there is no arrow $l\to k$.
We consider 
\begin{equation*}
   0\to \slW\to B_l \to \lW\to 0,
\qquad
   0\to \skW\to B_k \to \kW\to 0
\end{equation*}
We suppose that $\kW$ appears in $B_l$ with multiplicity $m > 0$ and
write $B_l = D_l\oplus {\kW}^{\oplus m}$ so that $\kW$ is not a direct
summand of $D_l$.

Starting from the second row and the third column, we have the
commutative diagram exact in columns and rows:
\begin{equation*}
  \begin{CD}
   @. 0 @. 0 @. @.
\\
   @.    @AAA    @AAA @. @.
\\
   0 @>>> \slW @>>> D_l \oplus {\kW}^{\oplus m} @>>> \lW @>>> 0
\\
   @.    @AAA      @AAA @| @. 
\\
   0 @>>> X @>>> D_l \oplus B_k^{\oplus m} @>>> \lW @>>> 0
\\
   @.    @AAA    @AAA @. @.
\\
   @.    {\skW}^{\oplus m} @= {\skW}^{\oplus m} @. @.
\\
   @.    @AAA    @AAA @. @.
\\
   @. 0 @. 0 @. @.
\end{CD}
\end{equation*}

Let $T\defeq \bigoplus_{j} \lsp{\bar\sigma}W^j$ and $T' \defeq
\bigoplus_{j\neq k,l} \lsp{\bar\sigma}W^j\oplus
\lsp{\bar\sigma*}W^k$. We claim $D_l \oplus B_k^{\oplus m}$ is in
$\add T'$. From the definition $D_l$, $B_k\in \add T$. From the
definition of $D_l$, it does not contain either $\kW$ nor $\lW$ as a
direct summand. Moreover $B_k$ does not contain $\lW$ since there is
no arrow $l\to k$. It does not contain $\kW$ by definition.
\begin{NB2}
  We remove $\kW$ from $T$. Then $D_l \oplus B_k^{\oplus m}\in \add
  (T\ominus \kW)$. Hence $D_l \oplus B_k^{\oplus m}\in \add
  ((T\ominus\kW)\oplus \skW)$. Next we remove $\lW$ from 
  $((T\ominus\kW)\oplus \skW)$.
\end{NB2}

We next claim $\Ext^1(X,T') = 0 = \Ext^1(T',X)$.
\end{NB}

\appendix
\section{Odd cohomology vanishing of quiver Grassmannians}
\label{sec:app}

In this appendix, we generalize our proof of the odd cohomology group
vanishing of the quiver Grassmannian of submodules of a rigid module
of a bipartite quiver (\propref{lem:real}(2)) to an acyclic one. Thus
we recover the main result of Caldero-Reineke \cite{caldero-reineke}.
It implies the positivity conjecture for an acyclic cluster algebra
(see \propref{prop:HL}) for the special case of an {\it initial\/}
seed.

After an earlier version of this article was posted on the arXiv, Qin
proved the quantum version of the cluster character formula for an
acyclic cluster algebra \cite{FQ}. As an application, he observed the
odd cohomology group vanishing of quiver Grassmannians. Our proof is
different from his.

Let us first fix the notation.
Let $\cQ = (I,\Omega)$ be a quiver and $W$ be an $I$-graded vector
space. We define
\begin{equation*}
  \bE_W = \bigoplus_{h\in\Omega}\Hom(W_{\vout(h)},W_{\vin(h)}).
\end{equation*}
Its dual space is
\begin{equation*}
  \bE_W^* = \bigoplus_{h\in\Omega}\Hom(W_{\vin(h)},W_{\vout(h)})
  = \bigoplus_{\overline{h}\in\overline{\Omega}}\Hom(W_{\vout(\overline{h})},
  W_{\vin(\overline{h})}).
\end{equation*}
Those are acted by $G_W = \prod_i \GL(W_i)$.

Let $\nu\in\Z_{\ge 0}^I$. Let $\mathcal F(\nu,W)$ be the product of
Grassmanian varieties $\Gr(\nu_i,W_i)$ parametrizing collections
of vector subspaces $X_i\subset W_i$ such that $\dim X_i = \nu_i$.
Let $\Tilde{\mathcal F}(\nu,W)$ be the variety of all pairs
$(\bigoplus \by_h,X)$ where $\bigoplus\by_h \in \bE_W$ and
$X\in\mathcal F(\nu,W)$ such that
\begin{equation*}
   \by_h(X_{\vout(h)}) = 0, \quad
   \by_h(W_{\vout(h)}) \subset X_{\vin(h)}
\end{equation*}
for all $h\in\Omega$. This is a vector bundle over $\mathcal
F(\nu,W)$. Let $\pi\colon \Tilde{\mathcal F}(\nu,W)\to
\bE_W$ be the projection.

Note that $\Tilde{\mathcal F}(\nu,W)$ is a subbundle of the trivial
bundle $\mathcal F(\nu,W)\times \bE_W$.
Let $\Tilde{\mathcal F}(\nu,W)^\perp$ be its annihilator in the dual
trivial bundle $\mathcal F(\nu,W)\times \bE_W^*$ and let
$\pi^\perp\colon\Tilde{\mathcal F}(\nu,W)^\perp\to\bE_W^*$ be the
projection.
More concretely, $\Tilde{\mathcal F}(\nu,W)^\perp$ is the variety of
all pairs $(\bigoplus \by_{\overline{h}}^*,X)$ where $\bigoplus
\by_{\overline{h}}^*\in\bE_W^*$ and $X\in\mathcal F(\nu,W)$ such that
\begin{equation*}
  \by_{\overline{h}}^*(X_{\vout(\overline{h})})\subset X_{\vin(\overline{h})}
\end{equation*}
for all $\overline{h}\in\overline\Omega$.
Therefore $(\pi^\perp)^{-1}(\bigoplus\by_{\overline{h}}^*)$ is the
quiver Grassmannian associated with the quiver representation
$\bigoplus\by_{\overline{h}}^*$.

\begin{Theorem}
  Assume that $\bE_W^*$ contains an open $G_W$-orbit and let
  $\bigoplus\by^*_{\overline{h}}\in\bE_W^*$ be a point in the
  orbit. Then the quiver Grassmannian
  $(\pi^\perp)^{-1}(\bigoplus\by_{\overline{h}}^*)$ has no odd
  cohomology.
\end{Theorem}

\begin{proof}
  Consider the fiber $\pi^{-1}(\bigoplus\by_h)$ of
  $\pi\colon\Tilde{\mathcal F}(\nu,W)\to \bE_W$.
  From the definition of $\Tilde{\mathcal F}(\nu,W)$, it is equal to
  $X\in\mathcal F(\nu,W)$ such that
  \begin{equation*}
    \sum_{h:\vin(h)= i} \Ima \by_{h} \subset X_i
    \subset \bigcap_{h:\vout(h)=i} \Ker\by_{h}.
  \end{equation*}
  Thus it is isomorphic to the product of the usual Grassmannian
  manifolds of subspaces of $\bigcap_{h:\vout(h)=i}
  \Ker\by_{h}/\sum_{h:\vin(h)= i} \Ima \by_{h}$ of dimension $\nu_i -
  \dim\sum_{h:\vin(h)= i} \Ima \by_{h}$. Thus
  $\pi^{-1}(\bigoplus\by_h)$ has no odd homology.

  In the main body of the article, the central fiber was denoted by
  $\NLa(V,W)$, and its odd cohomology vanishing was mentioned in
  \remref{rem:negative}. The remaining part of the proof is the same
  as in \propref{lem:real}(2). Let us sketch it for the sake of the
  reader.

  We consider the pushforward $\pi^\perp_!(1_{\Tilde{\mathcal
      F}(\nu,W)^\perp}[ \dim\Tilde{\mathcal F}(\nu,W)^\perp])$. By the
  decomposition theorem, it is isomorphic to a finite direct sum
  \begin{equation*}
  \bigoplus_{P,d} L_{P,d}\otimes P[d]
  \end{equation*}
  of various simple perverse sheaves $P$ and $d\in\Z$. Here $L_{P,d}$
  is a finite dimensional vector space. Since $\pi^\perp$ is
  $G_W$-equivariant, all $P$'s appearing above are
  equivariant. Therefore under our assumption, only the constant sheaf
  $1_{\bE_W^*}[\dim \bE_W^*]$ is supported on the whole $\bE_W^*$ and
  all other perverse sheaves $P$ have smaller supports. Taking a fiber
  at $\bigoplus y_{\overline{h}}^*$, we find that the cohomology group
  $H_k((\pi^\perp)^{-1}(\bigoplus\by_{\overline{h}}^*))$ is $L_{P,d}$
  with $P = 1_{\bE_W^*}[\dim \bE_W^*]$ and $d = k + \dim \bE_W^* -
  \dim \Tilde{\mathcal F}^\perp(\nu,W)$.

  We apply the Fourier-Sato-Deligne functor $\Psi$ for perverse
  sheaves on $\bE_W$. As in the proof of \thmref{thm:main}, we have
  \begin{equation*}
    \pi_!(1_{\Tilde{\mathcal F}(\nu,W)}[ \dim\Tilde{\mathcal F}(\nu,W)])
    = \bigoplus_{P,d} L_{P,d}\otimes \Psi(P)[d].
  \end{equation*}
  We also have that $\Psi(1_{\bE_W^*}[\dim \bE_W^*])$ is the
  sky-scraper sheaf $1_{\{0\}}$ at the origin of $\bE_W$. The fiber of
  $\pi_!(1_{\Tilde{\mathcal F}(\nu,W)}[ \dim\Tilde{\mathcal
    F}(\nu,W)])$ at $0\in \bE_W$ gives the homology group of
  $\pi^{-1}(0)$ which vanishes in odd degree as we have already
  observed.
  Therefore $L_{P,d}$ for $P = 1_{\bE_W^*}[\dim \bE_W^*]$ vanishes if
  $d+\dim\Tilde{\mathcal F}(\nu,W)$ is odd. Thus we have the assertion.
\end{proof}

The odd homology vanishing of the fiber over $\bigoplus \by_h = 0$ is
enough for the above argument. We give an analysis of arbitrary cases
to show that any fiber is isomorphic to the fiber of $0$ for a
different choice of $\nu$, $W$. In the main body of this article, this
is a consequence of \thmref{thm:slice}.
  
We also remark that we do not assume that the quiver contains no
oriented cycles in the above proof. However it is implicitly assumed
since we only consider the case when $\bE_W^*$ contains an open
orbit. Thus our result applies only to acyclic cluster algebras.

\bibliographystyle{myamsplain}
\bibliography{nakajima,mybib,cluster}

\def\cprime{$'$} \def\cprime{$'$}
\providecommand{\bysame}{\leavevmode\hbox to3em{\hrulefill}\thinspace}
\providecommand{\MR}{\relax\ifhmode\unskip\space\fi MR }
\providecommand{\MRhref}[2]{%
  \href{http://www.ams.org/mathscinet-getitem?mr=#1}{#2}
}
\providecommand{\href}[2]{#2}
\begin{thebibliography}{10}

\bibitem{ASS}
I.~Assem, D.~Simson, and A.~Skowro{\'n}ski, \emph{Elements of the
  representation theory of associative algebras. {V}ol. 1}, London Mathematical
  Society Student Texts, vol.~65, Cambridge University Press, Cambridge, 2006,
  Techniques of representation theory. \MR{MR2197389 (2006j:16020)}

\bibitem{BBD}
A.~A. Be{\u\i}linson, J.~Bernstein, and P.~Deligne, \emph{Faisceaux pervers},
  Analysis and topology on singular spaces, {I} ({L}uminy, 1981), Ast\'erisque,
  vol. 100, Soc. Math. France, Paris, 1982, pp.~5--171. \MR{MR751966
  (86g:32015)}

\bibitem{FomZel3}
A.~Berenstein, S.~Fomin, and A.~Zelevinsky, \emph{Cluster algebras. {III}.
  {U}pper bounds and double {B}ruhat cells}, Duke Math. J. \textbf{126} (2005),
  no.~1, 1--52. \MR{MR2110627 (2005i:16065)}

\bibitem{BerZel}
A.~Berenstein and A.~Zelevinsky, \emph{Quantum cluster algebras}, Adv. Math.
  \textbf{195} (2005), no.~2, 405--455. \MR{MR2146350 (2006a:20092)}

\bibitem{BGP}
I.~N. Bern{\v{s}}te{\u\i}n, I.~M. Gel{\cprime}fand, and V.~A. Ponomarev,
  \emph{Coxeter functors, and {G}abriel's theorem}, Uspehi Mat. Nauk
  \textbf{28} (1973), no.~2(170), 19--33. \MR{MR0393065 (52 \#13876)}

\bibitem{BMRRT}
A.~B. Buan, R.~Marsh, M.~Reineke, I.~Reiten, and G.~Todorov, \emph{Tilting
  theory and cluster combinatorics}, Adv. Math. \textbf{204} (2006), no.~2,
  572--618. \MR{MR2249625 (2007f:16033)}

\bibitem{CalderoChapoton}
P.~Caldero and F.~Chapoton, \emph{Cluster algebras as {H}all algebras of quiver
  representations}, Comment. Math. Helv. \textbf{81} (2006), no.~3, 595--616.
  \MR{MR2250855 (2008b:16015)}

\bibitem{Caldero-Keller}
P.~Caldero and B.~Keller, \emph{From triangulated categories to cluster
  algebras. {II}}, Ann. Sci. \'Ecole Norm. Sup. (4) \textbf{39} (2006), no.~6,
  983--1009. \MR{MR2316979 (2008m:16031)}

\bibitem{caldero-reineke}
P.~Caldero and M.~Reineke, \emph{On the quiver {G}rassmannian in the acyclic
  case}, J. Pure Appl. Algebra \textbf{212} (2008), no.~11, 2369--2380.
  \MR{MR2440252}

\bibitem{CalderoZelevinsky}
P.~Caldero and A.~Zelevinsky, \emph{Laurent expansions in cluster algebras via
  quiver representations}, Mosc. Math. J. \textbf{6} (2006), no.~3, 411--429.
  \MR{MR2274858 (2008j:16045)}

\bibitem{CharHer}
V.~Chari and D.~Hernandez, \emph{Beyond {K}irillov-{R}eshetikhin modules},
  Quantum affine algebras, extended affine {L}ie algebras, and their
  applications, Contemp. Math., vol. 506, Amer. Math. Soc., Providence, RI,
  2010, pp.~49--81. \MR{2642561}

\bibitem{CG}
N.~Chriss and V.~Ginzburg, \emph{Representation theory and complex geometry},
  Birkh\"auser Boston Inc., Boston, MA, 1997. \MR{MR1433132 (98i:22021)}

\bibitem{ConwayCoxeter}
J.~H. Conway and H.~S.~M. Coxeter, \emph{Triangulated polygons and frieze
  patterns}, Math. Gaz. \textbf{57} (1973), no.~400, 87--94. \MR{MR0461269 (57
  \#1254)}

\bibitem{CB:normal}
W.~Crawley-Boevey, \emph{Normality of {M}arsden-{W}einstein reductions for
  representations of quivers}, Math. Ann. \textbf{325} (2003), no.~1, 55--79.
  \MR{MR1957264 (2004c:16017)}

\bibitem{DanKho}
V.~I. Danilov and A.~G. Khovanski{\u\i}, \emph{Newton polyhedra and an
  algorithm for calculating {H}odge-{D}eligne numbers}, Izv. Akad. Nauk SSSR
  Ser. Mat. \textbf{50} (1986), no.~5, 925--945. \MR{MR873655 (88i:32032)}

\bibitem{Deligne3}
P.~Deligne, \emph{Th\'eorie de {H}odge. {III}}, Inst. Hautes \'Etudes Sci.
  Publ. Math. (1974), no.~44, 5--77. \MR{MR0498552 (58 \#16653b)}

\bibitem{DWZ2}
H.~Derksen, J.~Weyman, and A.~Zelevinsky, \emph{Quivers with potentials and
  their representations {I}{I}: Applications to cluster algebras}, Journal of
  the American Mathematical Society \textbf{23} (2010), no.~3, 749--790.

\bibitem{DF-K}
P.~Di~Francesco and R.~Kedem, \emph{{$Q$}-systems, heaps, paths and cluster
  positivity}, Comm. Math. Phys. \textbf{293} (2010), no.~3, 727--802.
  \MR{2566162}

\bibitem{DXX}
M.~Ding, J.~Xiao, and F.~Xu, \emph{Integral bases of cluster algebras and
  representations of tame quivers}, arXiv.org:0901.1937.

\bibitem{Dupont}
G.~Dupont, \emph{Generic variables in acyclic cluster algebras},
  arXiv.org:0811.2909.

\bibitem{FomZel}
S.~Fomin and A.~Zelevinsky, \emph{Cluster algebras. {I}. {F}oundations}, J.
  Amer. Math. Soc. \textbf{15} (2002), no.~2, 497--529 (electronic).
  \MR{MR1887642 (2003f:16050)}

\bibitem{FomZel2}
\bysame, \emph{Cluster algebras. {II}. {F}inite type classification}, Invent.
  Math. \textbf{154} (2003), no.~1, 63--121. \MR{MR2004457 (2004m:17011)}

\bibitem{FomZel4}
\bysame, \emph{Cluster algebras. {IV}. {C}oefficients}, Compos. Math.
  \textbf{143} (2007), no.~1, 112--164. \MR{MR2295199 (2008d:16049)}

\bibitem{Fre-Res}
E.~Frenkel and N.~Reshetikhin, \emph{The {$q$}-characters of representations of
  quantum affine algebras and deformations of {$\scr W$}-algebras}, Recent
  developments in quantum affine algebras and related topics ({R}aleigh, {NC},
  1998), Contemp. Math., vol. 248, Amer. Math. Soc., Providence, RI, 1999,
  pp.~163--205. \MR{MR1745260 (2002f:17022)}

\bibitem{GLS}
C.~Gei{\ss}, B.~Leclerc, and J.~Schr{\"o}er, \emph{Rigid modules over
  preprojective algebras}, Invent. Math. \textbf{165} (2006), no.~3, 589--632.
  \MR{MR2242628 (2007g:16023)}

\bibitem{GLS2}
\bysame, \emph{Cluster algebra structures and semicanoncial bases for unipotent
  groups}, arXiv.org:math/0703039.

\bibitem{GM}
S.~I. Gelfand and Y.~I. Manin, \emph{Methods of homological algebra}, second
  ed., Springer Monographs in Mathematics, Springer-Verlag, Berlin, 2003.
  \MR{MR1950475 (2003m:18001)}

\bibitem{HappelUnger}
D.~Happel and L.~Unger, \emph{Almost complete tilting modules}, Proc. Amer.
  Math. Soc. \textbf{107} (1989), no.~3, 603--610. \MR{MR984791 (90f:16026)}

\bibitem{Her}
D.~Hernandez, \emph{The {K}irillov-{R}eshetikhin conjecture and solutions of
  {$T$}-systems}, J. Reine Angew. Math. \textbf{596} (2006), 63--87.
  \MR{MR2254805 (2007j:17020)}

\bibitem{Her-small}
\bysame, \emph{Smallness problem for quantum affine algebras and quiver
  varieties}, Ann. Sci. \'Ec. Norm. Sup\'er. (4) \textbf{41} (2008), no.~2,
  271--306. \MR{MR2468483}

\bibitem{HerLec}
D.~Hernandez and B.~Leclerc, \emph{Cluster algebras and quantum affine
  algebras}, Duke Math. J. \textbf{154} (2010), no.~2, 265--341.

\bibitem{Hubery}
A.~Hubery, \emph{Acyclic cluster algebras via {R}ingel-{H}all algebras},
  preprint.

\bibitem{IIKNS}
R.~Inoue, O.~Iyama, A.~Kuniba, T.~Nakanishi, and J.~Suzuki, \emph{Periodicities
  of {T}-systems and {Y}-systems}, arXiv.org:0812.06.

\bibitem{Kac-quiver}
V.~G. Kac, \emph{Infinite root systems, representations of graphs and invariant
  theory}, Invent. Math. \textbf{56} (1980), no.~1, 57--92. \MR{MR557581
  (82j:16050)}

\bibitem{Kac-quiver2}
\bysame, \emph{Infinite root systems, representations of graphs and invariant
  theory. {II}}, J. Algebra \textbf{78} (1982), no.~1, 141--162. \MR{MR677715
  (85b:17003)}

\bibitem{KaSha}
M.~Kashiwara and P.~Schapira, \emph{Sheaves on manifolds}, Grundlehren der
  Mathematischen Wissenschaften [Fundamental Principles of Mathematical
  Sciences], vol. 292, Springer-Verlag, Berlin, 1990, With a chapter in French
  by Christian Houzel. \MR{MR1074006 (92a:58132)}

\bibitem{Keller}
B.~Keller, \emph{Cluster algebras, quiver representations and triangulated
  categories}, arXiv.org:0807.1960.

\bibitem{Knight}
H.~Knight, \emph{Spectra of tensor products of finite-dimensional
  representations of {Y}angians}, J. Algebra \textbf{174} (1995), no.~1,
  187--196. \MR{MR1332866 (96g:17015)}

\bibitem{Lau}
G.~Laumon, \emph{Transformation de {F}ourier, constantes d'\'equations
  fonctionnelles et conjecture de {W}eil}, Inst. Hautes \'Etudes Sci. Publ.
  Math. (1987), no.~65, 131--210. \MR{MR908218 (88g:14019)}

\bibitem{Lec2}
B.~Leclerc, \emph{Alg\`ebres affine quantiques et alg\`ebres amass\'ees}, talk
  at IHP.

\bibitem{Leclerc}
\bysame, \emph{Canonical and semicanonical bases}, talk at Reims.

\bibitem{Lu-affine}
G.~Lusztig, \emph{Affine quivers and canonical bases}, Inst. Hautes \'Etudes
  Sci. Publ. Math. (1992), no.~76, 111--163. \MR{MR1215594 (94h:16021)}

\bibitem{Lu-book}
\bysame, \emph{Introduction to quantum groups}, Progress in Mathematics, vol.
  110, Birkh\"auser Boston Inc., Boston, MA, 1993. \MR{MR1227098 (94m:17016)}

\bibitem{Lusztig-On-Quiver}
\bysame, \emph{On quiver varieties}, Adv. Math. \textbf{136} (1998), no.~1,
  141--182. \MR{MR1623674 (2000c:16016)}

\bibitem{Lu:semi}
\bysame, \emph{Semicanonical bases arising from enveloping algebras}, Adv.
  Math. \textbf{151} (2000), no.~2, 129--139. \MR{MR1758244 (2001e:17033)}

\bibitem{MRZ}
R.~Marsh, M.~Reineke, and A.~Zelevinsky, \emph{Generalized associahedra via
  quiver representations}, Trans. Amer. Math. Soc. \textbf{355} (2003), no.~10,
  4171--4186 (electronic). \MR{MR1990581 (2004g:52014)}

\bibitem{Na-quiver}
H.~Nakajima, \emph{Instantons on {ALE} spaces, quiver varieties, and
  {K}ac-{M}oody algebras}, Duke Math. J. \textbf{76} (1994), no.~2, 365--416.
  \MR{MR1302318 (95i:53051)}

\bibitem{Na-alg}
\bysame, \emph{Quiver varieties and {K}ac-{M}oody algebras}, Duke Math. J.
  \textbf{91} (1998), no.~3, 515--560. \MR{MR1604167 (99b:17033)}

\bibitem{Na-qaff}
\bysame, \emph{Quiver varieties and finite-dimensional representations of
  quantum affine algebras}, J. Amer. Math. Soc. \textbf{14} (2001), no.~1,
  145--238 (electronic). \MR{MR1808477 (2002i:17023)}

\bibitem{Na-tensor}
\bysame, \emph{Quiver varieties and tensor products}, Invent. Math.
  \textbf{146} (2001), no.~2, 399--449. \MR{MR1865400 (2003e:17023)}

\bibitem{MR1872257}
\bysame, \emph{{$T$}-analogue of the {$q$}-characters of finite dimensional
  representations of quantum affine algebras}, Physics and combinatorics, 2000
  (Nagoya), World Sci. Publ., River Edge, NJ, 2001, pp.~196--219. \MR{MR1872257
  (2003b:17020)}

\bibitem{MR1989196}
\bysame, \emph{Geometric construction of representations of affine algebras},
  Proceedings of the International Congress of Mathematicians, Vol. I (Beijing,
  2002) (Beijing), Higher Ed. Press, 2002, pp.~423--438. \MR{MR1989196
  (2004e:17012)}

\bibitem{MR1993360}
\bysame, \emph{{$t$}-analogs of {$q$}-characters of {K}irillov-{R}eshetikhin
  modules of quantum affine algebras}, Represent. Theory \textbf{7} (2003),
  259--274 (electronic). \MR{MR1993360 (2004e:17013)}

\bibitem{MR2144973}
\bysame, \emph{Quiver varieties and {$t$}-analogs of {$q$}-characters of
  quantum affine algebras}, Ann. of Math. (2) \textbf{160} (2004), no.~3,
  1057--1097. \MR{MR2144973 (2006k:17029)}

\bibitem{Na:E8}
\bysame, \emph{$t$--analogs of $q$--characters of quantum affine algebras of
  type ${E}_6$, ${E}_7$, ${E}_8$}, arXiv.org:math/0606637.

\bibitem{FQ}
F.~Qin, \emph{Quantum cluster variables via {S}erre polynomials},
  arXiv.org:1004.4171.

\bibitem{Schofield}
A.~Schofield, \emph{General representations of quivers}, Proc. London Math.
  Soc. (3) \textbf{65} (1992), no.~1, 46--64. \MR{MR1162487 (93d:16014)}

\bibitem{ShermanZelevinsky}
P.~Sherman and A.~Zelevinsky, \emph{Positivity and canonical bases in rank 2
  cluster algebras of finite and affine types}, Mosc. Math. J. \textbf{4}
  (2004), no.~4, 947--974, 982. \MR{MR2124174 (2006c:16052)}

\bibitem{VV2}
M.~Varagnolo and E.~Vasserot, \emph{Perverse sheaves and quantum {G}rothendieck
  rings}, Studies in memory of {I}ssai {S}chur ({C}hevaleret/{R}ehovot, 2000),
  Progr. Math., vol. 210, Birkh\"auser Boston, Boston, MA, 2003, pp.~345--365.
  \MR{MR1985732 (2004d:17023)}

\end{thebibliography}

\end{document}